\documentclass[11pt]{amsart}
\usepackage{amssymb}
\usepackage{amsfonts}
\usepackage{amscd}
\usepackage{amsmath}
\usepackage{mathrsfs}
\usepackage{curves}
\usepackage{epsfig}
\usepackage[pass]{geometry}
\usepackage{graphicx}
\usepackage{verbatim}
\usepackage{latexsym}
\usepackage{eucal}
\usepackage{hyperref}
\usepackage{color}
\usepackage{dsfont}
\usepackage[normalem]{ulem}

\numberwithin{equation}{section}

\newtheorem{theorem}{Theorem}[section]
\newtheorem*{theorem*}{Theorem}
\newtheorem{lemma}[theorem]{Lemma}
\newtheorem{proposition}[theorem]{Proposition}

\newtheorem{definition}[theorem]{Definition}

\theoremstyle{remark}
\newtheorem{remark}[theorem]{Remark}

\def \supp {\text{\rm supp\hspace{.1em}}}
\def \Lip {\text{\rm Lip\hspace{.1em}}}

\def \beq {\begin{equation}}
\def \eeq {\end{equation}}

\def \bpf {\begin{proof}}
\def \epf {\end{proof}}
\def \beqq {\begin{equation*}}
\def \eeqq {\end{equation*}}

\def \tilde {\widetilde}

\def \RR {{\mathbb R}}
\def \SS {{\mathbb S}}
\def \ii {\text{\rm i}}
\def \bfU {\text{\bf U}}
\def \bfu {\text{\bf u}}

\def \lvr {\langle v\rangle}
\def \lvpr {\langle v'\rangle}
\def \leVr {\langle \epsilon V\rangle}
\def \leVpr {\langle \epsilon V'\rangle}
\def \fu {\mathfrak{u}}
\def \cI {\mathcal{I}}
\def \tu {\tilde{u}}

\usepackage{xcolor}

\begin{document}
\title[Fractional Fokker--Planck \& Fermi pencil--beams
]{Pencil-beam approximation of fractional Fokker-Planck}
\author{Guillaume Bal$^\dagger$}
\thanks{$^\dagger$Departments of Statistics and Mathematics, University of Chicago, Chicago, IL.\\ {\em Email address:} \texttt{guillaumebal@uchicago.edu}}
\author{Benjamin Palacios$^\ddagger$}
\thanks{$^\ddagger$Department of Statistics, University of Chicago, Chicago, IL.
\\ {\em Email address:} \texttt{bpalacios@uchicago.edu}}

\maketitle

\begin{abstract}
 We consider the modeling of light beams propagating in highly forward-peaked turbulent media by fractional Fokker-Planck equations and their approximations by fractional Fermi pencil beam models. We obtain an error estimate in a 1-Wasserstein distance for the latter model showing that beam spreading is well captured by the Fermi pencil-beam approximation in the small diffusion limit.
\end{abstract}


\section{Introduction}\label{sec:intro}

The stationary Radiative Transfer Equation (sRTE) models the mesoscopic equilibrium state of particles, such as photons or electrons, propagating across heterogenous media. The density of particles at the equilibrium state results from the interaction of them with the microscopic constituents of the underlying background, by virtue of absorption and scattering. More specifically, denoting by $u(x,\theta)$ the phase-space density of particles, this is, those located at position $x\in\RR^d$ and moving in direction $\theta\in\SS^{d-1}$, the sRTE is given by
\beq\label{sRTE}
\theta\cdot\nabla_x u +\lambda u = \cI(u)+f,
\eeq
where $f(x,\theta)$ represents a source or sink; $\lambda(x)$ measures absorption caused by the medium; and $\cI(u)$ controls the rate of collisions between particles and the background. It is commonly defined as an integral operator of the form
\beqq
\cI(u) = \int_{\SS^{d-1}}k(x,\theta,\theta')\left(u(x,\theta') - u(x,\theta)\right)d\theta',
\eeqq
where the scattering kernel $k(x,\theta',\theta)$ represents the probability of a particle at a location $x$, changing direction from $\theta'$ to $\theta$ due to collisions. It is commonly assumed to be symmetric in the angular variable, this is, $k(x,\theta',\theta)=k(x,\theta,\theta')$.

In specific circumstances, a single or even a few scattering interactions are not enough to considerably change the direction of propagation of traveling particles. However, when collisions occur with high frequency, for example in the propagation of light in biological tissues or in thick atmospheres, their accumulative effect becomes significant and therefore visible at a larger scale. This is the regime of {\em highly forward-peaked scattering}. It intents to model a diffusive transport at a scale where the initial directionality of particles is not completely lost, but frequent enough to produce some mixing effect. It is suitable for the study of electron and photon transport \cite{BL2,LL,OF}, optical imaging and microscopy \cite{CCOCPH,KiKe}, and laser propagation on turbulent media \cite{HaBDH,RoRe}. More details on its derivation and application to inverse problem can be found in \cite{B,BKR}, and we refer the reader to \cite{DL} for a mathematical introduction to the theory of the transport equation.

A standard choice for scattering kernel in this regime is the well-known {\em Henyey--Greenstein} phase function \cite{HG}, defined in 3-dimensional space as
\beqq
k_{HG}(\theta',\theta) := \frac{1}{\Delta}\cdot \frac{1-g^2}{\left(1+g^2-2g\theta'\cdot\theta\right)^{3/2}}.
\eeqq
The parameter $\Delta>0$ is the {\em mean free path}, the mean distance between successive interactions of particles with the underlying medium; while $g\in(-1,1)$ is the {\em anisotropy factor}, representing forward-peaked scattering and determining the {\em mean scattering cosine} associated to $k_{HG}(\theta,\theta')$. The closer $g$ is to 1, the more forward-peaked the scattering is. Physically, the most meaningful distance is the {\em transport mean free path} given by $\Delta^*=\Delta/(1-g)$. It measures the average distance over which particles change directions `significantly' after undergoing a large number of forward-peaked collisions when $g$ is close to $1$. 

Its been pointed out in a number of publications that Henyey--Greenstein scattering has small but significant large-angle effect, and this behavior is not accurately captured by standard approximation models such as the (local) Fokker--Planck equation and its respective Fermi--Eyges pencil-beam equation, which are better for numerical implementations and practical purposes. See for instance \cite{BL2, KiKe}. Moreover, the derivation of these simpler approximation models are known to fail for this specific kernel \cite{Po}. Phenomena of this type are not specific to the Henyey--Greenstein kernel and occur also for the Rutherford scattering kernel in electron transport \cite{BL2,LL}. 

There has been previous attempts to find better suited approximation model such as the Boltzmann--Fokker--Planck equation and the Generalized Fokker--Planck equation (or Leakeas--Larsen equation) \cite{LL}.  
In the highly forward-peaked limit when $g\to 1$, it was shown in \cite{AS} (for the time-evolution case) that (a rescaling of) global solutions to RTE with Henyey-Greenstein scattering cross-section, weakly converge in $L^2$ to functions satisfying a fractional version of Fokker--Planck, where the limiting scattering term involves a singular integral operator which resembles a fractional Laplace-Beltrami operator on the unit sphere. The hypoelliptic property of the associated integro-differential operator was also analyzed. 
Similar results have been obtained in \cite{GPR1,GPR2} for the radiative transfer equation with long-range interactions. Analyzing the limiting scaling of this equation, the authors demonstrated the emergence of a Fokker--Planck type operator in the highly forward-peaked limit, with a diffusion component consisting of another singular integral operator whose high frequency behavior equals that of the Laplace-Beltrami operator on the sphere. The first publication provides precise hypoelliptic estimates for such operators.

The main goal of this paper is the study of the propagation of narrow beams in the highly forward-peaked regime. This corresponds to the analysis of the fractional Fokker--Planck equations obtained in \cite{AS}, for particle sources that are highly concentrated in phase-space and with a choice of scales so that $\Delta^*\gg 1$. The latter condition ensures that the beams remain narrow over a domain of interest. In this setting, the magnitude of the diffusion coefficient in the fractional Laplacian is small compared to the spatial dynamic of particles and characterized by a small parameter $\epsilon$. Analogously to what the authors have done in the local case ($s=1$) in \cite{BP}, we provide a higher order approximation to the Fokker--Planck solution by means of pencil-beams, which are solutions to appropriate fractional Fermi pencil-beam equations. We establish error estimates to contrast the accuracy of the approximations for the pencil-beam model and the ballistic (i.e., unscattered) transport solution. We perform this analysis in an adapted 1-Wasserstein sense with a level of accuracy given in terms of powers of the diffusion parameter $\epsilon$. Similarly as in \cite{AS}, we consider a generalized version of the Henyey--Greenstein forward-peaked scattering.

\subsection{Highly forward-peaked radiative transfer for narrow beams} 
In space dimension $d\geq 2$, 
we say  $k(x,\theta',\theta)$ is a {\em Henyey--Greenstein scattering kernel with parameters $(g,s,m)$} whenever
\beq\label{HG_kernels}
k_g(x,\theta',\theta) := \frac{b(x)}{\Delta}\cdot\frac{(1-g)^m(1+g)^m}{\big((1-g)^2+2g(1-\theta'\cdot\theta)\big)^{\frac{d-1}{2}+s}},\quad(x,\theta',\theta)\in\RR^d\times\SS^{d-1}\times\SS^{d-1},
\eeq
$b$ a Lipschitz function, and with $m>0$ so that 
\beq\label{def_m}
\int_{\SS^{d-1}}k_g(x,\theta',\theta)d\theta' = O(\Delta^{-1}).
\eeq
We recover the standard 3-dimensional kernel when the parameters are $(g,1/2,1)$. For a general fractional exponent $s\in(0,1)$, one easily verifies that in dimension $d=3$, $(g,s,2s)$ satisfies the required scaling for the integral of $b_g$ in \eqref{def_m}. 
The collision operator associated to such kernels is given by
\beqq
\cI_g(u) := \int_{\SS^{d-1}}k_g(x,\theta,\theta')\left(u(x,\theta') - u(x,\theta)\right)d\theta'.
\eeqq

We work in the {\em narrow beam regime for particle transport} defined by the following hypotheses:
\begin{itemize}
\item[i.] The scattering kernel is of Henyey--Greenstein type with parameters $(g,s,m)$.
\item[ii.] The transport mean free path is given by $\Delta^*:=\Delta/(1-g)^m$, where $(1-g)^{m} \ll \Delta$. More precisely, we take $(1-g)^{m} = \epsilon^{2s}\Delta$, for $\epsilon\ll1$.
\item[iii.] The source term $f(x,\theta)$ is highly concentrated near a single point $(x
_0,\theta_0)\in\RR^{d}\times\SS^{d-1}$. 
\end{itemize}
The exponent $2s$ in condition $ii$ is included here for notational convenience, and agrees with the exponent considered in \cite{BP} as we approach the local case $s=1$. Without loss of generality, we assume in most of the article that $(x_0,\theta_0) = (0,N)$ with $N=(0,\dots,0,1)$. We also consider the following modified version of $iii$:
\begin{itemize}
\item[iii'.]  $f\in L^\infty_{x,\theta}\cap L^1_{x,\theta}$ is a compactly supported $\delta$-approximation to the identity $\delta_{x}\delta_{N}(\theta)$, i.e., such that $\left|\int f\varphi dxd\theta - \varphi(0,N)\right|\lesssim \delta$ for all $\varphi\in C(\RR^d\times\SS^{d-1})$.
 \end{itemize}

\subsection{Fractional Fokker--Planck equation (fFPE)} 
As $g\to 1$, a Henyey-Greenstein kernel formally converges to the singular one
\beq\label{limiting_kernel}
k(x,\theta,\theta') := \frac{2^{m-\frac{d-1}{2}-s}\epsilon^{2s}b(x)}{\big(1-\theta\cdot\theta'\big)^{\frac{d-1}{2}+s}},
\eeq
whose link to the fractional Laplacian was noticed in \cite{AS} and established by means of a stereographic transformation $\mathcal{S}:\SS^{d-1}\backslash\{S\}\to \RR^{d-1}$, with $S=-N$. The unit sphere is considered to be an embedded hyper-surface in $\RR^{d}$, where spatial coordinates have been fixed so that $N = (0,\dots,0,1)=-S$. The stereographic coordinates (with respect to $N$) and its associated surface measure are then defined by
$$
v = \mathcal{S}(\theta):= \frac{1}{(1+\theta_d)}(\theta_1,\dots,\theta_{d-1})\quad\text{and}\quad d\theta=\frac{2^{d-1}}{\langle v\rangle^{2(d-1)}}dv,\quad\text{with}\quad \langle v\rangle = (1+|v|^2)^{1/2};
$$
while its inverse transformation is defined as
$$
\theta = \mathcal{S}^{-1}(v) := \left(\frac{2v}{\langle v\rangle^2}, \frac{1-|v|^2}{\langle v\rangle^2}\right).
$$
The identity
\beqq\label{cos_identity}
1-\theta\cdot\theta' = 2\frac{|v-v'|^2}{\lvr^2\lvpr^2},
\eeqq
allows us to write the limiting singular integral  $\int_{\RR^{d-1}}k(x,\theta',\theta)\big(u(\theta')-u(\theta)\big)d\theta'$ 
in stereographic coordinates as
\beqq\label{scat_op_fFP}
 \int_{\RR^{d-1}}\lvr^{2(d-1)} k_{\mathcal{S}}(x,v',v)\big(u(v')-u(v)\big)dv',
\eeqq
where $[u]_\mathcal{S} = u(v)$ stands for the pullback $u(\mathcal{S}^{-1}(v))$, and with kernel such that 
\beqq\label{scat_kernel_fFP}
k_{\mathcal{S}}(x,v',v)dv'dv = \left[k(x,\theta',\theta)d\theta' d\theta\right]_{\mathcal{S}}= \frac{2^{m-2s}\epsilon^{2s} b(x)dv'dv}{|v-v'|^{d-1+2s}\lvpr^{d-1-2s}\lvr^{d-1-2s}}.
\eeqq
By introducing the following version of the fractional Laplacian on the unit sphere (as in \cite{AS}), 
\beq\label{def:frac_Lap_sphere}
\left[(-\Delta_\theta)^su\right]_{\mathcal{S}} := \langle \cdot\rangle^{d-1+2s}\big(-\Delta_v\big)^sw_{\mathcal{S}},\quad w_{\mathcal{S}} = \frac{[u]_{\mathcal{S}}}{ \langle \cdot\rangle^{d-1-2s}},
\eeq
with the Euclidean fractional Laplacian given by the singular integral
\beqq
(-\Delta_v)^sg(v) := c_{d-1,s}\int_{\RR^{d-1}} \frac{f(v) - f(v+z)}{|z|^{d-1+2s}}dz,
\eeqq
we see that in the (so far formal) limit $g\to 1$, we encounter the {\em fractional Fokker--Planck equation}
\beq\label{fFP_diff}
\theta\cdot\nabla_x u +\lambda(x) u +\epsilon^{2s}\sigma(x)\Big((-\Delta_\theta)^su - cu\Big)= f,\quad \text{in }\RR^d\times \SS^{d-1},
\eeq
where $\sigma(x):= c^{-1}_{d-1,s}2^{m-2s}b(x)$ and $c=c_{s,d}>0$, for suitable positive constants $c_{d-1,s},c_{d,s}$ (see appendices A.2 and A.3 in \cite{AS} for more details). A rigorous convergence result is presented in Section \ref{sec:sRTE_fFP} (see also \cite{AS}). 

Using this notation, the limiting kernel \eqref{limiting_kernel} takes the form $k(x,\theta,\theta') = \epsilon^{2s}\sigma(x)K(\theta',\theta)$, where we define its angular part and the associated integral operator as
\beq\label{kernel_K}
K(\theta',\theta) := \frac{c_{d-1,s}2^{s-\frac{d-1}{2}}}{\left(1-\theta\cdot\theta'\right)^{\frac{d-1}{2}+s}},\quad \cI_\theta(u):= \int K(\theta',\theta)(u(\theta')-u(\theta))d\theta'.
\eeq
It will be occasionally more convenient to work with the integro-differential version
\beq\label{fFP_id}
\theta\cdot\nabla_x u +\lambda(x) u = \epsilon^{2s}\sigma(x)\cI_\theta(u)+f(x,\theta),
\eeq
where, in stereographic coordinates, the integral part takes the form
\beq\label{fFP_kernel}
\begin{aligned}
&[\cI_\theta(u)]_\mathcal{S}:= c_{d-1,s} \int_{\RR^{d-1}}\big(u(v')-u(v)\big)K_\mathcal{S}(v',v)\lvr^{2(d-1)}dv',\\
&\text{with}\quad  K_\mathcal{S}(v',v):= \frac{\lvpr^{-(d-1-2s)}\lvr^{-(d-1-2s)}}{|v-v'|^{d-1+2s}}.
\end{aligned}
\eeq
Throughout the paper, we will assume the following conditions on the coefficients of equation \eqref{fFP_id}:

\beq\label{cond_coeff}
\begin{aligned}
&\text{\em $\sigma$ and $\lambda$ are Lipschitz continuous, and there exist constants}\\
&\text{\em $\sigma_0,\lambda_0>0$, such that $\sigma_0\leq \sigma(x)\leq \sigma_0^{-1}$ and $\lambda_0\leq \lambda(x)\leq \lambda_0^{-1}$.} 
\end{aligned}
\eeq

\subsection{Fermi pencil-beam approximation}
Similarly as in the original derivation of the Fermi pencil-beam equation (see, for instance, \cite{BL1}), we can derive an approximation to the fractional Fokker--Planck equation \eqref{fFP_diff} by following simple formal computations. Let $u$ be the solution to \eqref{fFP_diff} (or equivalently \eqref{fFP_id}-\eqref{fFP_kernel}) upon which we perform a change of coordinates that we proceed to define. We introduce the {\em stretched coordinate system}, or {\em pencil-beam coordinates}, as
\beqq
X := ((2\epsilon)^{-1}x',x^d)\quad\text{and}\quad V = \epsilon^{-1}\mathcal{S}(\theta),
\eeqq
defined on $\RR^{d}\times\RR^{d-1}$. In addition, the volume form on $\RR^{d}\times\SS^{d-1}$ adopts the representation
\beqq
dxd\theta =\frac{(2\epsilon)^{2(d-1)}}{\leVr^{2(d-1)}}dXdV.
\eeqq
We see that in pencil-beam coordinates, the function $U(X,V) = (2\epsilon)^{2(d-1)}u(2\epsilon X', X^d,\mathcal{J}(\epsilon V))$ satisfies the equation
\beqq
\begin{aligned}
&\left(\frac{V}{\langle\epsilon V\rangle^2}, \frac{1-\epsilon^2|V|^2}{\langle \epsilon V\rangle^2}\right)\cdot \nabla_{X}U + \lambda(2\epsilon X',X^d) U\\
&= \epsilon^{2s}\sigma(2\epsilon X',X^d) \cdot c_{d-1,s} \int_{\RR^{d-1}} \left(U(X,V') - U(X,V)\right)K_\mathcal{S}(\epsilon V,\epsilon V')\leVr^{2(d-1)}dV' + F,
\end{aligned}
\eeqq
for  $F(X,V)=(2\epsilon)^{2(d-1)}f(2\epsilon X',X^d,\mathcal{J}(\epsilon V))$. By ignoring the dependence of $U$ and $F$ with respect to $\epsilon$, we can take the point-wise limit as $\epsilon\to 0$ in order to (formally) deduce the {\em fractional Fermi pencil-beam} equation
\beq\label{fFpb}
\begin{array}{ll}
\mathcal{P}(U):=\partial_{X^d}U + V\cdot\nabla_{X'} U +\tilde{\lambda} U + \tilde{\sigma}(-\Delta_V)^sU= F,& \text{in }\RR^d_+\times\RR^{d-1},
\end{array}
\eeq
where $\tilde{\sigma}(X^d) = \sigma(0,X^d)$ and $\tilde{\lambda}(X^d) = \lambda(0,X^d)$. $\RR^d_+$ stands for the half-space $X=(X',X^d)\in\RR^d$ with $X^d>0$.

We also define the {\em backward (or adjoint) problem} as
\beq\label{b_fFpb}
\begin{array}{ll}
\mathcal{P}'(W):=-\partial_{X^d}W - V\cdot\nabla_{X'} W +\tilde{\lambda} U + \tilde{\sigma}(-\Delta_V)^sW= F,& \text{in }\RR^d_+\times\RR^{d-1}.
\end{array}
\eeq
For more details on the derivation of the Fermi pencil-beam equation we refer the reader to \cite{BL1}.

\subsection{1-Wasserstein distance}
Narrow beam solutions are by construction singular. Approximation errors need to be measured using a metric that adequately captures beam spreading, which is intuitively a form of a `mass' transport. A natural notion to measure the mass transport of probability measures is given by the 1-Wasserstein (or earth-mover) distance. In our setting, while solutions preserve positivity, their total mass may vary. As we have done in, e.g., \cite{BJ, BP} for similar reasons, we consider a generalization of the standard 1-Wasserstein distance.


We denote by $BL_{1,\kappa}$ the set of $\psi\in C(\RR^d\times\SS^{d-1})$ such  that $\|\psi\|_\infty\leq 1$ and $\Lip(\psi)\leq\kappa$, for a fixed parameter $\kappa\geq 1$. 
Given two finite Borel measures in $\RR^d\times\SS^{d-1}$, $\mu$ and $\nu$, their 1-Wasserstein distance is defined as
\beqq
\mathcal{W}^1_\kappa(\mu,\nu) :=\sup\left\{\int \psi (\mu-\nu) : \psi\in BL_{1,\kappa}\right\}.
\eeqq
For any $\Omega\subset\RR^d$ open and bounded, we can similarly define a 1-Wasserstein distance in the compact region $\bar{\Omega}\times\SS^{d-1}$ by restricting the test functions to this set, which we denote by $\mathcal{W}^1_{\bar{\Omega},\kappa}(\mu,\nu)$. 

The parameter $\kappa$ defines the spatial scale $\kappa^{-1}$ over which we wish to penalize transport. The uniform bound on $\psi$ allows us to penalize variations in total mass.

\subsection{Main results} \label{subsec:main_results}
Our main result asserts that a suitable transformation of a Fermi pencil-beam solution approximates the solution to the fractional Fokker--Planck equation, at a higher accuracy than the ballistic solution ---the one that completely neglects diffusion (i.e., a solution to \eqref{fFP_diff} for $\epsilon=0$).
\begin{theorem}\label{thm:narrow_beam}
Assume that conditions {\em i}, {\em ii} and {\em iii'} hold with the latter satisfied for $\delta\lesssim \epsilon^{2s}\kappa$. Let $u$ be the solution to the fractional Fokker--Planck equation \eqref{fFP_diff}, $v$ the ballistic solution, and $U$ the solution to the corresponding fractional Fermi pencil-beam equation \eqref{fFpb}. The pencil-beam approximation in the original variables is defined by
\beqq
\mathfrak{u}(x,\theta) := (2\epsilon)^{-2(d-1)}U((2\epsilon)^{-1}x',x^d,\epsilon^{-1}\mathcal{S}(\theta)).
\eeqq
Then, for any $s'\in(0,s)$ in dimension $d\geq 3$, or any $s'\in(2s-1,s)$ in dimension $d=2$, there exist positive constants $M_1=M_1(d,s,s')$, $M_2=M_2(d,s,s')$ and $m=m(d,s)$ (depending also on $\lambda$ and $\sigma$) such that
$$
\mathcal{W}^1_{\kappa}(u,\mathfrak{u})\leq M_1\cdot\epsilon^{2s'}\kappa^{s'},\quad\text{and}\quad m\cdot\min\{\epsilon\kappa,1\}\leq \mathcal{W}^1_{\kappa}(v,\mathfrak{u})\leq M_2\cdot(\epsilon\kappa)^{\min\{2s',1\}},
$$
where $M_1\to \infty$ as $s'\to s$, while $M_2\to\infty$ when $s'\to s$ for $s\leq1/2$, otherwise, $M_2$ is independent of $s'$.
\end{theorem}

The proof of this result is split into Theorems \ref{thm:fFP_vs_pb} and \ref{thm:ball_vs_pb}, which we prove in Sections \ref{subsec:fFpb_approx} and \ref{subsec:ballistic_approx}, respectively. 

By linearity, we may generalize the above pencil-beam approximation to sources that do not satisfy condition {\em iii} (and hence {\em iii'}). We construct a general approximation by (continuously) superposing pencil-beams. The details of this definition are postponed to Section \ref{sec:superposition}.
\begin{theorem}[Approximation via continuous superposition]\label{thm:cont_sup}
Under the hypotheses of the previous theorem, except for hypothesis \text{\em iii}, 
the same conclusion holds for a broad source $f$ satisfying
\beqq
f\in L^1(\RR^d\times\SS^{d-1}),\quad f\geq 0\quad\text{and}\quad\text{supp}(f)\Subset \RR^d\times\SS^{d-1},
\eeqq
and $\mathfrak{u}(x,\theta) := \int_{\RR^d\times\SS^{d-1}}f(y,\eta)\mathfrak{u}(x,\theta;y,\eta)dyd\eta$, a continuous superposition of pencil-beams.
\end{theorem}

Our last result concerns the approximation to fractional Fokker--Planck solutions by means of a finite number of pencil-beams.
\begin{theorem}[Approximation via discrete superposition]\label{thm:disc_sup}
Under the hypotheses of the previous theorem, we reach the same conclusions for some $\mathfrak{u}(x,\theta) := \sum_{i=1}^Ia_i\cdot\mathfrak{u}(x,\theta;x_i,\theta_i)$, a discrete superposition of pencil-beams.
\end{theorem}

Throughout the paper, we adopt the notation $a \lesssim b$ to denote an estimate of the form $a\leq Cb$, for a constant that is independent of $\epsilon$ and $\kappa$. We will occasionally use the notation: $Q=\RR^d\times\SS^{d-1}$, $Q_+=\RR^d_+\times\SS^{d-1}$, $\mathcal{Q}=\RR^d\times\RR^{d-1}$, and $\mathcal{Q}_+=\RR^d_+\times\RR^{d-1}$.

The paper is organized as follows. We begin with the study of the sRTE and its Fokker--Planck approximation. Well-posedness results are proven in Section \ref{sec:sRTE_fFP}, as well a convergence result in the 1-Wasserstein framework, along the lines of \cite{AS} for the $L^2$-case. In Section \ref{sec:fFpb}, we show existence and uniqueness of solution to the fractional Fermi pencil-beam equation, and prove some integrability properties of their solutions which will be of great importance in subsequent sections. We then move to the approximation analysis where we define in details the pencil-beam approximation and provide the proof of our main Theorem \ref{thm:narrow_beam}. This is the content of Section \ref{sec:Fpb_approx}. We conclude the article with constructions, based on the narrow beam case, for the superposition of pencil-beams and the proof of Theorems \ref{thm:cont_sup} and \ref{thm:disc_sup} in Section \ref{sec:superposition}.

\section{Stationary Radiative Transfer and fractional Fokker--Planck equations}\label{sec:sRTE_fFP}
\subsection{Well-posedness of sRTE}\label{subsec:sRTE}
The existence and uniqueness of global solutions in $L^2$ is (easily) deduced from the well-posedness in bounded domains \cite{DL}. It follows by an approximation argument after confining equation \eqref{sRTE} to a bounded region $\Omega\times\SS^{n-1}$ with smooth boundary, and solving the null Cauchy data problem there. Defining
$$W^2_{\Omega}:=\{u\in L^2(\Omega\times\SS^{d-1}\} \;:\; \theta\cdot\nabla_x u\in L^2(\Omega\times\SS^{d-1}))\},$$
and similarly $W^2$ when $\Omega=\RR^d$ in the previous definition, 
the standard argument states that the solutions $u_{\Omega}\in W^2_\Omega$ converge to the desired global solution $u$ of \eqref{sRTE} as $\Omega\to \RR^d$, and moreover $u\in W^2$. A proof of this can be found in Appendix \ref{appdx:1}.
\begin{theorem}\label{thm:wellposedness_sRTE}
For $f\in L^2(Q)$, there exists a unique solution $u \in W^2$ to the stationary RTE \eqref{sRTE} for an integrable symmetric kernel. Furthermore, if $f\geq 0$ then $u\geq 0$.
\end{theorem}

We need the following two properties of solutions to the sRTE with Henyey--Greenstein kernels. The fist one follows directly from the mass-conservation property: $\mathcal{I}_g[\cdot]$:
\beq\label{mass_conservation}
\int_{\SS^{d-1}} \mathcal{I}_g[u^g](x,\theta)d\theta = \int_{\SS^{d-1}}  \int_{\SS^{d-1}}\big(u^g(\theta')-u^g(\theta)\big)k_g(x,\theta,\theta')d\theta'd\theta=0.
\eeq

\begin{lemma}\label{lemma:L1_est}
For a nonnegative source $f\in L^1(Q)$, the solution to \eqref{sRTE}-\eqref{HG_kernels} satisfies the estimate $\|u^g\|_{L^1}\leq \lambda_0^{-1}\|f\|_{L^1}$ for all $g\in(0,1)$.
\end{lemma}

\begin{proof}
This follows by integrating equation \eqref{sRTE} and the fact that for $f\geq 0$, $u^g\geq 0$. Indeed, by integration by parts and \eqref{mass_conservation} we get
$$
\underbrace{\int_Q \theta\cdot\nabla_x u^g dxd\theta}_{=0} +\int_Q \lambda u^gdxd\theta = \underbrace{\int_Q\cI_g[u^g]dxd\theta}_{=0}+\int fdxd\theta.
$$
\end{proof}

\begin{lemma}\label{lemma:Linfty_est}
For 
$f\in L^\infty(Q)\cap L^2(Q)$ 
the solution to \eqref{sRTE}-\eqref{HG_kernels} satisfies the estimate $\|u^g\|_\infty\leq \lambda_0^{-1}\|f\|_\infty$ for all  $g\in(0,1)$.
\end{lemma}
\begin{proof}
We follow \cite{HJJ}. 
Let $M=\|f\|_\infty$ and assume there exists $\alpha>0$ and a bounded set $A\subset\RR^n$ with positive measure such that (without loss of generality) $u(x,\theta)>M\lambda_0^{-1}+\alpha$ in $A$. 
For any small enough $\delta>0$ we can find a ball $B\subset Q$ such that
$$
\text{meas}(B\cap A)>(1-\delta)\text{meas}(B).
$$
This in particular implies the inequality $\text{meas}(B\backslash A)<\delta\cdot \text{meas}(B)$. 

For such ball we take a nonnegative function $h\in C^\infty_c(\bar{B})$ so that $h\lesssim \text{meas}(B)^{-1}$, and
$$\left|\int\Big(\frac{\chi_B}{\text{meas}(B)} - h\Big)dxd\theta\right| = \left|1-\int h dxd\theta\right|<\delta.$$
Let $\varphi\in W^2$, $\varphi\geq 0$, be the unique solution to the backward sRTE,
$$
-\theta\cdot\nabla_x\varphi + \lambda\varphi = \mathcal{I}_g[\varphi] + h,
$$
which according to the previous lemma it satisfies
$$
\int_Q\varphi dxd\theta\leq\lambda_0^{-1}\int_Q hdxd\theta<\lambda_0^{-1}(1+\delta).
$$
Its existence, uniqueness, and regularity follows similarly as in Theorem \ref{thm:wellposedness_sRTE} by means of \cite[Chapter XXI-Proposition 8]{DL}. 
Then,
$$
\int f\varphi dxd\theta\leq M\int_Q\varphi dxd\theta \leq M\lambda_0^{-1}(1+\delta).
$$
On the other hand,
\beq\label{eq1}
\int_Q f\varphi dxd\theta = \int_Q u^ghdxd\theta,
\eeq
where we split the integration into two integrals, one over $B\cap A$ and another one over $B\backslash A$. By our initial assumption,
$$
\begin{aligned}
\int_{B\cap A} u^g h dxd\theta&\geq (M\lambda_0^{-1}+\alpha)\left(\int_{B\cap A}\frac{\chi_B}{\text{meas}(B)}dxd\theta - \int_{B\cap A}\Big(\frac{\chi_B}{\text{meas}(B)}-h\Big)dxd\theta\right)\\
&\geq (M\lambda_0^{-1}+\alpha)\left(\frac{\text{meas}(B\cap A)}{\text{meas}(B)} - \delta\right) >(M\lambda_0^{-1}+\alpha)(1-2\delta),
\end{aligned}
$$ 
and also
$$
\begin{aligned}
\int_{B\backslash A} u^g h dxd\theta \leq \|u^g\|_{L^2}\|h\|_{L^2}
\leq \|u^g\|_{L^2}\|h\|_\infty^{1/2}\text{meas}(B\backslash A)^{1/2}\lesssim \|u^g\|_{L^2}\delta^{1/2}.
\end{aligned}
$$
We then obtain a contradiction by noticing that \eqref{eq1} and previous estimates imply that for some $C>0$, independent of $\delta$,
$$(M\lambda_0^{-1}+\alpha)(1-2\delta) - C\delta^{1/2}\leq M\lambda_0^{-1}(1+\delta),$$
and consequently
$$M\lambda_0^{-1}+\alpha \leq M\lambda_0^{-1} + C\delta^{1/2},$$
which cannot hold if $\delta$ is sufficiently small.
\end{proof}

\subsection{Well-posedness of fFPE}\label{subsec:fFP}
Let's introduce the bilinear form
\beqq
\mathcal{B}(u,\varphi) := \frac{1}{2}\int_{\SS^{d-1}}\int_{\SS^{d-1}}K(\theta',\theta)(u(\theta') - u(\theta))(\varphi(\theta') - \varphi(\theta))d\theta' d\theta,
\eeqq
for the kernel $K(\theta',\theta)$ as in \eqref{kernel_K}, and define its associated Hilbert space $H^s_{\mathcal{B}}$ given by
\beqq
H^s_{\mathcal{B}} := \{u\in L^{2}_\theta \;: \; \mathcal{B}(u,u)<+\infty\},\quad \text{for}\quad s\in(0,1),
\eeqq 
with inner product
\beqq
\langle u,\varphi\rangle_{H^s_{\mathcal{B}}} := \mathcal{B}(u,\varphi) + \int_{\SS^{d-1}}u\varphi d\theta.
\eeqq
It is worth mentioning how $H^s_{\mathcal{B}}$ and $H^s(\RR^{d-1})$ are related to each other when stereographic coordinates are considered. It turns out that 
passing to stereographic coordinates, $\|u\|^2_{H^s_{\mathcal{B}}}$ is 
equivalent (up to a multiplicative factor) to
\beqq
\int_{\RR^{d-1}}\int_{\RR^{d-1}}\frac{(u(v') - u(v))^2}{|v'-v|^{d-1+2s}\lvpr^{d-1-2s}\lvr^{d-1-2s}}dv' dv + \int_{\RR^{d-1}}\frac{|u(v)|^2}{\lvr^{2(d-1)}}dv,
\eeqq
which yields the inclusion $H^s(\RR^{d-1})\subset \left[ H^s_\mathcal{B}\right]_{\mathcal{S}}$. 
Moreover, one can also verify that
\beqq
\begin{aligned}
&\int_{\RR^{d-1}}\int_{\RR^{d-1}}\frac{(u(v') - u(v))^2}{|v'-v|^{d-1+2s}\lvpr^{d-1-2s}\lvr^{d-1-2s}}dv'dv \sim \|(-\Delta_v)^{s/2}w_{\mathcal{S}}\|^2_{L^2(\RR^{d-1})}+\|u\|^2_{L^2(\SS^{d-1})},
\end{aligned}
\eeqq 
which implies the following equivalence of norms 
\beq\label{equiv_norms}
 \|u\|_{H^s_\mathcal{B}}\sim \|(-\Delta_v)^{s/2}w_{\mathcal{S}}\|_{L^2(\RR^{d-1})}+\|u\|_{L^2(\SS^{d-1})} \sim \|(-\Delta_\theta)^{s/2}u\|.
\eeq

Let us consider now the Hilbert space $\mathcal{H}^s_\epsilon := L^2_x(\RR^d;H^s_\mathcal{B})$, equipped with the norm
$$\|f\|^2_{\mathcal{H}^s_\epsilon }:= \|f\|_{L^2_{x,\theta}}^2 + \epsilon^{2s}\int_{\RR^d}\|f\|^2_{H^s_\mathcal{B}}dx,$$
where we use the shorthand notation $L^2_{x,\theta}$ for $L^2(\RR^d\times\SS^{d-1})$. 
Its dual will be denoted by $(\mathcal{H}^s_\epsilon)'$. Denoting the transport operator $T = \theta\cdot \nabla_x$, we define the solution space for the fractional Fokker--Planck equation as
$$\mathcal{Y}^s_\epsilon := \{f\in \mathcal{H}^s_\epsilon\;:\; Tf\in(\mathcal{H}^s_\epsilon)'\},\quad\text{with norm}\quad \|f\|_{\mathcal{Y}^s_\epsilon}:= \|f\|_{\mathcal{H}^s_\epsilon} + \|Tf\|_{(\mathcal{H}^s_\epsilon)'}.$$

\begin{definition}\label{def:weak_sol}
We say that a function $u\in\mathcal{Y}^s_\epsilon$ is a weak solution of \eqref{fFP_diff} if for all $\varphi\in C^\infty_c(\RR^d\times\SS^{d-1})$, it satisfies
\begin{equation}\label{weak_sol}
\begin{aligned}
&\int  -u(\theta \cdot\nabla_x\varphi) + \lambda u\varphi dxd\theta +\epsilon^{2s}\int\sigma(x)\mathcal{B}(u,\varphi)dx = \int f\varphi dxd\theta.
\end{aligned}
\end{equation}
\end{definition}

The well-posedness and some properties of the solutions to the fractional Fokker--Planck equation are summarized next. Additional properties are stated in Theorem \ref{thm:wellposedness_FP2} after we obtain a convergence result for sRTE-solutions in the highly forward-peaked limit.

\begin{theorem}\label{thm:wellposedness_FP}
Let $\sigma,\lambda$ satisfying \eqref{cond_coeff}. For any $f\in L^2_{x,\theta}$ there exists a unique solution $u\in\mathcal{Y}^s_\epsilon$ to \eqref{fFP_diff}. Moreover,
there is $C>0$ such that $\|u\|_{\mathcal{Y}^s_\epsilon}\leq C\|f\|_{L^2_{x,\theta}}$ for all $f\in L^2_{x,\theta}$.
\end{theorem}
\begin{proof}
\noindent 1. Following the definition of weak solution, we introduce the bilinear form $a:\mathcal{H}^s_\epsilon\times C^\infty_c\to \RR$ so that the left hand side of \eqref{weak_sol} is given by $a(u,\varphi)$, while the right hand defines the linear operator $L(\varphi)$ mapping $L:C^\infty_c\to \RR$. The bilinear form becomes a bounded linear operator in $\mathcal{H}^s_\epsilon$ when freezing the second component $\varphi\in C^\infty_c$, and moreover, using integration by parts we get
$$
a(\varphi,\varphi)\geq \min\{\lambda_0,\sigma_0\}\|\varphi\|_{\mathcal{H}^s_\epsilon}^2,\quad\forall \varphi \in C^\infty_c.
$$
Then, applying \cite[Theorem 2.4]{BP} (see also \cite{L}) for the pre-Hilbert space $\mathcal{F} = (C^\infty_c, \|\cdot\|_{\mathcal{H}^s_\epsilon})$ we obtain the existence of a weak solution $u\in \mathcal{H}^s_\epsilon$, which in addition satisfies $u\in \mathcal{Y}^s_\epsilon$, since \eqref{weak_sol} allows to define $Tu$ in the sense of distributions, and $T$ is a bounded linear operator in $\mathcal{H}^s_\epsilon$.

It is not hard to verify that the set $C^\infty_c(\RR^d\times\SS^{d-1})$ is dense in $\mathcal{Y}^s_\epsilon$, and this allows us to obtain the integration by parts formula
\beq\label{ibp_formula}
\langle u_1,Tu_2\rangle_{\mathcal{H}^s_\epsilon,(\mathcal{H}^s_\epsilon)'} = -\langle u_2,Tu_1\rangle_{\mathcal{H}^s_\epsilon,(\mathcal{H}^s_\epsilon)'}, \quad\forall u_1,u_2\in \mathcal{Y}^s_\epsilon,
\eeq
In particular, $\langle u,Tu\rangle_{\mathcal{H}^s_\epsilon,(\mathcal{H}^s_\epsilon)'} = 0$. 
The identity \eqref{ibp_formula} is obtained as in \cite{Ba}. More recent proofs of similar density results are presented in \cite{AM,BS}. The density used in this paper is easier to establish as we do not deal with boundaries.

It follows from the weak formulation of the fractional Fokker--Planck equation that
\begin{equation*}
\langle u,Tu\rangle_{\mathcal{H}^s_\epsilon,(\mathcal{H}^s_\epsilon)'} + \int \lambda |u|^2 dxd\theta +\epsilon^{2s}\int\sigma(x)\mathcal{B}(u,u)dx = \int fudxd\theta,
\end{equation*}
and consequently,
$$
\min\{\lambda_0,\sigma_0\}\|u\|_{\mathcal{H}^s_\epsilon}^2\leq \|f\|_{L^2_{x,\theta}}\|u\|_{L^2_{x,\theta}} \leq  \frac{1}{2}\delta^{-2}\|f\|^2_{L^2_{x,\theta}}+\frac{1}{2}\delta^{2}\|u\|^2_{L^2_{x,\theta}},
$$
where we choose $\delta>0$ small enough so that we can absorb the last term with the left hand side.

Uniqueness of solutions follows directly from the continuous dependence estimate.
\end{proof}

\subsection{Convergence of sRTE and some properties of fFPE} 
In what follows, we take the kernel $k_g(x,\theta,\theta')$ to be in the narrow beam scaling, thus, we assume it satisfies condition $ii$. Then,
\beq\label{HGk}
k_g(x,\theta,\theta') := \frac{\epsilon^{2s}(1+g)^m b(x)}{\big(1+g^2-2g\theta\cdot\theta'\big)^{\frac{d-1}{2}+s}},
\eeq
whose limit as  $g\to 1$ is given by \eqref{limiting_kernel}

We next extend the weak-$L^2$ convergence result in \cite{AS} to the topology induced by the Wasserstein distance.
\begin{theorem}\label{thm:approx_fFP}
For any open bounded $\Omega\subset\RR^d$ and a nonnegative and integrable function $f$, there exists an increasing sequence $\{g_k\}_k\subset (0,1)$ converging to 1, and $\kappa_0>0$ depending on $\|\lambda\|_{C^1}$ and $\|f\|_{L^1}$, such that
\[
\mathcal{W}^1_{\bar{\Omega},\kappa}(u^{g_k},u)\to 0\quad\text{as}\quad k\to\infty,
\]
for all $\kappa\geq\kappa_0$, and where $u^{g_k}$ are the solutions to \eqref{sRTE} with kernel \eqref{HGk} and parameter $g=g_k$, and $u$ is the solution to the fractional Fokker--Planck equation \eqref{fFP_diff}.
\end{theorem}
\begin{proof}
Let $f$ be a nonnegative integrable function and consider an increasing sequence $\{g_n\}_{n\geq 1}\subset(0,1)$ so that $g_n\to 1$ as $n\to\infty$. We have that $u^{g_n}\geq 0$ for all $n$, and therefore, we can regard $\lambda u^{g_n}/\|f\|_{L^1}$ as a family of Radon probability measures since (as in Lemma \ref{lemma:L1_est})
$$
\int \lambda u^{g_n}dxd\theta = \|f\|_{L^1},\quad\forall n\geq 1.
$$
In the compact set $\bar{\Omega}\times\SS^{d-1}$, weak$^*$-compactness for measures (Banach-Alaoglu theorem) yields then the existence of a subsequence $\{u^{g_{n_k}}\}_k$, and another Radon probability measure $\mu$, for which
\begin{equation}\label{weak*_conv}
\frac{1}{\|f\|_{L^1}}\int_{\bar{\Omega}\times\SS^{d-1}} \varphi \lambda u^{g_{n_k}}dxd\theta \to \int_{\Omega\times\SS^{d-1}} \varphi \mu(x,\theta),\quad\forall \;\varphi\in C_b(\bar{\Omega}\times\SS^{d-1}),
\end{equation}
with $C_b(\bar{\Omega}\times\SS^{d-1})$ the set of bounded continuous functions in $\bar{\Omega}\times\SS^{d-1}$. Writing instead $\psi = \lambda\varphi/\|f\|_{L^1}$, and since $\lambda\in \text{Lip}(\RR^d\times\SS^{d-1})$, then we can find $\kappa_0>0$ such that for all $\kappa\geq \kappa_0$,
\beqq
\int_{\Omega\times\SS^{d-1}} \psi u^{g_{n_k}}dxd\theta \to \int_{\Omega\times\SS^{d-1}} \psi \big(\|f\|_{L^1}\lambda^{-1}\mu(x,\theta)\big),\quad\forall \;\psi\in BL_{1,\kappa}(\bar{\Omega}\times\SS^{d-1}),
\eeqq

Due to uniqueness of the distributional limit, one verifies that $\|f\|_{L^1}\lambda^{-1}\mu = u dxd\theta$, with $u$ solution to \eqref{sRTE} and limiting scattering cross-section \eqref{limiting_kernel} (i.e. the fractional Fokker--Planck equation \eqref{fFP_id}). This follows from similar computations as in \cite{AS}, leading to
$$
\|u^g\|_{L^2_{x,\theta}}+\|\theta\cdot\nabla_xu^g\|_{L^2_{x,\theta}} + \|\mathcal{I}_gu^g\|_{L^2_{x,\theta}}\leq C\|f\|_{L^2_{x,\theta}},
$$
uniformly in $g$, which subsequently imply weak-$L^2_{x,\theta}$ convergence to a limiting function $u$, i.e.
$$
\theta\cdot\nabla_xu^g + \lambda u^g - \mathcal{I}_g[u^g]\;\rightharpoonup\; \theta\cdot\nabla_xu + \lambda u - \mathcal{I}_\theta [u]\quad\text{as}\quad g\to 1.
$$
The previous then implies
\beqq
\mathcal{W}^{1}_{\bar{\Omega},\kappa}(u^{g_k},u)\to 0\quad\text{as}\quad k\to\infty.
\eeqq
\end{proof}

With the aid of the previous convergence theorem, some properties of the radiative transfer solution are inherited by the limiting function. This yields the following additional properties for the solution to the fractional Fokker--Planck equation.
\begin{theorem}\label{thm:wellposedness_FP2}
Let $\sigma,\lambda$ and $f$ as in Theorem \ref{thm:wellposedness_FP}. The unique solution $u$ to the fractional Fokker--Planck equation \eqref{fFP_diff} satisfies the following additional properties:
\begin{enumerate}
\item[1.] (non-negativity) if $f\geq 0$ then $u\geq 0$;
\item[2.] (continuous dependence for bounded sources) $\|u\|_\infty\leq\lambda_0^{-1}\|f\|_\infty$ for all $f\in L^\infty_{x,\theta}\cap L^2_{x,\theta}$.
\end{enumerate}
\end{theorem}
\begin{proof}
\noindent 1. The non-negativity of $u$ for a source term $f\geq 0$ follows from the analogous property of the sRTE in Theorem \ref{thm:wellposedness_sRTE} and the convergence in Theorem \ref{thm:approx_fFP}.\\

\noindent 2. The continuity in the $L^\infty$-norm follows in a similar fashion as Lemma \ref{lemma:Linfty_est} for sRTE. We point out here the main differences in the argument. The exact same computations yield to \eqref{eq1} for a test function $\varphi_g$, solution to the backward sRTE. Instead of estimating directly the right hand side $\int_Qu^g hdxd\theta$, we take the limit as $g\to 1$ and obtain
$$
\int f\varphi dxd\theta = \int uh dxd\theta,
$$
where $\varphi$ is the weak-$L^2_{x,\theta}$ limit of the solutions $\varphi_g$, thus it solves a backward fractional Fokker-Planck equation. It is of course nonnegative and with $L^1_{x,\theta}$-norm bounded by $\lambda_0^{-1}(1+\delta)$. 
The proof is then complete by repeating the remaining steps in the proof of Lemma \ref{lemma:Linfty_est}.
\end{proof}

\subsection{Regularity of solutions to fFPR}
The regularity results of this section are based on a local representation of the fractional Fokker--Planck equation, in a neighborhood of an arbitrary $(x_0,\theta_0)\in\RR^d\times\SS^{d-1}$. The idea is to pseudo-localize the equation in the sense that we can approximate a solution to the Fokker--Planck equation with smooth and compactly supported functions, satisfying a Fokker--Planck equation with a smooth source that might not be compactly supported.

\begin{theorem}[Regularity of solutions]\quad
\begin{enumerate}
\item[1.](Sub-elliptic regularity) For $\sigma,\lambda$ and $f$ as in Theorem \ref{thm:wellposedness_FP}, the unique solution $u$ to the fractional Fokker--Planck equation is a strong solution, and more precisely
$(-\Delta_\theta)^{s}u,\;Tu \in L^2_{loc}(\RR^d;L^2(\SS^{d-1}))$; 
\item[2.] (Continuity) For $f\in L^\infty_{x,\theta}\cap L^2_{x,\theta}$, the solution $u$ to the fractional Fokker--Planck equation \eqref{fFP_diff} is locally H\"older continuous in $\RR^d\times\SS^{d-1}$, and consequently, it belongs to $C(\RR^d\times\SS^{d-1})$. 
\end{enumerate}
\end{theorem}
\begin{proof}
\noindent 1. 
The subelliptic character of fractional order kinetic equation is known to derive, for example, from commutator identities and energy estimates \cite{Bo,A}. In order to adapt these arguments to the setting of the fractional Laplacian on the unit sphere, we pseudo-localize the equation and consider suitable local coordinates that allow us to apply known results.

By a standard localization and mollification technique (see Appendix \ref{appdx:loc_moll}) we can assume without loss of generality that $u$ is a smooth and compactly supported solution to the fractional Fokker--Planck equation \eqref{fFP_diff}, or equivalently \eqref{fFP_id}, with support contained in $\RR^n\times\SS^{d-1}_+$, and for a right-hand side $f\in C^\infty(\RR^{d}\times\SS^{d-1})\cap L^2(\RR^{d}\times\SS^{d-1})$. By considering beam coordinates, we see that for $\theta(v),\theta'(v')\in \SS^{d-1}_+$,
$$
\left(1-\theta\cdot\theta'\right)_{\mathcal{B}} = \frac{\lvr\lvpr -v\cdot v' -1}{\lvr\lvpr}.
$$
Therefore, multiplying \eqref{fFP_id} by $\lvr^{-\frac{d}{2}}$ and setting $\tilde{u}(x,v) = \lvr^{-\frac{d}{2}-1}u(x,v)$, we obtain the kinetic equation
\beqq
\partial_{x^d}\tu+v\cdot\nabla_{x'}\tu = \int \tilde{k}(x,v,v')(\tu(x,v')-\tu(x,v))dv'+ \tilde{f}(x,v),
\eeqq
for $\tilde{f} = \sum_{j=0}^3f_j$ with
\beq\label{func_tf}
\begin{aligned}
&f_0(x,v) = \frac{f(x,v)}{\lvr^{\frac{d}{2}}},\quad  f_1(x,v)=- \lambda(x)\frac{u(x,v)}{\lvr^{\frac{d}{2}}} \\
&f_2(x,v) = \frac{u(x,v)}{\lvr^{\frac{d}{2}+1}}\int \tilde{k}(x,v,v')\left(1 - \frac{\lvr^{\frac{d+1}{2}}}{\lvpr^{\frac{d+1}{2}}}\right)dv',\\
&f_3(x,v) = \int \tilde{k}(x,v,v')\frac{u(x,v')}{\lvpr^{\frac{d+1}{2}}}\left(\frac{1}{\lvr^{\frac{1}{2}}} - \frac{1}{\lvpr^{\frac{1}{2}}}\right)dv',\\
\end{aligned}
\eeq
and for a kernel defined by
\beq\label{kernel_tk}
\tilde{k}(x,v',v) = 2^{m-\frac{d-1}{2}-s}\epsilon^{2s}\sigma(x)\frac{\lvr^{s}\lvpr^{s}}{h(v,v')^{\frac{d-1}{2}+s}},\quad\text{with}\quad h(v,v') = \lvr\lvpr -v\cdot v' -1.
\eeq
Notice that here, $f(x,v)$ stands for $[f|_{\SS^{d-1}_+}]_{\mathcal{B}}$ which is in $L^2(\RR^d\times\RR^{d-1}; \lvr^{-d}dxdv)$.

We now estimate the $L^2$-norm of $(-\Delta_\theta)^su$ in terms of the kernel $\tilde{k}(x,v,v')$. By comparing \eqref{fFP_diff} and \eqref{fFP_id}, we have
\beq\label{fLap_tk}
\begin{aligned}
\|(-\Delta_\theta)^su\|_{L^2_{x,\theta}} &= \|\mathcal{I}_\theta(u) - cu\|_{L^2_{x,\theta}} 
\lesssim \left(\int \left([\mathcal{I}_\theta(u)]_{\mathcal{B}} + c[u]_{\mathcal{B}}\right)^2\frac{dv}{\lvr^{d}}dx\right)^{1/2}\\
&\lesssim \left\| \int \tilde{k}(x,v,v')(\tilde{u}(v') - \tilde{u}(v))dv'\right\|_{L^2_{x,v}} +\left\| u\right\|_{L^2_{x,\theta}}.
\end{aligned}
\eeq
The last inequality follows, in particular, from the fact that $\tilde{f}\in L^2(\RR^d\times\RR^{d-1};dxdv)$, which we prove next. The latter spaces is denoted in this section by $L^2_{x,v}$. Indeed, we directly have that $\|f_0\|_{L^2_{x,v}}\lesssim \|f\|_{L^2_{x,\theta}}$ and $\|f_1\|_{L^2_{x,v}}\lesssim \|u\|_{L^2_{x,\theta}}$, and moreover,
\beqq
\begin{aligned}
f_2(x,v) &= \frac{u(x,v)}{2\lvr^{\frac{d}{2}+1}}\int \tilde{k}(x,v,v+z)\left(2-\frac{\lvr^{\frac{d+1}{2}}}{\langle v+z\rangle^{\frac{d+1}{2}}}-\frac{\lvr^{\frac{d+1}{2}}}{\langle v-z\rangle^{\frac{d+1}{2}}}\right)dz\\
&\quad + \frac{u(x,v)}{2\lvr^{\frac{d}{2}+1}}\int \big(\tilde{k}(x,v,v-z) - \tilde{k}(x,v,v+z)\big)\left(1-\frac{\lvr^{\frac{d+1}{2}}}{\langle v-z\rangle^{\frac{d+1}{2}}}\right)dz;\\
f_3(x,v) &= \int \tilde{k}(x,v,v+z)\left(\frac{u(v+z)}{\langle v+z\rangle^{\frac{d+1}{2}}} - \frac{u(v)}{\lvr^{\frac{d+1}{2}}}\right)\left(\frac{1}{\lvr^{\frac{1}{2}}}-\frac{1}{\langle v+z\rangle^{\frac{1}{2}}}\right)dz\\
&\quad + \frac{u(x,v)}{2\lvr^{\frac{d+1}{2}}}\int \tilde{k}(x,v,v+z)\left(\frac{2}{\lvr^{\frac{1}{2}}}-\frac{1}{\langle v+z\rangle^{\frac{1}{2}}}-\frac{1}{\langle v-z\rangle^{\frac{1}{2}}}\right)dz\\
&\quad + \frac{u(x,v)}{2\lvr^{\frac{d+1}{2}}}\int \big(\tilde{k}(x,v,v-z) - \tilde{k}(x,v,v+z)\big)\left(\frac{1}{\lvr^{\frac{1}{2}}}-\frac{1}{\langle v-z\rangle^{\frac{1}{2}}}\right)dz\\
\end{aligned}
\eeqq
We show in Appendix \ref{appdx:func_h} that for some $\beta\in(0,1)$,
\beqq
\frac{\beta}{2}|v-v'|^{2}\leq h(v,v')\leq \frac{1}{2}|v-v'|^{2},
\eeqq
which then implies the inequality (see \eqref{ineq_h})
\beqq
\big|\tilde{k}(x,v,v-z) - \tilde{k}(x,v,v+z)\big|\lesssim \frac{\lvr^{2s}}{|z|^{d-1+2s-1}}.
\eeqq
The above (and the compact support of $u$) allows us to estimate the $L^2$-norm of $f_2$ and the last two terms in $f_3$, by a constant times $\|u\|_{L^2_{x,\theta}}$. For the remaining term in $f_3$, we notice that
\beqq
\begin{aligned}
&\left\| \int \tilde{k}(x,v,v+z)\left(\frac{u(v+z)}{\langle v+z\rangle^{\frac{d+1}{2}}} - \frac{u(v)}{\lvr^{\frac{d+1}{2}}}\right)\left(\frac{1}{\lvr^{\frac{1}{2}}}-\frac{1}{\langle v+z\rangle^{\frac{1}{2}}}\right)dz\right\|^2_{L^2_{x,v}}\\
&\lesssim \int \int \tilde{k}(x,v,v+z)\left(\frac{u(v+z)}{\langle v+z\rangle^{\frac{d+1}{2}}} - \frac{u(v)}{\lvr^{\frac{d+1}{2}}}\right)^2dzdv\\
&\lesssim \int[\mathcal{I}_\theta(u)]_{\mathcal{B}}udxd\theta + \|u\|_{L^2_{x,\theta}}^2\\
&\lesssim \|(-\Delta_\theta)^{s/2}u\|^2_{L^2_{x,\theta}} + \|u\|_{L^2_{x,\theta}}^2.
\end{aligned}
\eeqq
Summarizing, we have shown that
\beqq
\|\tilde{f}\|_{L^2_{x,v}}\lesssim \|f\|_{L^2_{x,\theta}}+\|u\|_{L^2_{x,\theta}}+\|(-\Delta_\theta)^{s/2}u\|_{L^2_{x,\theta}}.
\eeqq
It then follows from the above and results from \cite{A} (see Appendix \ref{appdx:subellipticity}) that
\beqq
\|(-\Delta_v)^s\tu\|_{L^2}+\|(-\Delta_x)^{\frac{s}{1+2s}}\tu\|_{L^2}\lesssim \|f\|_{L^2} + \|\tu\|_{L^2} + \|(-\Delta_v)^{s/2}\tu\|_{L^2},
\eeqq
and subsequently (see Remark \ref{remark_D}),
\beqq
\left\| \int \tilde{k}(x,v,v')(\tu(x,v')-\tu(x,v))dv'\right\|_{L^2_{x,v}}\leq \|f\|_{L^2} + \|\tu\|_{L^2} + \|(-\Delta_v)^{s/2}\tu\|_{L^2_{x,v}}.
\eeqq
Combining the above, \eqref{fLap_tk}, and the estimates $\|\tu\|_{L^2_{x,v}}\leq \|u\|_{L^2_{x,\theta}}$ and
\beqq
\|(-\Delta_v)^{s/2}\tu\|^2_{L^2_{x,v}}\lesssim \int \int \tilde{k}(x,v,v+z)
\left(\tilde{u}(v+z)-\tilde{u}(v)\right)^2dzdv\lesssim  \|(-\Delta_\theta)^{s/2}u\|^2_{L^2_{x,\theta}} + \|u\|_{L^2_{x,\theta}}^2,
\eeqq
we finally deduce that
\beqq
\|(-\Delta_\theta)^su\|^2_{L^2_{x,\theta}}\lesssim \|f\|_{L^2_{x,\theta}} + \|u\|_{L^2_{x,\theta}} + \|(-\Delta_\theta)^{s/2}u\|^2_{L^2_{x,\theta}}.
\eeqq 
Using equation \eqref{fFP_diff} one also has that $Tu\in L^2(\RR^d\times \SS^{d-1})$ with and analogous estimate. 

Lastly, given any bounded set $\Omega\subset\RR^d$  and a partition of unity defined on $\RR^d\times\SS^{d-1}$, the above yields that $(-\Delta_\theta)^s u$ and $Tu$ belong to  $L^2(\Omega\times\SS^{d-1})$ for any solution to the fractional Fokker--Planck equation \eqref{fFP_diff} with source $f\in L^2(\RR^d\times\SS^{d-1})$.\\

2. Similarly, the continuity of solutions is a consequence of the continuity property satisfied by solutions to fractional order kinetic equations, and more generally satisfied by solutions to some linear transport equations with singular scattering kernels. We again localize the solution in order to use beam coordinates on the sphere. 

We assume without loss of generality that $u(\cdot,\cdot)$ is a compactly supported solution to the fractional Fokker--Planck equation \eqref{fFP_diff}, with support in $\mathcal{V}\times\SS^{d-1}_+$, for some open bounded set $\mathcal{V}\subset\RR^d$, and for a source term $f\in L^2_{x,\theta}\cap L^\infty_{x,\theta}$. By the previous regularity result, it is indeed a strong solution.

By considering beam coordinates, we saw above that $\tilde{u}(x,v) = \lvr^{-\frac{d}{2}-1}u(x,v)$ solved
\beqq
\partial_{x^d}\tu+v\cdot\nabla_{x'}\tu = \int \tilde{k}(x,v,v')(\tu(x,v')-\tu(x,v))dv'+ \tilde{f}(x,v)
\eeqq
with $\tilde{f}$ and $\tilde{k}$ defined respectively in \eqref{func_tf} and \eqref{kernel_tk}. It turns out that the kernel (see Appendix \ref{appdx:func_h}) satisfies the point-wise inequalities,
\beqq
\frac{1}{|v-v'|^{d-1+2s}}\lesssim\tilde{k}(x,v,v')\lesssim \frac{\lvr^{s}\lvpr^{s}}{|v-v'|^{d-1+2s}},
\eeqq
which directly implies coercivity of the integral operator associated to $\tilde{k}$.  Namely, for every $R\geq1$, there is $C>0$ so that for all $\varphi:\RR^{d-1}\to \RR$ compactly supported inside the ball $B_{R}$, 
\beqq
\begin{aligned}
&C^{-1}\int_{\RR^{d-1}}\int_{\RR^{d-1}} \frac{|\varphi(v) -\varphi(v')|^2}{|v-v'|^{d-1+2s}}dvdv' \\
&\leq - \int_{\RR^{d-1}}\left(\int_{\RR^{d-1}}\tilde{k}(x,v,v')(\varphi(v')-\varphi(v))dv'\right)\varphi(v)dv + \|\varphi\|^2_{L^2(\RR^{d-1})}.
\end{aligned}
\eeqq
The same inequality helps us to deduce the non-degeneracy condition (in the case $s<1/2$):
\beqq
\inf_{|e|=1}\int_{B_r(v)}\max((v'-v)\cdot e,0)^2\tilde{k}(x,v,v')dv'\geq c r^{2-2s},\quad\forall v\in B_R.
\eeqq
Finally, we also verify that
\beqq
\begin{aligned}
&\int_{\RR^{d-1}\backslash B_r(v)}\tilde{k}(x,v,v')dv' \lesssim r^{-2s},\quad \forall r>0 \text{ and }v\in B_R;\\
&\int_{B_R\backslash B_r(v')}\tilde{k}(x,v,v')dv \lesssim r^{-2s},\quad \forall r>0 \text{ and }v'\in B_R.
\end{aligned}
\eeqq

As long as we have $\tilde{f}\in L^\infty(\RR^d\times B_R)$, the H\"older regularity result in \cite{IS} guarantees that $\tilde{u}$ is H\"older continuous in $\mathcal{V}\times\RR^{d-1}$, and consequently, $u(x,\theta)$ is continuous in a neighborhood of $(x_0,\theta_0)$.

In order to show that $\tilde{f}$ is bounded, and since $f,u \in L^\infty_{x,\theta}$ (by virtue of Theorem \ref{thm:wellposedness_FP2}), it remains to deduce the boundedness of $f_2$ and $f_3$ in \eqref{func_tf}. Indeed, denoting $\varphi(v')=\frac{\lvr^{\frac{d+1}{2}}}{\lvpr^{\frac{d+1}{2}}}$, we see that
\beqq
\begin{aligned}
f_2(x,v)&=\int \tilde{k}(x,v,v')\left(\varphi(v) - \varphi(v')\right)dv'\\
&=\frac{1}{2}\int \tilde{k}(x,v,v+z)(2\varphi(v) - \varphi(v+z) - \varphi(v-z))dz \\
&\quad+\frac{1}{2}\int (\tilde{k}(x,v,v-z)-\tilde{k}(x,v,v+z))(\varphi(v)-\varphi(v-z))dz.
\end{aligned}
\eeqq
Moreover, since $\varphi$ is twice-continuously differentiable and using also that $u(x,v)$ is compactly supported, we have that (again from \eqref{ineq_h})
\beqq
\begin{aligned}
\left|f_2(x,v)\right|&\lesssim |u(x,v)|\left(\int_{|z|<1} \frac{1}{|z|^{d-1-2(1-s)}} dz + \int_{|z|>1} \frac{1}{|z|^{d-1+2s}} dz \right)\lesssim \|u\|_{\infty}.
\end{aligned}
\eeqq
We similarly obtain $|f_3(x,v)| \lesssim \|u\|_{\infty}$, which allows us to conclude that $\|\tilde{f}\|_\infty\lesssim \|f\|_\infty$ and the proof is complete.
\end{proof}

We saw above that solutions to the fFPE can be viewed locally as solutions to evolution integro-differential kinetic equations, and it was this precise local form that allowed us to derive the regularity properties. How the optimal H\"older exponent  associated to solutions of \eqref{fFP_diff} depends on the regularity of the source $f$ term (e.g., Lipschitz continuous) remains unclear. 

Several recent results related to the more general and/or closely related non-local kinetic models, such as the Boltzmann equation without cut-off, may be found in \cite{IS,IS2,IS3,S}. 
It is expected that results along those lines could lead to quantitative approximation estimates in the 1-Wasserstein sense between the (stationary) radiative transfer solution, in the narrow beam regime, and the fractional Fermi pencil-beam solution. We proved above that the radiative transfer solution converged in the highly forward-peaked limit, in a weak sense (and with no known estimates), to the fractional Fokker--Planck solution (Theorem \ref{thm:approx_fFP}), and furthermore, the latter solution can be subsequently approximated by a pencil-beam in the narrow beam regime (Theorem \ref{thm:narrow_beam}, above).
The main obstacle that prevents the application of the regularity results in \cite{IS,IS2,IS3,S} resides in the a-priori regularity assumed on the solutions. The H\"older estimates proved in \cite{IS,S} hold for weak solutions and thus are appropriate for our work. In contrast, the local Schauder estimates and their global extension obtained respectively in \cite{IS2} and \cite{IS3}, assume solutions to be at least classical. A generalization of the Schauder estimates for weak solutions would be necessary to quantify the convergence of radiative transfer to the Fokker-Planck equation, something we do not pursue here. 

\section{Analysis of the Fractional Fermi pencil-beam equation}\label{sec:fFpb}
\subsection{Existence of solutions and properties}
Let $F\in C(\RR_+;\mathcal{S}'(\RR^{2(n-1)}))$ and $G\in \mathcal{S}'(\RR^{2(n-1)})$. In this section we will assume $\tilde{\sigma}(X^n),\tilde{\lambda}(X^n)\in C(\bar{\RR}_+)$. Analogously to the non-fractional case we have an explicit characterization of the solutions to the fractional Fermi pencil-beam equation. However, this time it is explicit only in the Fourier domain.

We denote by $\mathfrak{F}_{X'}[f](\xi,X^d,V) := \int e^{-\ii X'\cdot\xi} f(X,V)dX'$ the Fourier transformation of $f$ with respect to $X'$, and similarly we denote by $\mathfrak{F}_V$ the Fourier transform operator with respect to the angular variable $V$. We also write $\mathfrak{F}_{X',V}=\mathfrak{F}_{X'}\mathfrak{F}_{V}$.
\begin{lemma}\label{lemma:solution_fFPB}
For the above choice of parameters there exists a unique solution to 
\begin{equation}\label{fFPB}
\left\{\begin{array}{ll}
\partial_{X^n}U + V\cdot \nabla_{X'}U + \tilde{\sigma}(-\Delta_V)^sU + \tilde{\lambda} U =F, &X=(X',X^n)\in\RR^n_+,\; V\in\RR^{n-1}\\
U= G,&(X',V)\in \RR^{2(n-1)},\; X^n=0,
\end{array}\right.
\end{equation}
whose Fourier Transform (with respect to transversal and angular variables) is given by
\begin{equation}\label{fFBP_sol}
\begin{aligned}
\mathfrak{F}_{X',V}[U](\xi,X^n,\eta) &= e^{-\int^{X^n}_0\tilde{\lambda}(r)dr}\mathfrak{F}_{X',V}[G](\xi,\eta+X^n\xi)e^{-\int^{X^n}_0|\eta+(X^n-t)\xi|^{2s}\tilde{\sigma}(t)dt}\\
&\quad + \int^{X^n}_0e^{-\int^{X^n}_t\tilde{\lambda}(r)dr}\mathfrak{F}_{X',V}[F](\xi,t,\eta+(X^n-t)\xi)e^{-\int^{X^n}_t|\eta + (X^n-r)\xi|^{2s}\tilde{\sigma}(r)dr} dt.
\end{aligned}
\end{equation}
\end{lemma}
\begin{proof}
It follows in a similar fashion as in the non-fractional case. We refer the reader to the proof of \cite[Proposition 4.1]{BP}.
\end{proof}

\begin{lemma}\label{lemma:radial_sym} Let $U$ as in the previous lemma with null interior source ($F=0$) and $G=\delta(X)\delta(V)$. Then,
$$\int_{\RR^{2(d-1)}} \big(X'\cdot \xi_0\big)U(X',X^d,V)dX'dV=0,\quad\text{for all }\xi_0\in\RR^{d-1}\text{ and }X^d>0.$$
$$\int_{\RR^{d-1}} (-\Delta_V)^sU(X',X^d,V)dV=0,\quad\text{for all }X'\in\RR^{d-1}\text{ and }X^d>0.$$
\end{lemma}
\begin{proof}
Using that integration is equivalent to the zero Fourier coefficient we get
$$
\begin{aligned}
&\int_{\RR^{2(d-1)}} \big(X'\cdot \xi_0\big)U(X',X^d,V)dX'dV\\
&\hspace{2em}=\int_{\RR^{d-1}} \big(X'\cdot \xi_0\big)\mathfrak{F}_{V}[U](X',X^d,0)dX'\\
&\hspace{2em}=\ii\xi_0\cdot \big[\nabla_{\xi}\mathfrak{F}[U]\big](0,X^d,0)\\
&\hspace{2em}=\ii e^{-\int^{X^n}_0\tilde{\lambda}(r)dr}\left(2s\textstyle \int^{X^d}_0(X^d-t)^{2s}\tilde{\sigma}(t)dt\right)\big(\xi_0\cdot \hat{\xi} \big)|\xi|^{2s-1}e^{-|\xi|^{2s}\int^{X^n}_0(X^n-t)^{2s}\tilde{\sigma}(t)dt}\Big|_{\xi =0}=0.
\end{aligned}
$$
Similarly,
$$
\begin{aligned}
\int_{\RR^{d-1}} (-\Delta_V)^sU(X',X^d,V)dV&=\mathfrak{F}_{V}[(-\Delta_V)^sU](X',X^d,0)\\
&=|\eta|^{2s}\mathfrak{F}[U](X',X^d,\eta)\big|_{\eta=0}\quad =0.
\end{aligned}
$$
\end{proof}

\begin{lemma}\label{lemma:IbyP}
Let $U$ be as in Lemma \ref{lemma:solution_fFPB} with null interior source ($F=0$) and $G=\delta(\cdot-Y',\cdot-W)$. Then, for any $\Phi\in C(\RR^n_+\times\RR^{n-1})\cap H^1(\RR^n_+;L^2(\RR^{2(n-1)}))$ we have the following integration by parts formula
$$
\int_{\RR^n_+\times\RR^{n-1}}U\partial_{X^n}\Phi dXdV = -\int_{\RR^n_+\times\RR^{n-1}}\big(\partial_{X^n}U\big)\Phi dXdV - \Phi(Y',0,W).
$$
\end{lemma}
\begin{proof}
We use Plancherel's formula to write
$$
\begin{aligned}
&\int_{\RR^n_+\times\RR^{n-1}}U\partial_{X^n}\Phi dXdV = \int_{\RR^{2(n-1)}\times\RR_+}\mathfrak{F}_{X',V}[U]\big(\partial_{X^n}\mathfrak{F}_{X',V}[\Phi]\big)dXdV d\xi'd\eta dX^n\\
=&- \int_{\RR^{2(n-1)}\times\RR_+} \!\!\!\!\!\! \big(\partial_{X^n}\mathfrak{F}_{X',V}[U]\big)\mathfrak{F}_{X',V}[\Phi]dXdV d\xi'd\eta dX^n
-\int_{\RR^{2(n-1)}}\!\!\!\!\!\mathfrak{F}_{X',V}[U]|_{X^n=0}\mathfrak{F}_{X',V}[\Phi]|_{X^n=0}d\xi d\eta
\\
=&-\int_{\RR^n_+\times\RR^{n-1}}\big(\partial_{X^n}U\big)\Phi dXdV
- \int e^{\ii Y'\cdot \xi + \ii W\cdot\eta}\mathfrak{F}_{X',V}[\Phi]|_{X^n=0}d\xi d\eta.
\end{aligned}
$$
\end{proof}

\begin{lemma}\label{lemma:FPb_Lip}
Let $F\in C_c(\bar{Q})$ and Lipschitz continuous with respect to $Z=(X',V)$, and $G=0$. Let $U$ be the solution to \eqref{fFPB}. There exists $C>0$ so that $U_{X^n}(Z) = U(X',X^n,V)$ satisfies
$$|U_{X^n}(Z_1) -U_{X^n}(Z_2) |\leq C\Lip(F)|Z_1-Z_2|,$$
for all $Z_1,Z_2\in\RR^{2(n-1)}$ and $X^n>0$.
\end{lemma}
\begin{proof}
For arbitrary $(X'_1,X^n,V_1),(X'_2,X^n,V_2)\in \RR^{2(n-1)}$ we have
$$
\begin{aligned}
&U(X'_1,X^n,V_1)-U(X'_2,X^n,V_2) \\
&=\int e^{\ii(X'_1\cdot \xi)+\ii(V_1\cdot \eta)}\mathfrak{F}_{X',V}[U](\xi,X^n,\eta) d\xi d\eta-\int e^{\ii(X'_2\cdot \xi)+\ii(V_2\cdot \eta)}\mathfrak{F}_{X',V}[U](\xi,X^n,\eta) d\xi d\eta\\
&=\int \int^{X^n}_0e^{-\int^{X^n}_t\tilde{\lambda}(r)dr}\mathfrak{F}_{V}\mathfrak{F}_{X'}\big[\tau_{X_1' - (X^n-t)(V-V_1)}\tau_{V_1}F-\tau_{X_2' - (X^n-t)(V-V_2)}\tau_{V_2}F\big]\\
&\quad \times \left[e^{-\int^{X^n}_t|\eta + (X^n-r)\xi|^{2s}\tilde{\sigma}(r)dr} \right]dtd\xi dV\\
&=\int \int^{X^n}_0e^{-\int^{X^n}_t\tilde{\lambda}(r)dr}\big(\tau_{X_1' - (X^n-t)(V-V_1)}\tau_{V_1}F-\tau_{X_2' - (X^n-t)(V-V_2)}\tau_{V_2}F\big)\\
&\quad \times \mathfrak{F}_{X'}^{-1}\mathfrak{F}_{V}^{-1}\left[e^{-\int^{X^n}_t|\eta + (X^n-r)\xi|^{2s}\tilde{\sigma}(r)dr} \right]dtdX' dV.
\end{aligned}
$$
Therefore,
$$
\begin{aligned}
\big|U(X'_1,X^n,V_1)-U(X'_2,X^n,V_2) \big| &\leq \text{Lip}(F)\int \int^{X^n}_0|(X'_1-X'_2 - (X^n-t)(V_1-V_2),V_1-V_2)| \\
&\quad \times e^{-\int^{X^n}_t\tilde{\lambda}(r)dr}\Big|\mathfrak{F}_{X',V}^{-1}\left[e^{-\int^{X^n}_t|\eta + (X^n-r)\xi|^{2s}\tilde{\sigma}(r)dr} \right]\Big|dtdX' dV\\
&\leq C\Lip(F)\big(|X_1'-X_2'|+|V_1-V_2|\big),
\end{aligned}
$$
for $C>0$, a uniform upper bound of
$$
\int \int^{X^n}_0\big(1+(X^n-t)\big)e^{-\int^{X^n}_t\tilde{\lambda}(r)dr}\Big|\mathfrak{F}_{X',V}^{-1}\left[e^{-\int^{X^n}_t|\eta + (X^n-r)\xi|^{2s}\tilde{\sigma}(r)dr} \right]\Big|dtdX' dV<+\infty.
$$
\end{proof}

\subsection{Integrability properties of solutions}\label{subsec:integ_fFpb}

We state some integrability properties for the fundamental solution associated to the fractional Fermi equation that will be used in the following sections. The fractional Fermi equation \eqref{fFPB} is a slight generalization of the fractional Kolmogorov equation (see \S2.4 in \cite{IS})
\begin{equation}\label{Kol_eq}
f_t + v\cdot\nabla_xf + (-\Delta_v)^sf = h,\quad x,v\in\RR^{d-1}.
\end{equation}
Lemma \ref{lemma:solution_fFPB} provides us with a fundamental solution (solving \eqref{fFPB} with sources $F=0$ and $G=\delta(X)\delta(V)$) which takes the self-similar form
\begin{equation}\label{self_similar_J}
J(X,V) = c_{d-1}\frac{1}{(X^d)^{d-1+\frac{d-1}{s}}}\mathfrak{J}\left(\frac{X'}{(X^d)^{1+\frac{1}{2s}}},\frac{V}{(X^d)^{\frac{1}{2s}}}\right)\exp\left(-\int_0^{X^d}\tilde{\lambda}(\tau)d\tau\right),
\end{equation}
for some appropriate constant $c_{d-1}>0$, and with $\mathfrak{J}$ defined via its Fourier transform:
$$
\hat{\mathfrak{J}}(\xi,\eta;X^d) := \exp\left(-\int^1_0|\eta - \tau\xi|^{2s}\tilde{\sigma}\Big(X^d(1-\tau)\Big)d\tau\right).
$$
In what remains of this section, we abbreviate the exponential factor in the definition of $J$ by writing $\Lambda(X^d):= \exp\left(-\int_0^{X^d}\tilde{\lambda}(\tau)d\tau\right)$.
The solution to the initial value problem \eqref{fFPB} takes the form
\beq\label{sol_fFPB}
\begin{aligned}
U(X,V) &= \int_{\RR^{d-1}}\int_{\RR^{d-1}} G(Y',W)J(X'-Y'-X^dW,X^d,V-W)dWdY'\\
&\quad + \int_0^{X^d}\int_{\RR^{d-1}}\int_{\RR^{d-1}} F(Y,W)J(X'-Y'-(X^d-Y^d)W,X^d-Y^d,V-W)dWdY.
\end{aligned}
\eeq
The next proposition is the analogous to \cite[Proposition 2.1]{IS} for the fractional Fermi equation and we state it without proof.
\begin{proposition}
For $J$ and $\mathfrak{J}$ as above we have:
\begin{itemize}
\item[(1)] The function $\mathfrak{J}$ is $C^\infty$ and decays polynomially at infinity. Moreover, $\mathfrak{J}$ and all its derivatives are integrable in $\RR^{2(d-1)}$.
\item[(2)] For every $X^d>0$, $\int_{\RR^{2(d-1)}}J(X,V)dX'dV=\Lambda(X^d)$.
\item[(3)] $J\geq 0$ and $\mathfrak{J}\geq 0$.
\item[(4)] For any $X^d>0,$
\beqq
\begin{aligned}
\|J(\cdot,X^d,\cdot)\|_{L^1(\RR^{2(d-1)})} &= \Lambda(X^d)\|\mathfrak{J}\|_{L^1(\RR^{2(d-1)})}\\
\|(-\Delta_V)^{s/2}J(\cdot,X^d,\cdot)\|_{L^1(\RR^{2(d-1)})} &= (X^d)^{-1/2}\Lambda(X^d)\|(-\Delta_V)^{s/2}\mathfrak{J}\|_{L^1(\RR^{2(d-1)})}.
\end{aligned}
\eeqq
\end{itemize}
\end{proposition}

We use the Riemann-Lebesgue theorem to deduce explicit estimates for the decay at infinity of the kernel $\mathfrak{J}$. Indeed, for any pair of multi-indices $\alpha,\beta$ such that $|\alpha|+|\beta|=d-1$, one verifies that
\beqq
\Lambda(X^d)^{-1}|\partial^\alpha_{\xi}\partial^\beta_\eta\hat{\mathfrak{J}}(\xi,\eta;X^d)|\lesssim\int_0^1\frac{d\rho}{|\eta+\rho\xi|^{d-1-2s}}e^{-\sigma_0\int^1_0|\eta+\tau\xi|^{2s}d\tau},
\eeqq
which is integrable since Fubini's theorem implies the integrability of $(\rho,\eta,\xi)\mapsto \frac{e^{-\sigma_0\int^1_0|\eta+\tau\xi|^{2s}d\tau}}{|\eta+\rho\xi|^{d-1-2s}}$.
This yields (thanks to the Riemann-Lebesgue theorem) that $X'^\alpha V^\beta\mathfrak{J}(X',\cdot,V)\in C_0(\RR^{2(d-1)})$ and consequently, for some $\gamma>0$,
\beqq
\begin{aligned}
\Lambda(X^d)^{-1}\mathfrak{J}(X',X^d,V)&\lesssim \frac{1}{(1+|X'|^2+|V'|^2)^{\frac{d-1}{2}+\gamma}}\\
\end{aligned}
\eeqq
uniformly in $X^d>0$. 
By averaging $\mathfrak{J}$ with respect to either $X'$ or $V$ one deduces that $\gamma\leq s$. This follows from the fact 
\beqq
\begin{aligned}
&\mathfrak{F}_{V\to\eta}\left(\int\mathfrak{J}(X,V)dX'\right) = \exp\left(-\sigma_1(X^d)|\eta|^{2s}\right)\Lambda(X^d)\quad\text{and}\\
&\mathfrak{F}_{X'\to\xi}\left(\int\mathfrak{J}(X,V)dV\right) = \exp\left(-\sigma_2(X^d)|\xi|^{2s}\right)\Lambda(X^d),
\end{aligned}
\eeqq
where $\sigma_1(X^d):=\int^1_0\tilde{\sigma}(\tau X^d)d\tau$ and $\sigma_2(X^d) := \int^1_0\tau^{2s}\tilde{\sigma}(\tau X^d)d\tau$. Therefore, for some $C,C'>0$,
\beq\label{decay_avg_J}
\begin{aligned}
&\frac{C^{-1}\Lambda(X^d)}{\left(1+|V|^2\right)^{\frac{1}{2}(d-1 +2s)}}\leq \int\mathfrak{J}(X,V)dX' \leq\frac{C\Lambda(X^d)}{\left(1+|V|^2\right)^{\frac{1}{2}(d-1 +2s)}};\\
&\frac{C'^{-1}\Lambda(X^d)}{\left(1+|X'|^2\right)^{\frac{1}{2}(d-1 +2s)}}\leq \int\mathfrak{J}(X,V)dV' \leq\frac{C'\Lambda(X^d)}{\left(1+|X|^2\right)^{\frac{1}{2}(d-1 +2s)}}.
\end{aligned}
\eeq

In a similar fashion we can obtain upper bounds for the decay of the derivatives of $\mathfrak{J}$. For any $\alpha',\beta'$ multi-indices we see that
\beqq
\Lambda(X^d)^{-1}|\partial^\alpha_{\xi}\partial^\beta_\eta\big(\xi^{\alpha'}\eta^{\beta'}\hat{\mathfrak{J}}(\xi,\eta;X^d)\big)|\lesssim\int_0^1\frac{d\rho}{|\eta+\rho\xi|^{|\alpha|+|\beta|-|\alpha'|-|\beta'|-2s}}e^{-\sigma_0\int^t_0|\eta+\tau\xi|^{2s}d\tau},
\eeqq
uniformly in $X^d>0$, and the right hand side is still integrable for $|\alpha|+|\beta|=d-1+|\alpha'|+|\beta'|$ which yields the decay
\beq\label{decay_der_J}
\big|\partial_{X'}^{\alpha'}\partial_V^{\beta'}\mathfrak{J}(X,V)\big|\lesssim  \frac{\Lambda(X^d)}{(1+|X'|^2+|V|^2)^{\frac{d-1 +|\alpha'|+|\beta'|}{2}+\gamma'}},\quad\forall X^d>0,
\eeq
for some $0<\gamma'\leq s$.
We then obtain the next.
\begin{lemma}\label{lemma:int_der_J} For any nonnegative real numbers $m,n$, such that $m+n\leq |\alpha|+|\beta|$, with $\alpha,\beta$ multi-indices, 
\beqq
\begin{aligned}
\||X'|^m|V|^n\partial^\alpha_{X'}\partial^\beta_{V}J(X,V)\|_{L^1(\RR^{2(d-1)})} &= O((X^d)^{\left(1+\frac{1}{2s}\right)(m-|\alpha|) + \frac{1}{2s}(n-|\beta|)}\Lambda(X^d)).
\end{aligned}
\eeqq
\end{lemma}
\begin{proof}
We use the self-similar form of $J$ to obtain
\begin{equation*}
\begin{aligned}
&\int_{\RR^{2(d-1)}}|X'|^m|V|^n |\partial^\alpha_{X'}\partial^\beta_{V}J(X,V) |dX'dV\\
&=\frac{c_{d-1}\Lambda(X^d)}{(X^d)^{d-1+\frac{d-1}{s}}}\int_{\RR^{2(d-1)}}|X'|^m|V|^n \Big|\partial^\alpha_{X'}\partial^\beta_{V}\left(\mathfrak{J}\left(\frac{X'}{(X^d)^{1+\frac{1}{2s}}},\frac{V}{(X^d)^{\frac{1}{2s}}}\right)\right)\Big|dX'dV\\
&=\frac{c_{d-1}\Lambda(X^d)}{(X^d)^{d-1+\frac{d-1}{s}+|\alpha|\left(1+\frac{1}{2s}\right)+|\beta|\frac{1}{2s}}}\int_{\RR^{2(d-1)}}|X'|^m|V|^n \Big|(\partial^\alpha_{X'}\partial^\beta_{V}\mathfrak{J})\left(\frac{X'}{(X^d)^{1+\frac{1}{2s}}},\frac{V}{(X^d)^{\frac{1}{2s}}}\right)\Big|dX'dV\\
\end{aligned}
\end{equation*}
By means of a change of variable and the decay estimate \eqref{decay_der_J} we conclude the proof by noticing that
\beqq
\begin{aligned}
&\int_{\RR^{2(d-1)}}|X'|^m|V|^n |\partial^\alpha_{X'}\partial^\beta_{V}J(X,V) |dX'dV\\
&=\frac{c_{d-1}(X^d)^{(d-1)\left(1+\frac{1}{s}\right) + m\left(1+\frac{1}{2s}\right) + n\frac{1}{2s}}}{(X^d)^{d-1+\frac{d-1}{s}+|\alpha|\left(1+\frac{1}{2s}\right)+|\beta|\frac{1}{2s}}}\Lambda(X^d)\int_{\RR^{2(d-1)}}|X'|^m|V|^n \Big|\left(\partial^\alpha_{X'}\partial^\beta_{V}\mathfrak{J}\right)\left(X,V\right)\Big|dX'dV\\
&\lesssim (X^d)^{(m-|\alpha|)\left(1+\frac{1}{2s}\right) + (n-|\beta|)\frac{1}{2s}}\Lambda(X^d).
\end{aligned}
\eeqq
\end{proof}
\begin{remark}\label{rmk:sharp_decay_J}
We can improve the above estimates for the cases with no differentiation involved. Using the sharp decay estimates for the averages of $\mathfrak{J}$ in \eqref{decay_avg_J}, we obtain that for any $m<2s$,
\beqq
\begin{aligned}
&\||X'|^m\mathfrak{J}(X,V)\|_{L^1(\RR^{2(d-1)})}+\||V|^m\mathfrak{J}(X,V)\|_{L^1(\RR^{2(d-1)})}<\infty,\quad\forall X^d>0.
\end{aligned}
\eeqq
This implies
\beqq
\begin{aligned}
&\||X'|^mJ(X,V)\|_{L^1(\RR^{2(d-1)})} = O((X^d)^{\left(1+\frac{1}{2s}\right)m}\Lambda(X^d)), &&\text{for }m<2s;\\
&\||V|^nJ(X,V)\|_{L^1(\RR^{2(d-1)})} = O((X^d)^{\frac{1}{2s}n}\Lambda(X^d)),&& \text{for }n<2s.
\end{aligned}
\eeqq
\end{remark}
\begin{remark}
If instead of considering the whole half-space $\RR^d_+\times\RR^{d-1}$ as domain of integration one restricts the estimates to a compact region, it is possible to obtain estimates for exponents $m\geq 2s$ and $n\geq 2s$. However the appearing constants will grow with the (transversal, i.e. in $X',V$) size of the domain, logarithmically for $m\text{ (or $n$) }=2s$ and polynomially for $m\text{ (or $n$) }>2s$.
\end{remark}
We will need the following decay estimates for the fractional Laplacian of $J$.
\begin{lemma}\label{lemma:int_fLap_J}  For any nonnegative real numbers $m,n$, such that $m,n<2s$, 
\beqq
\begin{aligned}
&\||X'|^m(-\Delta_V)^{s}J(X,V)\|_{L^1(\RR^{2(d-1)})} = O((X^d)^{\left(1+\frac{1}{2s}\right)m-1}\Lambda(X^d));\\
&\||V|^n(-\Delta_V)^{s}J(X,V)\|_{L^1(\RR^{2(d-1)})} = O((X^d)^{\frac{1}{2s}n -1}\Lambda(X^d)).
\end{aligned}
\eeqq
\end{lemma}
\begin{remark}
The constants in the estimates blow up as $m$ or $n$ approaches $2s$.
\end{remark}
\begin{proof}
We do the estimation involving powers of $|V|$, the other case is simpler and follows similarly.

Recall the following singular integral definition of the fractional Laplacian:
$$
(-\Delta_V)^{s}J = \frac{c_{d,s}}{2}\int_{\RR^{d-1}}\frac{2J(V)-J(V+z)-J(V-z)}{|z|^{d-1+2s}}dz.
$$
Then,
\begin{equation*}
\begin{aligned}
&\int_{\RR^{2(d-1)}} |V|^n |(-\Delta_V)^{s}J(X,V) |dX'dV\\
&\lesssim \int_{\RR^{2(d-1)}}|V|^n \int_{\RR^{d-1}}\frac{\left| 2J(V)-J(V+z)-J(V-z)\right|}{|z|^{d-1+2s}}dzdX'dV\\
&\lesssim (X^d)^{n/2s -1}\Lambda(X^d)\int_{\RR^{2(d-1)}}|V|^n \int_{\RR^{d-1}}\frac{\left|2\mathfrak{J}(V)-\mathfrak{J}(V+z)-\mathfrak{J}(V-z)\right|}{|z|^{d-1+2s}}dzdX'dV\\
\end{aligned}
\end{equation*}
We split the integral with respect to $z$ into two integrals for the respective regions $|z|<1$ and $|z|>1$. For the former, we see that
\beqq
\begin{aligned}
&\int_{\RR^{2(d-1)}}|V|^n \int_{|z|<1}\frac{\left|2\mathfrak{J}(V)-\mathfrak{J}(V+z)-\mathfrak{J}(V-z)\right|}{|z|^{d-1+2s}}dzdX'dV\\
&\lesssim \int_{\RR^{d-1}} \frac{1}{|z|^{d-1-2(1-s)}}\int^1_0\int_{\RR^{2(d-1)}}|V|^n|\nabla_V^2\mathfrak{J}(X,V+rz)|dX'dVdrdz\\
&\lesssim \int_{|z|<1} \frac{1}{|z|^{d-1-2(1-s)}}\int^1_0\int_{\RR^{2(d-1)}}|V-rz|^n|\nabla_V^2\mathfrak{J}(X,V)|dX'dVdrdz\\
&\lesssim \left(\int_{|z|<1} \frac{dz}{|z|^{d-1-2(1-s)}}\right)\int_{\RR^{2(d-1)}}(|V|^n+1)|\nabla_V^2\mathfrak{J}(X,V)|dX'dV <\infty,
\end{aligned}
\eeqq
where the finiteness of the last integral with respect to $(X',V)$ follows from the previous lemma. 

On the other hand, for $|z|>1$ and using Remark \ref{rmk:sharp_decay_J} we easily obtain that
\beqq
\begin{aligned}
&\int_{\RR^{2(d-1)}}|V|^n \int_{|z|>1}\frac{\left|2\mathfrak{J}(V)-\mathfrak{J}(V+z)-\mathfrak{J}(V-z)\right|}{|z|^{d-1+2s}}dzdX'dV\\
&\lesssim \left(\int_{|z|>1} \frac{dz}{|z|^{d-1+2s-n}}\right)\int_{\RR^{2(d-1)}}|V|^n\mathfrak{J}(X,V)dX'dV <\infty.
\end{aligned}
\eeqq
\end{proof}
\begin{remark}\label{rmk:gral_integ}
Similar estimates hold for more general solution to \eqref{fFPB} by means of the representation formula \eqref{sol_fFPB}. We indeed use this in the proof of Lemma \ref{lemma:ord1_approx}.
\end{remark}

\section{Approximation analysis for narrow beams}\label{sec:Fpb_approx}
\subsection{Fractional Fermi pencil-beam approximation}\label{subsec:fFpb_approx}
The next lemma is a crucial step in the proof of the pencil-beam approximation result.
\begin{lemma}\label{lemma:approx_fFpb}
Let $\psi\in C_c(Q)\cap \Lip_{\kappa}(Q)$ so that the backward fractional Fokker-Planck system,
\begin{equation}\label{backward_fFP}
- \theta\cdot\nabla_x\varphi + \lambda\varphi = \mathcal{I}_\theta(\varphi) + \psi.
\end{equation}
has a unique continuous strong solution $\varphi$. Let $U$ be the solution to the fractional Fermi pencil-beam system \eqref{fFPB} with 
$$F=0,\quad G=\delta(X')\delta(V),\quad \tilde{\sigma}=\sigma(0 ,X^d)\quad\text{and}\quad \tilde{\lambda}=\lambda(0 ,X^d),$$ 
and extended by zero to $X^d<0$.  (i.e. for $X^n>0$, $U$ coincides with the fundamental solution $J$ defined in the previous section). In dimension $d=2$ we take $s'\in(2s-1,s)$ and for $d\geq 3$ we choose $s'\in(0,s)$. Then, there exists $C_{s'}>0$ so that the rescaling
$$\mathfrak{u}(x,\theta) := (2\epsilon)^{-2(d-1)}U((2\epsilon)^{-1}x',x^d,\epsilon^{-1}\mathcal{S}(\theta))$$
satisfies
$$
\left|\int_Q \mathfrak{u}(x,\theta)\psi(x,\theta)dxd\theta -\varphi(0,N)\right|\leq C_{s'}\epsilon^{2s'}\kappa^{s'},
$$
where $N = (0,\dots,0,1)\in\SS^{d-1}$ and $C_{s'}\to\infty$ as $s'\to s$.
\end{lemma}
\begin{proof}
We define $\Psi$ and $\Phi$ as the following rescaling of $\psi$ and $\varphi$ respectively:
\beq\label{rescaling_test_functions}
\varphi(x,\theta) = \Phi((2\epsilon)^{-1}x',x^d,\epsilon^{-1}\mathcal{S}(\theta))\quad\text{and}\quad \psi(x,\theta) = \Psi((2\epsilon)^{-1}x',x^d,\epsilon^{-1}\mathcal{S}(\theta)),
\eeq
thus $\|\Phi\|_\infty=\|\varphi\|_\infty\lesssim \|\psi\|_\infty=\|\Psi\|_\infty$. 
Then,
$$
\begin{aligned}
\int_{Q}\mathfrak{u}\psi dxd\theta &=  \int_{Q}\mathfrak{u}\left(-\theta \cdot\nabla_x\varphi+ \lambda \varphi -\cI_\theta(\varphi)\right)dxd\theta \\
&=  \int_{Q}\mathfrak{u}\left(-\theta\cdot\nabla_x \varphi +\lambda \varphi +\textstyle\epsilon^{2s}\sigma(x)\Big((-\Delta_\theta)^s\varphi - c_{d,s}\varphi\Big)\right)dxd\theta\\
&=:I_1+I_2+I_3 + I_4.
\end{aligned}
$$
The proof consists in showing that
$$
I_1+I_2+I_3 + I_4 = \Phi(0,0,0) + \int \underbrace{\mathcal{P}(U)}_{=0}\frac{\Phi}{\leVr^{2(d-1)}}dXdV + O(\epsilon^{2s'}\kappa^{s'}),
$$
where we recall $\mathcal{P}$ is the fractional Fermi pencil-beam operator defined in \eqref{fFpb}. \\

\noindent{\it Estimation of $I_1$.} The advection component $I_1=\textstyle -\int_Q\fu\big(\theta\cdot\nabla_x\varphi\big)dxd\theta$ is computed as follows. A change of variables gives us
$$
\begin{aligned}
I_1&= - \int_{\RR^{d}_+\times\RR^{d-1}} \Big[\Big(\frac{V\cdot \nabla_{X'} \Phi}{1+\epsilon^2|V|^2}\Big)U + \Big(\frac{1-\epsilon^2|V|^2}{1+\epsilon^2|V|^2}\Big)(\partial_{X^d}\Phi)U \Big]\frac{dXdV}{\leVr^{2(d-1)}},
\end{aligned}
$$
where the integration by parts formula from Lemma \ref{lemma:IbyP} and the decay of $U$ at infinity lead to
$$
\begin{aligned}
I_1= &
 \int_{\RR^{d}_+\times\RR^{d-1}} \Big[\Big(\frac{V\cdot \nabla_{X'} U}{1+\epsilon^2|V|^2}\Big)\Phi + \Big(\frac{1-\epsilon^2|V|^2}{1+\epsilon^2|V|^2}\Big)(\partial_{X^d}U)\Phi \Big]\frac{dXdV}{\leVr^{2(d-1)}}\\
&+\Phi(0,0,0)\\
=& \int_{\RR^{d}_+\times\RR^{d-1}}(V\cdot\nabla_{X'}U+\partial_{X^d}U)\frac{\Phi}{\leVr^{2(d-1)}} dXdV +\Phi(0,0,0)+E_1(\Phi).
\end{aligned}
$$
The error term is given by
$$
\begin{aligned}
E_1 &= -2\epsilon^2\int_{\RR^{d}_+\times\RR^{d-1}} {\textstyle\frac{|V|^2}{\leVr^{2d}}}\big(V\cdot \nabla_{X'}U +\partial_{X^d}U\big)\Phi dXdV + \epsilon^2\int_{\RR^{d}_+\times\RR^{d-1}} \textstyle\frac{|V|^2}{\leVr^{2d}}\big(V\cdot \nabla_{X'}U\big)\Phi dXdV.
\end{aligned}
$$

The next simple inequality is used extensively in subsequent estimations in order to reduce the powers of $|V|$ and obtain integrability:
\beq\label{s'_ineq}
\frac{\epsilon^2|V|^2}{\leVr^{2m}}\leq \epsilon^{2s'}|V|^{2s'}\frac{(\epsilon^2|V|^2)^{1-s'}}{(1+\epsilon^2|V|^2)^{1-s'}}\leq \epsilon^{2s'}|V|^{2s'},\quad\text{for any }s'\in(0,s)\text{ and }m\geq1-s'.
\eeq
Then,
$$
\begin{aligned}
|E_1| &\lesssim \epsilon^{2s'}\int_0^\infty\big(\||V|^{2s'}(-\Delta_V)^{s}U\|_{L^1(\RR^{2(d-1)}}\\
&\hspace{5em}+\||V|^{2s'} V\cdot \nabla_{X'}U\|_{L^1(\RR^{2(d-1)}}+\||V|^{2s'}U\|_{L^1(\RR^{2(d-1)}}\big)dX^d,
\end{aligned}
$$
where the integrals on the right hand side are finite according to Lemmas \ref{lemma:int_der_J} and \ref{lemma:int_fLap_J}.\\

\noindent{\it Estimation of $I_2$.} Similarly, for $I_2 = \textstyle \int_{Q}\lambda\mathfrak{u}\varphi dxd\theta$ we obtain
\begin{equation}\label{defJ1}
\begin{aligned}
I_2=& \int_{\RR^d_+\times\RR^{d-1}}\tilde{\lambda} U\frac{\Phi}{\leVr^{2(d-1)}} dXdV
+\underbrace{\int_{\RR^d_+\times\RR^{d-1}} \big(\lambda(\epsilon X',X^d) - \lambda(0 ,X^d)\big)U\frac{\Phi}{\leVr^{2(d-1)}} dXdV}_{=:J_1(\Phi)}\\
=&  \int_{\RR^d_+\times\RR^{d-1}}\tilde{\lambda} U\frac{\Phi}{\leVr^{2(d-1)}} dXdV  
+ J_1(\Phi),
\end{aligned}
\end{equation}
Let us skip for a moment the estimation of $J_1(\Phi)$ and move on to $I_3$ and $I_4$.\\ 

\noindent{\it Estimation of $I_4$.} We see that
$$
I_4 = - {\textstyle\epsilon^{2s}\sigma(x)c_{d,s}}\int_{Q}\fu\varphi dxd\theta=- {\textstyle\epsilon^{2s}\sigma(x)c_{d,s}}\int_{Q}U\frac{\Phi}{\leVr^{2(d-1)}}dXdV,
$$
thus, from the the explicit dependence on $\epsilon$ we easily verify that $|I_4|\lesssim \epsilon^{2s}\|\psi\|_\infty$. 

\noindent{\it Estimation of $I_3$.} Recall $I_3= \textstyle\epsilon^{2s}\sigma(x)\int_{Q}\fu(-\Delta_\theta)^s\varphi dxd\theta$. We drop for a moment the dependence in $X'$ and by abusing notation write $\sigma$ instead of $\sigma(2\epsilon X',X^d)$. We have
$$
\begin{aligned}
I_3 = & \int_{\RR^d_+\times\RR^{d-1}}\int_{\RR^{d-1}}\frac{U(V)}{\leVr^{d-1-2s}}\Big(\frac{\Phi(V)}{\leVr^{d-1-2s}}-\frac{\Phi(V')}{\leVpr^{d-1-2s}}\Big)\frac{\sigma\;dV'dXdV}{2^{2s}|V-V'|^{d-1+2s}}\\
= & \int_{\RR^d_+\times\RR^{d-1}}\int_{\RR^{d-1}}\frac{\Phi(V)}{\leVr^{d-1-2s}}\Big(\frac{U(V)}{\leVr^{d-1-2s}}-\frac{U(V')}{\leVpr^{d-1-2s}}\Big)\frac{\sigma\;dV'dXdV}{2^{2s}|V-V'|^{d-1+2s}}\\
=& \frac{1}{2^{1+2s}} \int_{\RR^d_+\times\RR^{d-1}}\int_{\RR^{d-1}}\frac{\Phi(V)}{\leVr^{d-1-2s}}\Big(\frac{2U(V)}{\leVr^{d-1-2s}}-\frac{U(V+z)}{\langle \epsilon(V+z)\rangle^{d-1-2s}}-\frac{U(V-z)}{\langle \epsilon(V-z)\rangle^{d-1-2s}}\Big)\frac{\sigma\;dzdXdV}{|z|^{d-1+2s}}\\
=& \frac{1}{2^{1+2s}} \int_{\RR^d_+\times\RR^{d-1}}\int_{\RR^{d-1}}\frac{\Phi(V)}{\leVr^{2(d-1-2s)}}\frac{\big(2U(V)-U(V+z)-U(X,V-z)\big)}{|z|^{d-1+2s}}\sigma dzdXdV\\
 &+\frac{1}{2^{1+2s}} \int_{\RR^d_+\times\RR^{d-1}}\int_{\RR^{d-1}}\frac{\Phi(V)}{\leVr^{d-1-2s}}\left(\frac{1}{\leVr^{d-1-2s}}-\frac{1}{\langle \epsilon(V+z)\rangle^{d-1-2s}}\right)\frac{U(V+z)}{|z|^{d-1+2s}}\sigma\;dzdXdV\\
  &+\frac{1}{2^{1+2s}} \int_{\RR^d_+\times\RR^{d-1}}\int_{\RR^{d-1}}\frac{\Phi(V)}{\leVr^{d-1-2s}}\left(\frac{1}{\leVr^{d-1-2s}}-\frac{1}{\langle \epsilon(V-z)\rangle^{d-1-2s}}\right)\frac{U(V-z)}{|z|^{d-1+2s}}\sigma\;dzdXdV\\
  =:& I_{3,1}+I_{3,2}+I_{3,3}.
\end{aligned}
$$
The first integral gives
\begin{equation}\label{I_31}
I_{3,1} = \int \tilde{\sigma}(X^d)\frac{\Phi}{\leVr^{2(d-1)}}(-\Delta_V)^sU dXdV + E_{3}(\Phi) + J_2(\Phi),
\end{equation}
with error terms
$$
\begin{aligned}
E_3 =& \int_{\RR^d_+\times\RR^{d-1}}\frac{\Phi}{\leVr^{2(d-1)}} (-\Delta_V)^sU\left(\leVr^{4s}-1\right)\frac{1}{2^{2s}}\sigma(2\epsilon X',X^d)\;dXdV,
\end{aligned}
$$
and
\begin{equation}\label{defJ2}
\begin{aligned}
J_2 =& \frac{1}{2^{2s}}\int_{\RR^d_+\times\RR^{2(d-1)}}\big(\sigma(2\epsilon X',X^d)-\sigma(0,X^d)\big)\frac{\Phi}{\leVr^{2(d-1)}}(-\Delta_V)^sU dX dV.
\end{aligned}
\end{equation}
On the other hand,
$$
\begin{aligned}
I_{3,2}+I_{3,3}&= \frac{1}{2^{1+2s}} \int_{\RR^d_+\times\RR^{d-1}}\int_{\RR^{d-1}}\sigma(2\epsilon X',X^d)\frac{\Phi(X,V)}{\leVr^{d-1-2s}}\\
&\hspace{3em} \times\left(\frac{2}{\leVr^{d-1-2s}}-\frac{1}{\langle \epsilon(V+z)\rangle^{d-1-2s}}-\frac{1}{\langle \epsilon(V-z)\rangle^{d-1-2s}}\right)\frac{U(V+z)}{|z|^{d-1+2s}}\;dzdXdV\\
&\quad +\frac{1}{2^{1+2s}} \int_{\RR^d_+\times\RR^{d-1}}\int_{\RR^{d-1}}\sigma(2\epsilon X',X^d)\frac{\Phi(X,V)}{\leVr^{d-1-2s}}\\
&\hspace{3em}\times \left(\frac{1}{\leVr^{d-1-2s}}-\frac{1}{\langle \epsilon(V-z)\rangle^{d-1-2s}}\right)\frac{(U(V-z) - U(V+z))}{|z|^{d-1+2s}}\;dzdXdV.
\end{aligned}
$$
By denoting $h(z) = \frac{1}{\langle \epsilon(V+z)\rangle^{d-1-2s}}$, the above simplifies to
$$
\begin{aligned}
&I_{3,2}+I_{3,3} \\
&= \frac{1}{2^{1+2s}} \int_{\RR^d_+\times\RR^{d-1}}\int_{\RR^{d-1}}\sigma(2\epsilon X',X^d)\frac{\Phi(X,V)}{\leVr^{d-1-2s}} \frac{\left(2h(0)-h(z)-h(-z)\right)}{|z|^{d-1+2s}}U(V+z)\;dzdXdV\\
&\quad +\frac{1}{2^{1+2s}} \int_{\RR^d_+\times\RR^{d-1}}\int_{\RR^{d-1}}\sigma(2\epsilon X',X^d)\frac{\Phi(X,V)}{\leVr^{d-1-2s}}\left(h(0)-h(-z)\right)\frac{(U(V-z) - U(V+z))}{|z|^{d-1+2s}}\;dzdXdV.
\end{aligned}
$$
To continue with the estimation we split the each integrals into two, for the respective regions $|z|>1$ and $|z|\leq 1$. For the former we use that
$$
\nabla_zh(z)=\frac{(d-1-2s)\epsilon^2(V+z)}{\langle \epsilon(V+z)\rangle^{d+1-2s}}$$
therefore
\begin{equation*}\label{der1_h}
|h(0)-h(z)|\leq C\epsilon^2|z|(|V\pm z|+|z|), \quad\forall |z|>1,\, V\in\RR^{d-1},
\end{equation*}
which directly implies
$$
|2h(0)-h(z)-h(-z)|\leq|h(0)-h(z)|+|h(0)-h(-z)|\leq  C\epsilon^2|z|(|V+z|+|z|).
$$
Since $|h|\leq1$ we then have that for any $s'\in(0,s)$,
$$|h(0)-h(z)|^{s'+(1-s')}\leq C|h(0)-h(z)|^{s'}\leq C\epsilon^{2s'}(|V\pm z|^{2s'}+|z|^{2s'}),$$
and similarly for $|2h(0)-h(z)-h(-z)|$. 
The above gives us that on the region $|z|>1$,
$$
|I_{3,2}+I_{3,3}|\leq C\epsilon^{2s'}\left(\textstyle \int_{|z|>1}\frac{1}{|z|^{d-1+2(s-s')}}dz\right)\|\sigma\|_\infty\|\Phi\|_{\infty}\|(1+|V|)U\|_{L^1}\leq C\epsilon^{2s'}\|\psi\|_\infty.
$$
For the second part, $|z|\leq 1$, we compute
$$
\nabla_z^2h(z) = \frac{(d-1-2s)\epsilon^2 I}{\langle \epsilon(V+z)\rangle^{d+1-2s}} + \frac{(d-1-2s)(d+1-2s)\epsilon^4(V+z)^t(V+z)}{\langle \epsilon(V+z)\rangle^{d+3-2s}},
$$
thus, from the inequality $\frac{\epsilon^2|V+z|^2}{\langle \epsilon(V+z)\rangle^{d+3-2s}}\leq 1$ it follows
\begin{equation}\label{der2_h}
|2h(0)-h(z)-h(-z)|\leq C\epsilon^2|z|^2,\quad\forall |z|\leq 1,\, V\in\RR^{d-1},
\end{equation}
and this leads to the estimate
$$
\begin{aligned}
& \int_{\RR^d_+\times\RR^{d-1}}\int_{\RR^{d-1}}\sigma(2\epsilon X',X^d)\frac{\Phi(X,V)}{\leVr^{d-1-2s}} \frac{\left(2h(0)-h(z)-h(-z)\right)}{|z|^{d-1+2s}}U(V+z)\;dzdXdV\\
&\quad \leq C\epsilon^2\|\sigma\|_\infty\|\Phi\|_\infty\|U\|_{L^1}\left(\textstyle \int_{|z|<1}\frac{1}{|z|^{d-1 - 2(1-s)}}dz\right)\leq C\epsilon^2\|\psi\|_\infty.
\end{aligned}
$$
On the other hand, we notice that
$$
h(0)-h(-z) = \int^1_0z\cdot\nabla h(tz)dt = (d-1-2s)\epsilon^2 \int^1_0\frac{z\cdot(V-tz)}{\langle \epsilon(V-tz)\rangle^{d+1-2s}}dt.
$$
Thus
$$
\begin{aligned}
&\int_{\RR^d_+\times\RR^{d-1}}\int_{\RR^{d-1}}\sigma(2\epsilon X',X^d)\frac{\Phi(X,V)}{\leVr^{d-1-2s}}\left(h(0)-h(-z)\right)\frac{(U(V-z) - U(V+z))}{|z|^{d-1+2s}}\;dzdXdV\\
&\quad = (d-1-2s)\epsilon^2 \int\int^1_0\sigma(2\epsilon X',X^d)\frac{\Phi(X,V)}{\leVr^{d-1-2s}}\frac{z\cdot(V-tz)}{\langle \epsilon(V-tz)\rangle^{d+1-2s}}\frac{(U(V-z) - U(V+z))}{|z|^{d-1+2s}}\;dtdzdXdV\\
&\quad = C\epsilon^2 \int\int^1_0\frac{\sigma(2\epsilon X',X^d)\Phi(X,V)}{\leVr^{d-1-2s}\langle \epsilon(V-tz)\rangle^{d+1-2s}}\frac{z\cdot(U(V-z)V - U(V+z)V)}{|z|^{d-1+2s}}\;dtdzdXdV+ O(\epsilon^2),
\end{aligned}
$$
where the error terms is given by
$$
\begin{aligned}
&C\epsilon^2 \int\frac{\sigma(2\epsilon X',X^d)\Phi(X,V)}{\leVr^{d-1-2s}}\left(\int^1_0\frac{tdt}{\langle \epsilon(V-tz)\rangle^{d+1-2s}}\right)\frac{(U(V-z) - U(V+z))}{|z|^{d-1-2(1-s)}}\;dzdXdV \\
&\hspace{3em}\leq C\epsilon^2\|\sigma\|_\infty\|\Phi\|_\infty\|U\|_{L^1}\Big(\textstyle\int_{|z|<1}\frac{1}{|z|^{d-1-2(1-s)}}dz\Big).
\end{aligned}
$$
For the remaining integral we see that
$$
z\cdot(U(V-z)V - U(V+z)V) = \big(U(V-z)z\cdot(V-z) - U(V+z)z\cdot(V+z)\big) + |z|^2\big(U(V-z)+U(V+z)\big).
$$
Thus, for any $\hat{z} = z/|z|$, with $0<|z|<1$, by defining the function
$$
g_{\hat{z}}(V) = (\hat{z}\cdot V)U(V),
$$
we have that 
$$
\begin{aligned}
&\int\int^1_0\frac{\sigma(2\epsilon X',X^d)\Phi(X,V)}{\leVr^{d-1-2s}\langle \epsilon(V-tz)\rangle^{d+1-2s}}\frac{z\cdot(U(V-z)V - U(V+z)V)}{|z|^{d-1+2s}}\;dtdzdXdV\\
&\hspace{3em} = \int\int^1_0\frac{\sigma(2\epsilon X',X^d)\Phi(X,V)}{\leVr^{d-1-2s}\langle \epsilon(V-tz)\rangle^{d+1-2s}}\left(\frac{g_{\hat{z}}(V-z) - g_{\hat{z}}(V+z)}{|z|}\right)\frac{1}{|z|^{d-1-2(1-s)}}\;dtdzdXdV\\
&\hspace{4em} + \int\int^1_0\frac{\sigma(2\epsilon X',X^d)\Phi(X,V)}{\leVr^{d-1-2s}\langle \epsilon(V-tz)\rangle^{d+1-2s}}\left(U(V-z)+U(V+z)\right)\frac{1}{|z|^{d-1-2(1-s)}}\;dtdzdXdV,
\end{aligned}
$$
and hence, denoting by $D^{z}_Vg_{\hat{z}}:=\frac{g_{\hat{z}}(V+z) - g_{\hat{z}}(V)}{|z|}$ the difference quotient of $g_{\hat{z}}(V)$, the integral is bounded by
$$
\begin{aligned}
& C\|\sigma\|_\infty\|\Phi\|_\infty\left(\int \left(\int |D^{z}_Vg_{\hat{z}}(V)| dVdX\right)\frac{1}{|z|^{d-1-2(1-s)}}dz + \|U\|_{L^1}\Big(\textstyle\int_{|z|<1}\frac{1}{|z|^{d-1-2(1-s)}}dz\Big)\right)\\
&\quad \leq C\|\sigma\|_\infty\|\Phi\|_\infty\left(\sup_{\hat{w}\in \SS^{d-2}}\Big(\sup_{|z|<1}\|D^{z}_Vg_{\hat{w}}(V)\|_{L^1}\Big) + \|U\|_{L^1}\right)\Big(\textstyle\int_{|z|<1}\frac{1}{|z|^{d-1-2(1-s)}}dz\Big)\\
&\quad\leq C\|\sigma\|_\infty\|\Phi\|_\infty\left(\||V|\nabla_V U\|_{L^1} + \|U\|_{L^1}\right),
 \end{aligned}
$$
where the last inequality follows from \cite[\S 5.8.2]{Evans}, while the boundedness of the $L^1$-norms of $U$ are due to Lemma \ref{lemma:int_der_J}. This concludes the proof that $|I_{3,2} + I_{3,3}|\leq C\epsilon^{2s'}\|\psi\|_\infty$. 

We still need to estimate the error terms of $I_{3,1}$ (defined in \eqref{I_31}). 
By simply noticing that $|\leVr^{4s}-1|\lesssim \epsilon^2|V|^2\leVr^{2(2s-1)}$, then
$$
|E_3|\lesssim \epsilon^2\|\sigma\|_\infty \int_{\RR^d_+\times\RR^{d-1}}\frac{|\Phi||V|^2}{\leVr^{2(d-2s)}} |(-\Delta_V)^sU|\;dXdV.
$$
Therefore, we get an $O(\epsilon^{2s'})$ upper bound by means of the inequality \eqref{s'_ineq} with $m=d-2s$ 
and the integrability Lemma \ref{lemma:int_fLap_J}. 
Notice that for \eqref{s'_ineq} to hold we need $d-2s\geq1-s'$, or equivalently $s'\geq2s+1-d$. For $d=2$ this translates into $s'\geq 2s-1$, while for $d\geq 3$ it is always satisfied. This explains the hypothesis imposed in the statement of the theorem.

We deduce $|E_3|\lesssim\epsilon^{2s'}\|\psi\|_\infty$ with constant depending on $\int_{\RR_+}\||V|^{2s'}(-\Delta_V)^sU\|_{L^1(\RR^{2(d-1)})}dX^d$.\\

\noindent{\it Estimation of $J_1$ and $J_2$.}  Let us now estimate $J_1$ and $J_2$ (defined respectively in \eqref{defJ1} and \eqref{defJ2}). 
We write $J_1$ and $J_2$ as follows,
\beqq
\begin{aligned}
J_1 =&  \frac{1}{2}\int_{\RR^d_+\times\RR^{d-1}} \big(\lambda(\epsilon X',X^d) - \lambda(0 ,X^d)\big)U(X,V)\Phi(X,V)dXdV + O(\epsilon^{2s'}).
\end{aligned}
\eeqq
with a reminder depending on $\||V|^{2s'}U(X,V)\|_{L^1}$, 
and
$$
\begin{aligned}
J_2 =&  \frac{1}{2^{2s}}\int_{\RR^d_+\times\RR^{d-1}}\big(\sigma(2\epsilon X',X^d)-\sigma(0,X^d)\big)\Phi(X,V)\big[(-\Delta_V)^sU(X,V)\big]dV dX + O(\epsilon^{2s'}),
\end{aligned}
$$
with a reminder depending on $ \||V|^{2s'}(-\Delta_V)^sU(X,V)\|_{L^1}$. To proceed we need the next lemma whose proof is postponed until the end of this section.

\begin{lemma}[Sub-optimal approximation for the adjoint equation] \label{lemma:ord1_approx}
Let $\varphi(x,\theta) = \Phi((2\epsilon)^{-1}x',x^d,\epsilon^{-1}\mathcal{S}(\theta))$ and $\psi(x,\theta) = \Psi((2\epsilon)^{-1}x',x^d,\epsilon^{-1}\mathcal{S}(\theta))$ with $\varphi$ solution to the backward fractional Fokker--Planck equation
$$-\theta\cdot\nabla_x\varphi + \lambda\varphi = \mathcal{I}_\theta[\varphi] + \psi,\quad \psi \in C_c(Q).$$
Let $W$ be the unique solution to the backward fractional Fermi pencil-beam equation \eqref{b_fFpb} with source $\Psi$. Then, for any $f$ absolutely integrable function in $\RR^d_+\times\RR^{d-1}$ and for any $s'\in (0,s)$, there exists $C>0$ (depending on $\sigma,\lambda$ and $\|f\|_{L^1}$) such that
\begin{equation*}
\Big|\int_{\RR^d_+\times\RR^{d-1}} f\Phi dXdV - \int_{\RR^d_+\times\RR^{d-1}} f WdXdV \Big|\leq C\epsilon^{\min\{1,2s'\}}\|\Psi\|_\infty.
\end{equation*}
with $C=C_{s'}$ blowing up to infinity as $s'\to s$ for $s<1/2$, and independent of $s'$ otherwise.
\end{lemma}

In view of the previous lemma we get
\beqq
\begin{aligned}
J_1 =&  \frac{1}{2}\int_{\RR^d_+\times\RR^{d-1}} \big(\lambda(\epsilon X',X^d) - \lambda(0 ,X^d)\big) W(X,V)U(X,V)dXdV + O(\epsilon^{2s'}),
\end{aligned}
\eeqq
and
\beqq
\begin{aligned}
J_2 =& \frac{1}{2^{2s}} \int_{\RR^d_+\times\RR^{d-1}}\big(\sigma(2\epsilon X',X^d)-\sigma(0,X^d)\big)W(X,V)\big[(-\Delta_V)^sU(X,V)\big]dV dX + O(\epsilon^{2s'}),
\end{aligned}
\eeqq
where we used the Lipschitz property of $\lambda$ and $\sigma$ to obtain the $O(\epsilon^2)$-error when interchanging $\Phi$ with $W$. In order to utilize the Lipschitz continuity of $W$ we employ Lemma \ref{lemma:radial_sym}. We first split the computations into two cases. For $s\leq1/2$  we do not need Lemma \ref{lemma:radial_sym} and simply obtain
\beqq
\begin{aligned}
|J_1| \lesssim& \; \epsilon^{2s'}\|\lambda\|_{C^1}^{2s'}\|\lambda\|_\infty^{1-2s'}\|\psi\|_\infty^{1-2s'}\||X|^{2s'}U(X,V)\|_{L^1} + C\epsilon^{2s'},
\end{aligned}
\eeqq
for some $s'\in(0,s)$. Otherwise, for $s\in(1/2,1)$, the lemma allows us to write
\beqq
\begin{aligned}
J_1 =&  \frac{\epsilon}{2}\int_{\RR^d_+\times\RR^{d-1}} \nabla_{X'}\lambda(0,X^n)\cdot X' \left(W(X,V) - W(0,V)\right)U(X,V)dXdV + O(\epsilon^{2s'}).
\end{aligned}
\eeqq
from which we deduce
\beqq
|J_1|\lesssim \epsilon^{1+r}\|\psi\|_\infty^{1-r}\kappa^{r}\|\lambda\|_{C^1}\||X'|^{1+r}U(X,V)\|_{L^1} + C\epsilon^{2s'},
\eeqq
for $r=2s'-1$ and $s'\in(1/2,s)$. In both cases, we end up with an estimate of the form $|J_1|\lesssim \epsilon^{2s'}\kappa^{s'}$.

Regarding $J_2$, Lemma \ref{lemma:radial_sym} yields
\beqq
\begin{aligned}
J_2 =& \frac{1}{2^{2s}}\int_{\RR^d_+\times\RR^{d-1}}\big(\sigma(2\epsilon X',X^d)-\sigma(0,X^d)\big)\\
&\hspace{5em}\times\big(W(X,V) - W(X,0)\big)\big[(-\Delta_V)^sU(X,V)\big]dV dX + O(\epsilon^{2s'}).
\end{aligned}
\eeqq
Similarly as above, we reduce the exponent by interpolating upper bounds and obtain 
\beqq
\begin{aligned}
|J_2| \leq  \frac{1}{2^{2s}}\|\sigma\|_{\infty}^{1-s'}\|\psi\|_{\infty}^{1-s'}& \int_{\RR^d_+\times\RR^{d-1}}\big|\sigma(2\epsilon X',X^d)-\sigma(0,X^d)\big|^{s'}\\
&\times\big|W(X,V) - W(X,0)\big|^{s'}\big|(-\Delta_V)^sU(X,V)\big|dV dX + O(\epsilon^{2s'}),
\end{aligned}
\eeqq
which then gives rise to
$$
\begin{aligned}
|J_2| &\leq C\epsilon^{2s'}\|\sigma\|_{\infty}^{1-s'}\|\sigma\|^{s'}_{C^1}\|\psi\|_{\infty}^{1-s'}\kappa^{s'} \int_{\RR^d_+\times\RR^{d-1}}|X'|^{s'}|V|^{s'}\big|(-\Delta_V)^sU(X,V)\big|dV dX + O(\epsilon^{2s'}),
\end{aligned}
$$
and consequently $|J_2| \leq C\epsilon^{2s'}\kappa^{s'}$. The positive constant depends on 
\beqq
\begin{aligned}
\||X'|^{2s'}(-\Delta_V)^sU(X,V)\|_{L^1}\quad\text{and}\quad  \||V|^{2s'}(-\Delta_V)^sU(X,V)\|_{L^1}.
\end{aligned}
\eeqq
which are finite thanks to Lemma \ref{lemma:int_fLap_J}.
\end{proof}

\begin{theorem}\label{thm:fFP_vs_pb}
Let $\delta\lesssim \epsilon^2\kappa$ and $f\in L^\infty_{x,\theta}\cap L^1_{x,\theta}$ be a compactly supported $\delta$-approximation to the identity $\delta_{0}(x)\delta_{N}(\theta)$, this is $\left|\int f\varphi dxd\theta - \varphi(0,N)\right|\lesssim \delta$ for all $\varphi\in C(Q)$. We let $u(x,\theta)$ be the solution to the fractional Fokker--Planck equation \eqref{fFP_diff} (or equivalently \eqref{fFP_id}-\eqref{fFP_kernel}). 
For any $s'\in(2s-1,s)$ in dimension $d=2$, and $s'\in(0,s)$ for $d\geq 3$, there exists  a constant $C=C(s',d,\lambda,\sigma)>0$, blowing up as $s'\to s$, such that
$$
\mathcal{W}^1_{\kappa}(\mathfrak{u},u)\leq C\epsilon^{2s'}\kappa^{s'}.
$$
\end{theorem}
\begin{proof}
For an arbitrary $\psi\in C_c(Q)\cap \Lip_{\kappa}(Q)$ let $\varphi$ be the unique solution to \eqref{backward_fFP}. We see that
\begin{equation*}
\begin{aligned}
\left| \int (\mathfrak{u}-u)\psi dxd\theta \right|&=\left| \int \mathfrak{u}\psi dxd\theta - \int f\varphi dxd\theta\right|\\
&\leq \left| \int \mathfrak{u}\psi dxd\theta -\varphi(0,N)\right|+\left|\varphi(0,N) - \int f\varphi dxd\theta\right|.
\end{aligned}
\end{equation*}
The proof follows from the hypothesis on the source $f$ and Lemma \ref{lemma:approx_fFpb}.
\end{proof}

\begin{proof}[Proof of Lemma \ref{lemma:ord1_approx}]
We let $\bfU$ be the solution to the inhomogeneous fractional Fermi pencil-beam equation $\mathcal{P}(\bfU) = f$ with boundary conditions $\bfU|_{X^n=0}=0$, and let $\bfu$ be its rescaling 
\beqq
\bfu(x,\theta) = (2\epsilon)^{-2(d-1)}\bfU((2\epsilon)^{-1}x',x^d,\epsilon^{-1}\mathcal{S}(\theta)).
\eeqq
One verifies that
$$
\int_{\mathcal{Q}_+} f WdXdV = \int_{\mathcal{Q}_+} \bfU \Psi dXdV,
$$
and hence a change variables leads to
$$
\int_{\mathcal{Q}_+} f WdXdV = \int_{Q_+} \bfu \psi dxd\theta + R(\epsilon^{2s'}),
$$
with the reminder term given explicitly by 
$$
R(\epsilon^{2s'})=\int_{\mathcal{Q}} \bfU\Psi\left(1-\leVr^{-(d-1)}\right)dXdV \leq \epsilon^{2s'}\big(\|\bfU\|_{L^1(\RR^d\times B_1)} + \||V|^{2s'}\bfU\|\big)_{L^1(\RR^d\times B_1^c)}\|\Psi\|_\infty.
$$
The proof then reduces to showing that
\begin{equation*}
\Big|\int_{Q_+} \bfu\psi dxd\theta - \int_{\mathcal{Q}_+} f \Phi dXdV \Big|\leq C\epsilon^{\min\{1,2s'\}}\|\Psi\|_\infty,
\end{equation*}
for any $s'\in(0,1)$, with the constant depending on this choice.

We proceed as in the proof of Lemma \ref{lemma:approx_fFpb} using the following decomposition
$$
\begin{aligned}
\int_{Q_+}\bfu\psi dxd\theta &=  \int_{\mathcal{Q}_+}\bfu\left(-\theta \cdot\nabla_x\varphi+ \lambda \varphi -\cI_\theta(\varphi)\right)dxd\theta \\
&=:I_1+I_2+I_3 + I_4,
\end{aligned}
$$
where the term $I_j$ are defined as in Lemma \ref{lemma:approx_fFpb} with $U$ replaced by $\bfU$. The objective is to show that
\begin{equation}\label{sum_Ij}
\begin{aligned}
I_1+I_2+I_3 + I_4 &= \int_{\RR^{d-1}\times\RR^{d-1}}\underbrace{\bfU(X',0,V)}_0\frac{\Phi(X',0,V)}{\leVr^{2(d-1)}}dX'dV\\
&\hspace{10em} + \int_{\mathcal{Q}_+} \underbrace{\mathcal{P}(\bfU)}_{=f}\Phi dXdV + O(\epsilon^{\min\{1,2s'\}}).
\end{aligned}
\end{equation}
Most of the estimates performed in the proof of Lemma \ref{lemma:approx_fFpb} remain identical, although with upper bounds now depending on various integrals of $\bfU$ (which are guaranteed to be finite by virtue of Lemma \ref{lemma:int_der_J} and the convolution formula for pencil-beam solutions \eqref{sol_fFPB}). These are subsequently bounded by the $L^1$-norm of $f$ (see \S\ref{subsec:integ_fFpb} and Remark \ref{rmk:gral_integ}). 
The main difference is that we are interested here in obtaining a sub-optimal accuracy for the error, namely $O(\epsilon^{\min\{1,2s'\}})$, so that Lipschitz-continuity of $\sigma$ and $\lambda$ is enough to deduce \eqref{sum_Ij}.
\end{proof}

\subsection{Ballistic approximation}\label{subsec:ballistic_approx}
We denote by $v$ the solution to the {\em ballistic transport equation}
\beq\label{ballistic_eq}
\theta\cdot\nabla_x v +\lambda(x) v = f(x,\theta),
\eeq
which, for a source $f\in L^1(\RR^{d}\times\SS^{d-1})$, has the explicit form
\beq\label{ballistic_sol}
v(x,\theta) =\mathcal{L}[f](x,\theta):= \int^{\infty}_0 e^{-\int^t_0\lambda(x-s\theta)ds}f(x-t\theta,\theta)dt.
\eeq
For a point source $f = \delta_0(x)\delta_{N}(\theta)$ the solution is defined in the distributional sense as
\beq\label{ballistic_sol_dist}
\langle v,\psi\rangle = \int^\infty_0e^{-\int^t_0\lambda(\vec{0},s)ds}\psi(tN,N)dt,\quad \forall \psi \in C(\RR^{d}\times\SS^{d-1}),
\eeq
where $\vec{0}$ stands for the origin in $\RR^{d-1}$. 

We now compare the order of approximation between the ballistic solution and the pencil-beam approximation introduced in the previous section. We obtain a lower and upper bound for their 1-Wasserstein distance for different choices of parameters.

\begin{theorem}\label{thm:ball_vs_pb}
There exists a constant $c>0$ (independent of $\epsilon$ and $\kappa$) such that
$$
c\epsilon\kappa\leq \mathcal{W}^1_{\kappa}(v,\mathfrak{u})\quad \text{for}\quad1\leq\kappa\leq\epsilon^{-1}\quad\text{and} \quad c\leq \mathcal{W}^1_{\kappa}(v,\mathfrak{u})\quad  \text{for}\quad \kappa>\epsilon^{-1}.
$$
Moreover, there exist constants $C_{s'}>0$ and $C_{s}>0$, with the former blowing up to infinity as $s'\in(0,s)$ approaches $s$, and the latter blowing up to infinity as $s \searrow\frac{1}{2}$, such that 
\beqq
\mathcal{W}^1_{\kappa}(v,\mathfrak{u})\leq C_{s'}(\epsilon\kappa)^{2s'}\quad \text{if}\quad s\in(0,1/2]\quad\text{and}\quad \mathcal{W}^1_{\kappa}(v,\mathfrak{u})\leq C_s\epsilon\kappa\quad \text{if}\quad s\in(1/2,1).
\eeqq
\end{theorem}
\begin{proof}
\noindent{\em Upper bound.} Let $\psi$ be an arbitrary test function in $BL_{1,\kappa}(\RR^{d}\times\SS^{d-1})$ and $\Psi$ its rescaling according to \eqref{rescaling_test_functions}, which is Lipschitz with constant $O(\epsilon\kappa)$. For $s'\in(0,s)$ we have that
\beqq
\begin{aligned}
\langle v-\mathfrak{u},\psi\rangle &= \int^\infty_0e^{-\int^t_0\lambda(\vec{0},s)ds}\Psi(\vec{0},t,\vec{0})dt -\int_0^\infty\int_{\RR^{2(d-1)}}U(X',t,V)\Psi(X',t,V) \frac{dX'dVdt}{\leVr^{2(d-1)}} \\
&= \int^\infty_0\left(e^{-\int^t_0\lambda(\vec{0},s)ds}\Psi(\vec{0},t,\vec{0}) -\int_0^\infty\int_{\RR^{2(d-1)}}U(X',t,V)\Psi(X',t,V) dX'dV\right)dt + O(\epsilon^{2s'}).
\end{aligned}
\eeqq
The expression inside parentheses is equivalent to
\beqq
\int_{\RR^{2(d-1)}}J(X',t,V)\left(\Psi(\vec{0},t,\vec{0}) -
\Psi(X',t,V) \right)dX'dV,
\eeqq
with $J$ the fundamental solution to the fractional Fermi pencil-beam of Section \ref{subsec:integ_fFpb}. 
There are two scenarios depending on the values of the exponent $s$. For $s>1/2$, we take $s'=1/2$ above and get
\beqq
\langle v-\mathfrak{u},\psi\rangle\lesssim \epsilon\kappa\int^\infty_0 e^{-\int^t_0\lambda(\vec{0},r)dr}\||(X',V)|J(X',t,V)\|_{L^1(\RR^{2(d-1)})}dt=O(\epsilon\kappa).
\eeqq
On the other hand, for $s\leq1/2$, we interpolate the Lipschitz and the $L^\infty$ bound of the difference $\Psi(\vec{0},t,\vec{0}) -
\Psi(X',t,V)$ to get
\beqq
\langle v-\mathfrak{u},\psi\rangle\lesssim (\epsilon\kappa)^{2s'}\int^\infty_0 e^{-\int^t_0\lambda(\vec{0},r)dr}\||(X',V)|^{2s'}J(X',t,V)\|_{L^1(\RR^{2(d-1)})}dt=O((\epsilon\kappa)^{2s'}).
\eeqq
for any $s'\in(0,s)$ with the constant in the estimate blowing up as $s'$ approaches $s$.

The desired estimates follow after taking supremum with respect to $\psi$.\\

\noindent{\em Lower bound.} Let us choose a specific test function: $\psi(x,\theta)=\psi_1(x')= e^{-\kappa|x'|}$.
When passing to stretched coordinate we denote $\Psi_1(X') = e^{-2\epsilon\kappa|X'|}$ 
and therefore $\Psi(X,V)=\Psi_1(X')$. The Fourier Transform of $\Psi_1$ is given by 
$$\hat{\Psi}_1(\xi):=\mathfrak{F}_{X'\to\xi}[\Psi_1](\xi) = c_{d}\frac{\epsilon\kappa}{\left(4\pi^2\epsilon^2\kappa^2+|\xi|^2\right)^{\frac{d}{2}}},$$
for some constant $c_d>0$. 
We find that
\beq\label{ball_test}
\begin{aligned}
\langle v,\psi\rangle &= \int^\infty_0e^{-\int^t_0\lambda(\vec{0},s)ds}\Psi(\vec{0},t,\vec{0})dt =  \int^\infty_0e^{-\int^t_0\lambda(\vec{0},s)ds}\int_{\RR^{d-1}}\hat{\Psi}_1(\xi)d\xi dt\\
&=c_d\epsilon\kappa\int^\infty_0e^{-\int^t_0\lambda(\vec{0},s)ds}\int_{\RR^{d-1}}\frac{1}{\left(4\pi^2\epsilon^2\kappa^2+|\xi|^2\right)^{\frac{d}{2}}}d\xi dt.
\end{aligned}
\eeq
On the other hand,
$$
\begin{aligned}
\langle \mathfrak{u},\psi\rangle &= \int_0^\infty\int_{\RR^{2(d-1)}}U(X',t,V)\Psi(X',t,V) \frac{dX'dVdt}{\leVr^{2(d-1)}} \\
&\leq \int_0^\infty\int_{\RR^{2(d-1)}}U(X',t,V)\Psi_1(X') dX'dVdt. 
\end{aligned}
$$
Then, the integration with respect to $V$ and the Plancherel's identity yield
\beq\label{pb_test}
\begin{aligned}
\langle \mathfrak{u},\psi\rangle&\leq \int_0^\infty\int_{\RR^{d-1}}\hat{U}(\xi,t,0)\hat{\Psi}_1(\xi,t) d\xi dt\\
&\leq c_d\epsilon\kappa\int_0^\infty e^{-\int^t_0\lambda(\vec{0},s)ds}\int_{\RR^{d-1}}\frac{e^{-\tau(t)|\xi|^{2s}}}{\left(4\pi^2\epsilon^2\kappa^2+|\xi|^2\right)^{\frac{d}{2}}}d\xi dt,
\end{aligned}
\eeq
where $\tau(t) := \int^t_0r^{2s}\tilde{\sigma}(r)dr$ ($\approx t^{2s+1}$). 
Let us compare this to \eqref{ball_test}. We get
$$
\begin{aligned}
\langle v-\mathfrak{u},\psi\rangle \geq c_d\epsilon\kappa \int_0^\infty e^{-\int^t_0\lambda(\vec{0},s)ds}\int_{\RR^{d-1}}\frac{\big(1-e^{-\tau(t)|\xi|^{2s}}\big)}{\left(4\pi^2\epsilon^2\kappa^2+|\xi|^2\right)^{\frac{d}{2}}}d\xi dt > 0.
\end{aligned}
$$
To simplify notation, we consider a generic constant $c>0$ independent of $\epsilon$ and $\kappa$ so that once passing to spherical coordinates we get
$$
\begin{aligned}
\langle v-\mathfrak{u},\psi\rangle &\geq c\epsilon\kappa\int_0^\infty e^{-\int^t_0\lambda(\vec{0},s)ds}\int_{0}^\infty \frac{\big(1-e^{-\tau(t)r^{2s}}\big)}{\left(4\pi^2\epsilon^2\kappa^2+r^2\right)^{\frac{d}{2}}}r^{d-2}drdt.
\end{aligned}
$$
We now analyze the integral $I(t) := \epsilon\kappa\int_{0}^\infty \frac{\big(1-e^{-\tau(t)r^{2s}}\big)}{\left(4\pi^2\epsilon^2\kappa^2+r^2\right)^{\frac{d}{2}}}r^{d-2}dr$. The term inside the parentheses is bounded from below by
$$
\begin{aligned}
1-e^{-\tau(t)r^{2s}}&=2s\tau(t)\int^r_0\rho^{2s-1}e^{-\tau(t)\rho^{2s}}d\rho \\
&\geq 2s\tau(t)e^{-\tau(t)r^{2s}}\int^r_0\rho^{2s-1}d\rho\\
&= \tau(t)e^{-\tau(t)r^{2s}}r^{2s}
\end{aligned}
$$
which leads to
$$
I(t) \geq \epsilon\kappa\tau(t)\int_{0}^\infty \frac{e^{-\tau(t)r^{2s}}r^{d-2+2s}}{\left(4\pi^2\epsilon^2\kappa^2+r^2\right)^{\frac{d}{2}}}dr.
$$
We then get
$$
I(t) \geq \epsilon\kappa\tau(t)e^{-\tau(t)}\int_{0}^1 \frac{r^{d-2(1-s)}}{\left(4\pi^2(\epsilon\kappa)^2+1\right)^{\frac{d}{2}}}dr\geq c\epsilon\kappa\left(\frac{1}{\left((2\pi\epsilon\kappa)^2+1\right)^{\frac{d}{2}}}\right)\tau(t)e^{-\tau(t)}.
$$
Consequently, under the resolution restriction $\kappa\leq \epsilon^{-1}$, we obtain the estimate
$$
\langle v-\mathfrak{u},\psi\rangle\geq c\epsilon\kappa\int^\infty_0\tau(t)e^{-\tau(t)}e^{-\int^t_0\lambda(\vec{0},r)dr}dt.
$$
Let us now consider the case $\kappa>\epsilon^{-1}$. We can use a different lower bound, namely
$$
\begin{aligned}
\epsilon\kappa\int_{0}^\infty \frac{e^{-\tau(t)r^{2s}}r^{d-2+2s}}{\left(4\pi^2\epsilon^2\kappa^2+r^2\right)^{\frac{d}{2}}}dr
&\geq \frac{\epsilon\kappa e^{-\tau(t)(\epsilon\kappa)^{2s}}}{(\epsilon\kappa)^{d}\left(4\pi^2+1\right)^{\frac{d}{2}}}\int_{0}^{\epsilon\kappa} r^{d-2+2s}dr\geq c(\epsilon\kappa)^{2s}e^{-\tau(t)(\epsilon\kappa)^{2s}},
\end{aligned}
$$
thus
$$
\langle v-\mathfrak{u},\psi\rangle\geq c(\epsilon\kappa)^{2s}\int^\infty_0\tau(t)e^{-\tau(t)(\epsilon\kappa)^{2s}}e^{-\int^t_0\lambda(\vec{0},s)ds}dt \gtrsim (\epsilon\kappa)^{2s}\int^\infty_0t^{2s+1}e^{-\alpha t^{2s+1}(\epsilon\kappa)^{2s}}e^{-\beta t}dt
$$
for some $\alpha,\beta>0$ depending on $\|\sigma\|_\infty$ and $\|\lambda\|_{\infty}$. Performing the change of variables $\rho = t(\epsilon\kappa)^{2s/(2s+1)}$, we get
\beqq
\begin{aligned}
(\epsilon\kappa)^{2s}\int^\infty_0t^{2s+1}e^{-\alpha t^{2s+1}(\epsilon\kappa)^{2s}}e^{-\beta t}dt &\geq (\epsilon\kappa)^{-\frac{2s}{2s+1}}\int^{\infty}_0\rho^{2s+1}e^{-\alpha \rho^{2s+1}-\beta\rho}d\rho.
\end{aligned}
\eeqq
which is a worse lower bound since $\epsilon\kappa>(\epsilon\kappa)^{-\frac{2s}{2s+1}}$ for all $s\in(0,1)$ when $\epsilon\kappa>1$. 

Since the requirement on the Lipschitz constant for $\psi$ is $\text{Lip}(\psi)\leq \kappa$ (with $\kappa>\epsilon^{-1}$), by taking $\psi_1(x) = e^{-\epsilon^{-1}|x'|}$ and repeating previous computations we arrive to the stronger lower bound $\langle v-\mathfrak{u},\psi\rangle = O(1).$
\end{proof}

\section{Approximation via superposition of pencil-beams}\label{sec:superposition}
Let $f\in L^1(Q)$ be such that 
\beqq
f\geq 0\quad\text{and}\quad\text{supp}(f)\Subset \RR^d\times\SS^{d-1}.
\eeqq
We denote by $\{U_{y,\eta}(X,V)\}_{y,\eta}$ a continuous family of pencil-beams indexed by the parameters $(y,\eta)$, so that $U_{y,\eta}$ is the solution to the fractional Fermi pencil-beam equation \eqref{fFpb} with initial condition $G=\delta(X')\delta(V - \epsilon^{-1}\mathcal{S}(\eta))$, extended by zero for $X^d<0$ (the source $F$ is always taken to be null), and for diffusion and absorption coefficients
\beqq
\tilde{\sigma}(X^d) = \sigma(y+X^d\eta)\quad\text{and}\quad \tilde{\lambda}(X^d) = \lambda(y+X^d\eta).
\eeqq
The construction of the approximation via pencil beams is based on the superposition of the $U_{y,\eta}$'s. We define their respective affine transformations as
\beqq
\mathfrak{u}(\cdot,\cdot ; y,\eta) := (2\epsilon)^{-2(d-1)}U_{y,\eta}\circ T^\epsilon_y,\quad (y,\eta)\in \supp(f),
\eeqq 
where each $T^\epsilon_{y,\eta}:\RR^d\times\SS^{d-1}\ni (x,\theta)\mapsto (X,V)\in \RR^d\times\RR^{d-1}$ encodes an affine transformation of the spatial coordinates plus a rescaled stereographic projection of the angular variables. They are given by
\beqq
T^\epsilon_{y,\eta}(x,\theta) := \left((2\epsilon)^{-1}\Pi_{\eta^\perp}(x-y), \eta\cdot(x-y),\epsilon^{-1}\mathcal{S}_{\eta}(\theta)\right),
\eeqq
where $\Pi_{\eta^\perp}x := (\text{Id}-\eta\eta^T)x$ is the orthogonal projection of $x$ onto the subspace $\eta^\perp$, while $S_{\eta}$ denotes the stereographic projection of the unit sphere that takes $\eta$ to the origin in $\RR^{d-1}$  (defined similarly as in Section \ref{sec:intro}).

The approximation via pencil-beams is then defined as
\beq\label{def:general_approx}
\mathfrak{u}(x,\theta) := \int_{Q}f(y,\eta)\mathfrak{u}(x,\theta;y,\eta)dyd\eta,
\eeq
which acts in the distributional sense according to the rule
\beqq
\langle\mathfrak{u},\phi\rangle := \int_{Q}\int_Qf(y,\eta)\mathfrak{u}(x,\theta;y,\eta)\phi(x,\theta)dxd\theta dyd\eta, \quad \phi\in C_0(Q).
\eeqq 

We give once again the statement of the approximations results for the superposition of beams, as formulated in Section \ref{subsec:main_results}. These results follows almost directly from the narrow beam case.
\begin{theorem*}[Theorem \ref{thm:cont_sup}]
The same conclusion of Theorem \ref{thm:narrow_beam} holds for a source $f$ satisfying
\beqq
f\in L^1(Q),\quad f\geq 0\quad\text{and}\quad\text{supp}(f)\Subset \RR^d\times\SS^{d-1}.
\eeqq
and $\mathfrak{u}$ given by the continuous superposition of pencil-beam in \eqref{def:general_approx}.
\end{theorem*}
\begin{proof}
Let $\psi$ be an arbitrary test function in $BL_{1,\kappa}(\RR^{d}\times\SS^{d-1})$ and $\varphi$ the solution to the backward Fokker--Planck system. We see that
\beqq
\begin{aligned}
\int_Q\psi(x,\theta)\left(u-\mathfrak{u}\right)dxd\theta &= \int_{Q}f(y,\eta)\varphi(y,\eta)dyd\eta\\
&\quad - \int_{Q}f(y,\eta)\left(\int_Q\mathfrak{u}(x,\theta;y,\eta)\psi(x,\theta)dxd\theta \right)dyd\eta.
\end{aligned}
\eeqq
Therefore, the estimate for the 1-Wasserstein distance between the Fokker--Planck solution and the superposition of pencil-beams follows by noticing that 
\beqq
\left|\int_Q\mathfrak{u}(x,\theta;y,\eta)\psi(x,\theta)dxd\theta - \varphi(y,\eta)\right|\lesssim \epsilon^{2s'}\kappa^{s'}
\eeqq
as a consequence of Lemma \ref{lemma:approx_fFpb}, and for a constant uniform with respect to $(y,\eta)$ but depending on $\|f\|_{L^1}$, $\|\lambda\|_{C^1}$ and $\|\sigma\|_{C^1}$.

Regarding the upper and lower bounds of the distance between the ballistic transport solution (i.e. $v(x,\theta)$) and the pencil-beam approximation, we again use the explicit expression of $v$ given in \eqref{ballistic_sol}, and see that
\beqq
\begin{aligned}
\int_Q\psi(x,\theta)\left(v-\mathfrak{u}\right)dxd\theta &= \int_Q\int^{\infty}_0 e^{-\int^t_0\lambda(y-s\eta)ds}f(y-t\eta,\eta)\psi(y,\eta)dtdyd\eta\\
&\quad - \int_{Q}f(y,\eta)\left(\int_Q\mathfrak{u}(x,\theta;y,\eta)\psi(x,\theta)dxd\theta \right)dyd\eta\\
&= \int_Q f(y,\eta)\int^{\infty}_0e^{-\int^t_0\lambda(y+s\eta)ds}\psi(y+t\eta,\eta)dtdyd\eta\\
&\quad - \int_{Q}f(y,\eta)\left(\int_Q\mathfrak{u}(x,\theta;y,\eta)\psi(x,\theta)dxd\theta \right)dyd\eta.
\end{aligned}
\eeqq
Therefore, the upper and lower bounds arise from the analysis of the difference
\beq\label{diff_lb}
\int^{\infty}_0e^{-\int^t_0\lambda(y+s\eta)ds}\psi(y+t\eta,\eta)dt - \int_Q\mathfrak{u}(x,\theta;y,\eta)\psi(x,\theta)dxd\theta.
\eeq
For each pair $(y,\eta)$, a coordinates rotation plus passing to stretched coordinates allow us to simplify the second integral and get
\beqq
\int_0^\infty\int_{\RR^{2(d-1)}}U_{y,\eta}(X',t,V)\Psi(X',t,V) \frac{dX'dVdt}{\leVr^{2(d-1)}}, 
\eeqq
for $\Psi$ defined by the relation $\varphi(x,\theta) = \Psi\circ T^\epsilon_{y,\eta}(x,\theta)$. 
Following the proof of Theorem \ref{thm:ball_vs_pb} it is not hard to realize that
\beqq
\int_0^\infty\int_{\RR^{2(d-1)}}U_{y,\eta}(X',t,V)\Psi(X',t,V) \frac{dX'dVdt}{\leVr^{2(d-1)}} =  \int^\infty_0e^{-\int^t_0\lambda(y+s\eta)ds}\Psi(\vec{0},t,\vec{0})dt  + O((\epsilon\kappa)^{2s'}),
\eeqq
where we take $s'\in(0,s)$ for $s\in(0,1/2]$ and $s'=1/2$ for $s\in(1/2,1)$, and with the remainder independent of the values of $y$ and $\eta$. By noticing that
\beqq
\Psi(\vec{0},t,\vec{0}) = \psi(y+t\eta,\eta),\footnotemark
\eeqq
\footnotetext{If $t=\eta\cdot(x-y)$, then $0 = (\text{Id} -\eta\eta^T)(x-y) = (x-y) - t\eta$, this is, $x=y+t\eta$.}
the upper bound follows directly.

The same computations performed in the proof of Theorem \ref{thm:ball_vs_pb} to obtain a lower bound can be carried out here, in order to bound from below the differences in \eqref{diff_lb} uniformly with respect to $y$ and $\eta$.
\end{proof}

We end this article by stating the analogous result in the case of a discrete superposition of pencil-beams. Given a spread out source $f$ as in the previous theorem, which we assume to be bounded, we can approximate it in the 1-Wasserstein sense as a finite sum of delta sources. Indeed, there exists a simple function $g = \sum_{i=1}^Ia_i\chi_{R_i}$, $a_i\geq 0$ and with $\chi_{R_i}$ the characteristic function for the open set $R_i\subset Q$, which without loss of generality can be taken of diameter $|R_i| \leq \epsilon^2$, so that $\int|f-g|d  xd\theta <\epsilon^2\kappa$. We set $\tilde{a}_i = a_i\text{meas}(R_i)$, with $\text{meas}(R_i)$ denoting the measure of $R_i$ with respect to $dxd\theta$. Then, 
\beqq
\mathcal{W}^1_\kappa(f,\tilde{f})\lesssim \epsilon^{2}\kappa,\quad\tilde{f}(x,\theta):= \sum_{i=1}^I \tilde{a}_i\delta_{x_i}(x)\delta_{\theta_i}(\theta),
\eeqq 
since for any $\psi\in BL_{1,\kappa}$,
\beqq
\begin{aligned}
\left|\int \psi(f-\tilde{f})dxd\theta\right|&\leq \left|\int \psi(f-g)dxd\theta\right|+\left|\int \psi(g-\tilde{f})dxd\theta\right|\\
&\leq \|\psi\|_\infty\epsilon^2\kappa + \sum^I_{i=1}a_i\left|\int_{R_i}\left(\psi(x,\theta) - \psi(x_i,\theta_i)\right)dxd\theta\right|\\
&\lesssim \epsilon^2\kappa + \sum_{i=1}^Ia_i\text{meas}(R_i)|R_i|\kappa = O(\epsilon^2\kappa).
\end{aligned}
\eeqq

Considering the same affine transformations defined previously, we set
\beq\label{def:disc_approx}
\mathfrak{u}(x,\theta) := \sum_{i=1}^I\tilde{a}_i\cdot\mathfrak{u}(x,\theta;x_i,\theta_i),
\eeq
a discrete superposition of pencil beams, where
\beqq
\mathfrak{u}(\cdot,\cdot ; x_i,\theta_i) := (2\epsilon)^{-2(d-1)}U_{x_i,\theta_i}\circ T^\epsilon_{x_i,\theta_i},
\eeqq 
with the $U_{x_i,\theta_i}$'s defined as in the continuous case. Following the same steps as in the previous proof, we deduce the following result:
\begin{theorem*}[Theorem \ref{thm:disc_sup}]
The same conclusion of the Theorem \ref{thm:narrow_beam} holds for a source $f$ satisfying
\beqq
f\in L^1(Q)\cap L^\infty(Q),\quad f\geq 0\quad\text{and}\quad\text{supp}(f)\Subset \RR^d\times\SS^{d-1}.
\eeqq
and $\mathfrak{u}$ given by the discrete superposition of pencil-beam in \eqref{def:disc_approx}.
\end{theorem*}
\bigskip

\section*{Acknowledgment}
This research was partially supported by the Office of Naval Research, Grant N00014-17-1-2096 and by the National Science Foundation, Grant DMS-1908736.
\bigskip

\begin{appendix}
\section{Proof of Theorem \ref{thm:wellposedness_sRTE}}\label{appdx:1}
Let $f\in L^2(\RR^d\times\SS^{d-1})$ and $u_\Omega \in W^2_\Omega$ the unique solution to
\beq\label{sRTE_bc}
\left\{\begin{array}{ll}
\theta\cdot\nabla_x u +\lambda u = \cI(u)+f,& \text{in }\Omega\times\SS^{d-1},\\
u=0,&\text{on }\Gamma_-,
\end{array}\right.
\eeq
which is non-negative whenever $f\geq 0$. Since the scattering kernel is integrable and assumed to be symmetric with respect to its angular entries, the existence and uniqueness of $u_\Omega$ follows from \cite[Chapter XXI-Theorem 4]{DL} provided $\lambda\geq \lambda_0>0$. The non-negativity property can be derived from the proof of Proposition 6 in the same reference. 
Moreover, by zero-continuing $u_{\Omega}$ to the whole space, and regardless of $\Omega$, we have the estimates 
$$\|u_{\Omega}\|_{L^2(\RR^{d}\times\SS^{n-1})}\leq \|f\|_{L^2(\RR^n\times\SS^{d-1})},$$
thus, the Banach-Alaoglu theorem gives us a limit function $u\in L^2(\RR^d\times\SS^{d-1})$ and a convergent sequence $\{u_{\Omega_{k}}\}_k$, such that $u_{\Omega_k}\to u$ weakly in $L^2$. Moreover, the limit satisfies the same inequality which then gives us uniqueness of solutions. It remains to show that $u\in W^2$ and it satisfies \eqref{sRTE}.

Without loss of generality, we assume $\{\Omega_k\}_k$ is an increasing sequence of open bounded subsets of $\RR^d$. Let $\Omega_{k_0}$ be an arbitrary element. For any $\varphi\in C_c^\infty(\RR^d\times\SS^{d-1})$ there exist $k_0>0$ so that $\varphi\in C_c^\infty(\Omega_{k_0}\times\SS^{d-1})$, and for $k> k_0$ we have that
$$
\begin{aligned}
&\int_{\RR^d\times\SS^{d-1}}\big((-\theta\cdot\nabla_x\varphi)u_{\Omega_k} +\lambda\varphi u_{\Omega_k} \big)dxd\theta\\
&\hspace{2em}= \int_{\RR^d\times\SS^{d-1}}\int_{\SS^{d-1}}k(x,\theta',\theta)(u_{\Omega_k}(x,\theta') - u_{\Omega_k}(x,\theta))\varphi(x,\theta)d\theta'dxd\theta + \int_{\RR^d\times\SS^{d-1}} f\varphi dxd\theta.
\end{aligned}
$$
Therefore, we can take the limit as $k\to\infty$ and obtain
$$
\begin{aligned}
&\int_{\RR^d\times\SS^{d-1}}\big((-\theta\cdot\nabla_x\varphi)u +\lambda\varphi u \big)dxd\theta\\
&\hspace{2em}= \int_{\RR^d\times\SS^{d-1}}\int_{\SS^{d-1}}k(x,\theta,\theta')(u(x,\theta) - u(x,\theta'))\varphi(x,\theta)d\theta'dxd\theta + \int_{\RR^d\times\SS^{d-1}} f\varphi dxd\theta.
\end{aligned}
$$
On the other hand, using \eqref{sRTE_bc} we see that $\|(\theta\cdot\nabla_xu_{\Omega_k})\mathds{1}_{\Omega_k}\|_{L^2(\RR^d\times\SS^{d-1})}\leq C_{\lambda,b}\|f\|_{L^2(\RR^d\times\SS^{d-1})}$, thus passing to a subsequence we deduce that $u\in W^2$ and $u$ is a strong solution of \eqref{sRTE}. 
We directly deduce that $u\geq 0$ for $f\geq 0$.

\section{Localization and mollification}\label{appdx:loc_moll}
We consider stereographic coordinates on the unit sphere and smooth and compactly supported mollification kernels $\xi(x)$ and $\eta(v)$. The localization and mollification technique is a standard procedure for local operators and we thus focus on the singular integral term in \eqref{fFP_id}. For a fixed pair $(y,v'')\in \RR^d\times\RR^{d-1}$, we consider the test function $\varphi(x,\theta)$ defined via stereographic coordinates as $[\varphi]_{\mathcal{S}}(x,v) = \chi(y,v'')\xi^\delta(y-x)\eta^\delta(v''-v)\lvr^{2(d-1)}$ with $\chi$ a smooth compactly supported function, and 
$$\xi^\delta(x) = \delta^{-d}\xi(\delta^{-d}x)\quad\text{and}\quad\eta^\delta(x) = \delta^{-(d-1)}\eta(\delta^{-(d-1)}x),$$
with $\delta>0$ small, and $\xi$ and $\eta$ smooth, nonnegative functions, with compact support around their respective origins, and such that $\int\xi(x) dx=\int\eta(v) dv=1$. 

In the weak formulation of Fokker--Planck \eqref{weak_sol}, the non-local part corresponds to the integral $\int\sigma(x)\mathcal{B}(u,\varphi)dx$, which in stereographic coordinates rewrites as 
\beqq
\begin{aligned}
&\int\sigma(x)\mathcal{B}(u,\varphi)dx\\
&=\frac{1}{2}\int\int \sigma(x)\int K(\theta',\theta)(u(x,\theta') - u(x,\theta))(\varphi(x,\theta')-\varphi(x,\theta))d\theta'd\theta dx\\ 
&=\int\int \sigma(x)(-\Delta_v)^{s/2}\left(\frac{u(x,v)}{\lvr^{d-1-2s}}\right)\cdot (-\Delta_v)^{s/2}\left(\frac{\varphi(x,v)}{\lvr^{d-1-2s}}\right)dvdx \\
&\quad - c\int\int \sigma(x)u(x,v)\varphi(x,v)\lvr^{-2(d-1)}dvdx.
\end{aligned}
\eeqq
We use the notation $w(x,v)=\frac{u(x,v)}{\lvr^{d-1-2s}}$. Then, for the test function introduced above, we have that
\beqq
\begin{aligned}
&\int\sigma(x)\mathcal{B}(u,\varphi)dx\\
&=\chi(y,v'')\int\int \sigma(x)\xi^\delta(y-x)((-\Delta_v)^{s/2}w)(x,v)\int \frac{\eta^\delta(v''-v)\lvr^{d-1+2s}-\eta^\delta(v''-v')\lvpr^{d-1+2s}}{|v-v'|^{d-1+s}}dv'dv dx \\
&\quad+ c\chi(y,v'')\int\int\sigma(x)u(x,v)\xi^\delta(y-x)\eta^\delta(v''-v)dxdv\\
&=\chi(y,v'')\langle v''\rangle^{d-1+2s}\int\int \sigma(x)\xi^\delta(y-x)((-\Delta_v)^{s/2}w)(x,v)((-\Delta_v)^{s/2}\eta )(v''-v)dv dx \\
&\quad- \chi(y,v'')\int\int \sigma(x)\xi^\delta(y-x)((-\Delta_v)^{s/2}w)(x,v)\\
&\quad\times \int \frac{\eta^\delta(v''-v)\left(\lvr^{d-1+2s} - \langle v''\rangle^{d-1+2s}\right)-\eta^\delta(v''-v')\left(\lvpr^{d-1+2s}-\langle v''\rangle^{d-1+2s}\right)}{|v-v'|^{d-1+s}}dv'dvdx\\
&\quad+ c(\chi\cdot(\xi^\delta\otimes\eta^\delta)*(\sigma u))(y,v'')\\
&=\sigma(y)\chi(y,v'')\langle v''\rangle^{d-1+2s}\int\int \xi^\delta(y-x)((-\Delta_v)^{s/2}w)(x,v)((-\Delta_v)^{s/2}\eta )(v''-v)dv dx \\
&\quad+\chi(y,v'')\langle v''\rangle^{d-1+2s}\int\int (\sigma(x)-\sigma(y))\xi^\delta(y-x)((-\Delta_v)^{s/2}w)(x,v)((-\Delta_v)^{s/2}\eta )(v''-v)dv dx \\
&\quad- \chi(y,v'')\int\int \sigma(x)\xi^\delta(y-x)((-\Delta_v)^{s/2}w)(x,v)\\
&\quad\times \int \frac{\eta^\delta(v''-v')\left(\lvpr^{d-1+2s} - \langle v''\rangle^{d-1+2s}\right)-\eta^\delta(v''-v)\left(\lvr^{d-1+2s}-\langle v''\rangle^{d-1+2s}\right)}{|v-v'|^{d-1+s}}dv'dvdx\\
&\quad+ c(\chi\cdot(\xi^\delta\otimes\eta^\delta)*(\sigma u))(y,v'').
\end{aligned}
\eeqq
Using the Fourier definition of the fractional (Euclidean) Laplacian, we see that
$$
(-\Delta_v)^{s}(f*g) = \mathfrak{F}^{-1}(|k|^{s}\mathfrak{F}(f)\cdot |k|^{s}\mathfrak{F}(g)) = \big((-\Delta_v)^{s}f\big)*\big((-\Delta_v)^{s}g\big).
$$
Thus, the first term on the right-hand side is equivalent to
\beqq
\sigma(y)\langle v''\rangle^{d-1+2s}\chi(y,v'')(-\Delta_v)^{s}\left(\xi^\delta\otimes\eta^\delta*w\right)(y,v'').
\eeqq
However, we would like to obtain an expression for $\lvr^{-(d-1-2s)}(\chi\cdot(\xi^\delta\otimes\eta^\delta*u))$, which involves computing the commutators between convolution and multiplication by $\lvr^{-(d-1-2s)}$, and between the fractional Laplacian and multiplication by $\chi$. We have
\beqq
\eta*(fg) = f(\eta*g) + \int\eta(v')g(v-v')(f(v-v')-f(v))dv'
\eeqq
and also that
\beqq
(-\Delta_v)^s(fg) = f(-\Delta_v)^s(g) + g(-\Delta_v)^s(f).
\eeqq
Then, denoting $\tilde{u}(x,v)=\chi(x,v)\cdot (\xi^\delta\otimes\eta^\delta*u
)(x,v)$ and $\tilde{w}(x,v) = \frac{\tilde{u}(x,v)}{\lvr^{d-1-2s}}$, we see that
\beqq
\begin{aligned}
&\chi(y,v'')(-\Delta_v)^{s}\left(\xi^\delta\otimes\eta^\delta*w\right)(y,v'')\\
&= ((-\Delta_v)^{s}\tilde{w})(y,v'') - \frac{(\xi^\delta\otimes\eta^\delta*u
)(y,v'')}{\langle v''\rangle^{d-1-2s}}((-\Delta_v)^{s}\chi)(y,v'')\\
&\quad - \int \frac{\left(\frac{(\xi^\delta\otimes\eta^\delta*u
)(y,v)}{\langle v\rangle^{d-1-2s}}-\frac{(\xi^\delta\otimes\eta^\delta*u
)(y,v'')}{\langle v''\rangle^{d-1-2s}}\right)\left(\chi(y,v)-\chi(y,v'')\right)}{|v-v''|^{d-1+2s}}dv\\
&\quad +\chi(y,v'')\left((-\Delta_v)^{s}\int\int \xi^\delta(y-x)\eta^\delta(v-v')u(x,v')\left(\frac{1}{\langle v'\rangle^{d-1-2s}} - \frac{1}{\lvr^{d-1-2s}}\right)dv'dx \right)(y,v'').
\end{aligned}
\eeqq
Bringing the above computations together, we conclude that
\beqq
\int\sigma(x)\mathcal{B}(u,\varphi)dx = \sigma(y)[(-\Delta_\theta)^s(\tilde{u})](y,v'') + h(y,v''),
\eeqq
with an error $h(y,v'')$ given by
\beqq
\begin{aligned}
&h(y,v'') \\
&= - \sigma(y)\langle v''\rangle^{d-1+2s}\frac{(\xi^\delta\otimes\eta^\delta*u
)(x,v'')}{\langle v''\rangle^{d-1-2s}}((-\Delta_v)^{s}\chi)(y,v'')\\
&\quad - \sigma(y)\langle v''\rangle^{d-1+2s}\int \frac{\left(\frac{(\xi^\delta\otimes\eta^\delta*u
)(x,v)}{\langle v\rangle^{d-1-2s}}-\frac{(\xi^\delta\otimes\eta^\delta*u
)(x,v'')}{\langle v''\rangle^{d-1-2s}}\right)\left(\chi(y,v)-\chi(y,v'')\right)}{|v-v''|^{d-1+2s}}dv\\
&\quad +\sigma(y)\langle v''\rangle^{d-1+2s}\chi(y,v'')\left((-\Delta_v)^{s}\int\int \xi^\delta(y-x)\eta^\delta(v-v')u(x,v')\zeta_1(v';v)dv'dx \right)(y,v'')\\
&\quad+\chi(y,v'')\langle v''\rangle^{d-1+2s}\int\int (\sigma(x)-\sigma(y))\xi^\delta(y-x)((-\Delta_v)^{s/2}w)(x,v)((-\Delta_v)^{s/2}\eta^\delta )(v''-v)dv dx \\
&\quad- \chi(y,v'')\int\int \sigma(x)\xi^\delta(y-x)((-\Delta_v)^{s/2}w)(x,v)\int \frac{\eta^\delta(v''-v')\zeta_2(v',v'')-\eta^\delta(v''-v)\zeta_2(v,v'')}{|v-v'|^{d-1+s}}dv'dvdx\\
&\quad+ c(\chi\cdot(\xi^\delta\otimes\eta^\delta)*(\sigma u))(y,v'')\\
&=:\sum_{i=1}^6I_i,
\end{aligned}
\eeqq
where $\zeta_1(v';v)=\frac{1}{\langle v'\rangle^{d-1-2s}} - \frac{1}{\lvr^{d-1-2s}}$ and $\zeta_2(v,v'') = \langle v\rangle^{d-1+2s}-\langle v''\rangle^{d-1+2s}$.

We will show that $h\in L^2(\RR^d\times\RR^{d-1};\lvr^{-2(d-1)}dxdv)$, with norm bounded by the sum of $\|u\|_{L^2}$ and $\|(-\Delta_\theta)^{s/2}u\|_{L^2}$, times a constant factors that is independent of the mollifying parameter $\delta$. We recall that integrating over the unit sphere translates into integration on $\RR^{d-1}$ for the measure $\lvr^{-2(d-1)}dv$. Our main tool throughout the next computations is Young's inequality.

$\bullet$ We start with $I_6$, whose upper bounds is obtained by a simple application of Young's inequality:
\beqq
\|I_6\|_{L^2(\RR^{d}\times\SS^{d-1})}\lesssim \|\lvr^{-(d-1)}\chi \cdot\xi^\delta\otimes\eta^\delta*\sigma u\|_{L^2(\RR^{d}\times\RR^{d-1})}\lesssim \|\xi^\delta\otimes\eta^\delta\|_{L^1(\RR^{d}\times\RR^{d-1})}\|u\|_{L^2(\RR^{d}\times\SS^{d-1})}.
\eeqq

$\bullet$ Let's recall the equivalence of norms in \eqref{equiv_norms}. For $I_5$, we have that
\beqq
\|I_5\|_{L^2(\RR^{d}\times\SS^{d-1})} \lesssim C\|(-\Delta_\theta)^{s/2}u\|_{L^2(\RR^{d}\times\SS^{d-1})}, 
\eeqq
with a constant depending on
\beq\label{Young_1}
\begin{aligned}
&\sup_{v''\in \sup(\chi)}\left\|\langle v''-\cdot\rangle^{d-1-2s}\int \frac{\eta^\delta(v')\zeta_2(v''-v',v'')-\eta^\delta(\cdot)\zeta_2(v''-\cdot,v'')}{|\cdot-v'|^{d-1+s}}dv'\right\|_{L^1(\RR^{d-1})};\\
&\sup_{v\in \RR^{d-1}}\left\|\lvr^{d-1-2s}\int \frac{\eta^\delta(\cdot-v')\zeta_2(v',\cdot)-\eta^\delta(\cdot-v)\zeta_2(v,\cdot)}{|v-v'|^{d-1+s}}dv'\right\|_{L^1(\supp(\chi))}.
\end{aligned}
\eeq
Let us verify that the previous quantities are finite and remain bounded as $\delta\to 0$. After  rescaling the integration variables, the first integral takes the form
\beqq
\int \langle v''-\delta v\rangle^{d-1-2s}\left|\int \frac{\zeta_3(v+z,v'')-\zeta_3(v,v'')}{\delta^s|z|^{d-1+s}}dz\right|dv,
\eeqq
where $\zeta_3(v,v'')=\eta(v)\zeta_2(v''-\delta v,v'')=\eta(v)(\langle v''-\delta v\rangle^{d-1+2s} - \langle v''\rangle^{d-1+2s})$. For $|z|<1$,
\beqq
\begin{aligned}
&|\zeta_3(v+z,v'')-\zeta_3(v,v'')|\\
&\leq |z|\int^1_0|\nabla \zeta_3(v+tz,v'')|dt\\
&\leq \delta|z|\int^1_0\int^1_0|\nabla\eta(v+tz)||(v''-\tau\delta (v+tz))\cdot (v+tz)|\langle v''-\tau\delta (v+tz)\rangle^{d-3+2s}d\tau dt \\
&\quad + \delta |z|\int^1_0|\eta(v+tz)||v+tz|\langle v''-\delta (v+z)\rangle^{d-3+2s}dt.
\end{aligned}
\eeqq
Therefore, since $v''\in\supp(\chi)$ and $\eta$ is compactly supported, we have
\beqq
\begin{aligned}
&\int \langle v''-\delta v\rangle^{d-1-2s}\int_{|z|<1} \frac{\zeta_3(v+z,v'')-\zeta_3(v,v'')}{\delta^s|z|^{d-1+s}}dzdv\\
&\leq \delta^{1-s} \int\int_{|z|<1}\int^1_0\int^1_0 \frac{|\nabla\eta(v+tz)||v+tz|\langle v''-\tau\delta (v+tz)\rangle^{d-2+2s}\langle v''-\delta v\rangle^{d-1-2s}}{|z|^{d-1-(1-s)}}d\tau dtdzdv\\
&\quad + \delta^{1-s} \int\int_{|z|<1} \int^1_0\frac{|\eta(v+tz)||v+tz|\langle v''-\delta (v+z)\rangle^{d-3+2s}\langle v''-\delta v\rangle^{d-1-2s}}{|z|^{d-1-(1-s)}}dt dzdv\\
&\lesssim \delta^{1-s}\int_{|z|<1} \frac{1}{|z|^{d-1-(1-s)}}dz.
\end{aligned}
\eeqq
For $|z|\geq 1$ we instead have that
\beqq
\begin{aligned}
&|\zeta_3(v+z,v'')-\zeta_3(v,v'')|\\
&\leq \delta\eta(v+z)\int^1_0|(v''-t\delta(v+z))\cdot(v+z)|\langle v''-t\delta(v+z)\rangle^{d-3+2s}dt\\
&\quad + \delta\eta(v)\int^1_0|(v''-t\delta v)\cdot v|\langle v''-t\delta v\rangle^{d-3+2s}dt.
\end{aligned}
\eeqq
Thus, using again the compactness of the support of $\eta$, for all $v''\in\supp(\chi)$,
\beqq
\begin{aligned}
&\int \langle v''-\delta v\rangle^{d-1-2s}\int_{|z|\geq1} \frac{\zeta_3(v+z,v'')-\zeta_3(v,v'')}{\delta^s|z|^{d-1+s}}dzdv\\
&\leq \delta^{1-s}\int\int_{|z|\geq 1}\int^1_0 \frac{\eta(v+z)|(v''-t\delta(v+z))\cdot(v+z)|\langle v''-t\delta(v+z)\rangle^{d-3+2s}\langle v''-\delta v\rangle^{d-1-2s}}{|z|^{d-1+s}}dtdzdv\\
&\quad + \delta^{1-s}\int\int_{|z|\geq 1}\int^1_0\frac{\eta(v)|(v''-t\delta v)\cdot v|\langle v''-t\delta v\rangle^{d-3+2s}\langle v''-\delta v\rangle^{d-1-2s}}{|z|^{d-1+s}}dtdzdv\\
&\lesssim \delta^{1-s}\int_{|z|\geq1} \frac{1}{|z|^{d-1+s}}dz.
\end{aligned}
\eeqq
Regarding the second quantity in \eqref{Young_1} we see that after a rescaling of all the variables $v,v'$ and $v''$ we get
\beqq
\int_{\delta\supp(\chi)}\langle \delta v\rangle^{d-1-2s}\left|\int \frac{\eta(v''-v')\zeta_2(\delta v',\delta v'')-\eta(v''-v)\zeta_2(\delta v,\delta v'')}{\delta^s|v-v'|^{d-1+s}}dv'\right|dv''.
\eeqq
We enlarge the domain of integration for $v''$ to a fixed ball $B_r$ containing $\supp(\chi)$. Since $\eta$ has compact support, we can find $R>0$ large enough so that $\eta(v''-v')=\eta(v''-v) = 0$ whenever $|v'|,|v|>R$.

Following similar computations as above we deduce the following upper bound for the region $|v|,|v'|<R$:
\beqq
\begin{aligned}
&\int_{B_r}\langle \delta v\rangle^{d-1-2s}\left|\int_{|v'|<R} \frac{\eta(v''-v')\zeta_2(\delta v',\delta v'')-\eta(v''-v)\zeta_2(\delta v,\delta v'')}{\delta^s|v-v'|^{d-1+s}}dv'\right|dv''\\
&\lesssim \delta^{1-s}\int_{|z|<2R} \frac{1}{|z|^{d-1-(1-s)}}dz.
\end{aligned}
\eeqq
In the case of $|v|<R$ and $|v'|>R$, we obtain an upper bound, up to a constant factor, given by
\beqq
\begin{aligned}
&\int_{B_r}\left|\int_{R/2<|v'|<2R} \frac{\eta(v''-v')\zeta_2(\delta v',\delta v'')-\eta(v''-v)\zeta_2(\delta v,\delta v'')}{\delta^s|v-v'|^{d-1+s}}dv'\right|dv''\\
&+\int_{B_r}\left|\int_{|v'|>2R} \frac{\zeta_2(\delta v,\delta v'')}{\delta^s|v-v'|^{d-1+s}}dv'\right|dv''\\
&\lesssim \delta^{1-s}\left(\int_{|z|<R}\frac{1}{|z|^{d-1-(1-s)}}dz+\int_{|z|>2R}\frac{1}{|z|^{d-1+s}}dz\right),
\end{aligned}
\eeqq
and similarly, for $|v|>R/2$ and $|v'|<R$ we get the upper bounds
\beqq
\begin{aligned}
&\int_{B_r}\left|\int_{|v'|<R} \frac{\eta(v''-v')\zeta_2(\delta v',\delta v'')-\eta(v''-v)\zeta_2(\delta v,\delta v'')}{\delta^s|v-v'|^{d-1+s}}dv'\right|dv''\\
&\lesssim \delta^{1-s}\int_{|z|<3R}\frac{1}{|z|^{d-1-(1-s)}}dz,
\end{aligned}
\eeqq
if $R/2<|v|<2R$, and for $|v|>2R$,
\beqq
\begin{aligned}
&\int_{B_r}\langle \delta v\rangle^{d-1-2s}\left|\int_{|v'|<R} \frac{\eta(v''-v')\zeta_2(\delta v',\delta v'')}{\delta^s|v-v'|^{d-1+s}}dv'\right|dv''\\
&\lesssim \delta^{1-s}\frac{\langle \delta v\rangle^{d-1-2s}}{(|v|-|R|)^{d-1+s}} \;\lesssim\;\delta^{1-s}.
\end{aligned}
\eeqq

$\bullet$ We analyze now the $L^2$-norm of $I_4$. Again by Young's inequality we obtain the estimate
\beqq
\|I_4\|_{L^2(\RR^{d}\times\SS^{d-1})} \lesssim C\|(-\Delta_\theta)^{s/2}u\|_{L^2(\RR^{d}\times\SS^{d-1})},
\eeqq
with a constant depending on
\beq\label{Young_2}
\begin{aligned}
&\sup_{(y,v'')\in \supp(\chi)}\left\|(\sigma(\cdot)-\sigma(y))\xi^\delta(y-\cdot)((-\Delta_v)^{s/2}\eta^\delta)(v''-\cdot)\right\|_{L^1(\RR^d\times\RR^{d-1})};\\
&\sup_{(x,v)\in \RR^d\times \RR^{d-1}}\left\|(\sigma(x)-\sigma(\cdot))\xi^\delta(\cdot-x)((-\Delta_v)^{s/2}\eta^\delta)(\cdot-v)\right\|_{L^1(\supp(\chi))}.
\end{aligned}
\eeq
Both terms above are $O(\delta^{1-s})$, since for all $x,y\in \RR^{d}$,
\beqq
\|(\sigma(\cdot)-\sigma(y))\xi^\delta(y-\cdot)\|_{L^1(\RR^d)}, \|(\sigma(x)-\sigma(\cdot))\xi^\delta(\cdot-x)\|_{L^1(\RR^d)}\lesssim \delta
\eeqq
and for all $v\in\RR^{d-1}$ and $v''\in\supp(\chi)$,
\beqq
\|((-\Delta_v)^{s/2}\eta^\delta)(v''-\cdot)\|_{L^1(\RR^{d-1})}, \|((-\Delta_v)^{s/2}\eta^\delta)(\cdot-v)\|_{L^1(\supp(\chi))}\lesssim \delta^{-s}.
\eeqq

$\bullet$ We write $I_3$ as follows
\beqq
\begin{aligned}
I_3 &= \sigma(y)\langle v''\rangle^{d-1+2s}\chi(y,v'')\int \xi^\delta(y-x) \frac{u(x,v')}{\lvpr^{d-1}}\\
&\quad \times \int\frac{\eta^\delta(v-v')\left(\lvpr^{2s}-\frac{\lvpr^{d-1}}{\lvr^{d-1-2s}}\right)-\eta^\delta(v''-v')\left(\lvpr^{2s}-\frac{\lvpr^{d-1}}{\langle v''\rangle^{d-1-2s}}\right)}{|v-v''|^{d-1+2s}}dvdv'dx.
\end{aligned}
\eeqq
Thus, we obtain the estimate
\beqq
\|I_3\|_{L^2(\RR^{d}\times\SS^{d-1})}\leq C\|u\|_{L^2(\RR^{d}\times\SS^{d-1})},
\eeqq 
for a constant depending on
\beq\label{Young_3}
\begin{aligned}
&\sup_{v''\in \supp(\chi)}\left\| \int\frac{\eta^\delta(v-\cdot)\left(\langle \cdot\rangle^{2s}-\frac{\langle \cdot \rangle^{d-1}}{\lvr^{d-1-2s}}\right)-\eta^\delta(v''-\cdot)\left(\langle \cdot\rangle^{2s}-\frac{\langle \cdot\rangle^{d-1}}{\langle v''\rangle^{d-1-2s}}\right)}{|v-v''|^{d-1+2s}}dv\right\|_{L^1(\RR^{d-1})};\\
&\sup_{v'\in  \RR^{d-1}}\left\| \int\frac{\eta^\delta(v-v')\left(\lvpr^{2s}-\frac{\lvpr^{d-1}}{\lvr^{d-1-2s}}\right)-\eta^\delta(v''-v')\left(\lvpr^{2s}-\frac{\lvpr^{d-1}}{\langle \cdot\rangle^{d-1-2s}}\right)}{|v-\cdot|^{d-1+2s}}dv \right\|_{L^1(\supp(\chi))}.
\end{aligned}
\eeq
By rescaling $v,v'$ and $v''$ and after a change of variables the integrand takes the form
\beqq
\begin{aligned}
&\int\frac{\eta^\delta(v''+z-v')\left(\lvpr^{2s}-\frac{\lvpr^{d-1}}{\langle v''+z\rangle^{d-1-2s}}\right)-\eta^\delta(v''-v')\left(\lvpr^{2s}-\frac{\lvpr^{d-1}}{\langle v''\rangle^{d-1-2s}}\right)}{|z|^{d-1+2s}}dz\\
& = \frac{1}{2}\int\frac{\zeta_4(v''+z,v')+\zeta_4(v''-z,v')-2\zeta_4(v'',v')}{\delta^{2s}|z|^{d-1+2s}}dz
\end{aligned}
\eeqq
for $\zeta_4(v,v') = \eta(v-v')\left(\langle \delta v'\rangle^{2s}-\frac{\langle \delta v'\rangle^{d-1}}{\langle \delta v\rangle^{d-1-2s}}\right)$. We also notice that for $v''\in B_r$ (see definition above), since $\eta$ has compact support, the integrand is compactly supported as a function of $z$ and $v'$.  Therefore, we set $R>0$ so that such support is contained $|z|,|v'|<R$. We have that
\beqq
\begin{aligned}
\left|\zeta_4(v''+z,v')+\zeta_4(v''-z,v')-2\zeta_4(v'',v')\right| \leq \|\nabla^2 \zeta_4\|_\infty|z|^2
\end{aligned}
\eeqq
and the desired estimates follow after showing that $\|\nabla^2 \zeta_4\|_\infty\lesssim \delta^2$. With this in mind, we first observe that
\beqq
\left|\langle \delta v'\rangle^{2s}-\frac{\langle \delta v'\rangle^{d-1}}{\langle \delta v\rangle^{d-1-2s}}\right|\lesssim \delta^2\int^1_0\frac{|v+t(v'-v)|}{\langle \delta(v+t(v'-v))\rangle^{d+1-2s}}|v-v'|dt\lesssim \delta^2,
\eeqq
in the support of the integrand. We also have that on the same support,
\beqq
\left|\nabla_v\left(\langle \delta v'\rangle^{2s}-\frac{\langle \delta v'\rangle^{d-1}}{\langle \delta v\rangle^{d-1-2s}}\right)\right|,\left|\nabla^2_v\left(\langle \delta v'\rangle^{2s}-\frac{\langle \delta v'\rangle^{d-1}}{\langle \delta v\rangle^{d-1-2s}}\right)\right|\lesssim \delta^2.
\eeqq

$\bullet$ For $I_2$ we see that for $|v''|>R$ with $R>0$ large enough so the support of $\chi$ is contained in the ball $B_{R/2}\times B_{R/2}$, then
\beqq
\begin{aligned}
|I_2| &\lesssim \langle v''\rangle^{d-1+2s}\int \frac{\left(\frac{(\xi^\delta\otimes\eta^\delta*u
)(y,v)}{\langle v\rangle^{d-1-2s}}-\frac{(\xi^\delta\otimes\eta^\delta*u
)(y,v'')}{\langle v''\rangle^{d-1-2s}}\right)\chi(y,v)}{(|v''|-R/2)^{d-1+2s}}dv\\
&\lesssim \int_{B_{R/2}}\left(\frac{(\xi^\delta\otimes\eta^\delta*u
)(y,v)}{\langle v\rangle^{d-1-2s}}-\frac{(\xi^\delta\otimes\eta^\delta*u
)(y,v'')}{\langle v''\rangle^{d-1-2s}}\right)dv\\
&\lesssim \int_{B_{R/2}}|(\xi^\delta\otimes\eta^\delta*u
)(y,v)|dv + |(\xi^\delta\otimes\eta^\delta*u
)(y,v'')|,
\end{aligned}
\eeqq
and consequently,
\beqq
\begin{aligned}
&\int_{\RR^d\times B_R^c} |I_2(y,v'')|^2\frac{dydv''}{\langle v''\rangle^{2(d-1)}}\\
&\lesssim \int_{B_{R/2}\times B_R^c} \left|\int_{B_{R/2}}|(\xi^\delta\otimes\eta^\delta*u
)(y,v)|dv + |(\xi^\delta\otimes\eta^\delta*u
)(y,v'')|\right|^2\frac{dydv''}{\langle v''\rangle^{2(d-1)}}
\lesssim \|u\|_{L^2(\RR^d\times\SS^{d-1})}.
\end{aligned}
\eeqq
On the other hand, noticing that
\beqq
\begin{aligned}
\frac{(\eta^\delta\otimes\eta^\delta*u)}{\lvr^{d-1-2s}}
&= \int \eta^\delta(v')w(v-v')dv'
+ \int \eta^\delta(v-v')u(v')\left(\frac{1}{\langle v\rangle^{d-1-2s}} - \frac{1}{\lvpr^{d-1-2s}}\right)dv',
\end{aligned}
\eeqq
and denoting $\zeta_5(v,v')=\eta^\delta(v-v')\left(\frac{1}{\langle v\rangle^{d-1-2s}} - \frac{1}{\lvpr^{d-1-2s}}\right)$, we have that
\beqq
\begin{aligned}
&\int_{\RR^d\times B_{R}} |I_2(y,v'')|^2\frac{dydv''}{\langle v''\rangle^{2(d-1)}}\\
&\lesssim \int_{B_{R/2}\times B_{R}}\langle v''\rangle^{4s}\left|\int \frac{\left(\frac{(\xi^\delta\otimes\eta^\delta*u
)( y, v)}{\langle v\rangle^{d-1-2s}}-\frac{(\xi^\delta\otimes\eta^\delta*u
)( y,v'')}{\langle  v''\rangle^{d-1-2s}}\right)(\chi( y,v) - \chi( y,v''))}{|v-v''|^{d-1+2s}}dv\right|^2dydv''\\
&\lesssim \int_{B_{R/2}\times B_{R}}\langle v''\rangle^{4s}\left|\int \frac{\left((\xi^\delta\otimes\eta^\delta*w
)( y, v)-(\xi^\delta\otimes\eta^\delta*w
)( y,v'')\right)(\chi( v) - \chi( v''))}{|v-v''|^{d-1+2s}}dv\right|^2dydv''\\
&\quad + \int_{B_{R/2}\times B_{R}}\langle v''\rangle^{4s}\left|\int (\chi( y,v) - \chi(y, v''))\right.\\
&\hspace{3em} \times\left.\frac{ \left(\int \xi^\delta(y-x)\zeta_5(v,v')u(x,v')dxdv' - \int \xi^\delta(y-x)\zeta_5(v'',v')u(x,v')dxdv' \right)}{|v-v''|^{d-1+2s}}dv\right|^2dydv''\\
&\lesssim \|(-\Delta_v)^{s/2}(\xi^\delta\otimes\eta^\delta*w
)\|_{L^2}^2\\
&\quad + \int_{B_{R/2}\times B_{R}}\left|\int\xi^\delta(y-x)\left( \int\int \frac{ \left|\zeta_5(v,v')- \zeta_5(v'',v')\right|}{|v-v''|^{d-1+2s-1}}|u(x,v')|dv' dv\right)dx\right|^2dydv''\\
&\lesssim \|\xi^\delta\otimes\eta^\delta\|_{L^1}^2\|(-\Delta_v)^{s/2}w
\|_{L^2}^2\\
&\quad + \int_{B_{R/2}\times B_{R}}\left|\int\int \frac{ \left|\zeta_5(v,v')- \zeta_5(v'',v')\right|}{|v-v''|^{d-1+2s-1}}|u(y,v')|dv' dv\right|^2dydv''\\
&\lesssim \|(-\Delta_v)^{s/2}w\|_{L^2}^2\\
&\quad + \int_{B_{R/2}\times B_{R}}\left|\int\left(\int \frac{ \left|\zeta_5(v,v')- \zeta_5(v'',v')\right|}{|v-v''|^{d-1+2s-1}}\lvpr^{d-1}dv\right) \frac{|u(y,v')|}{\lvpr^{d-1}}dv' \right|^2dydv''\\
&\lesssim \|(-\Delta_\theta)^{s/2}u\|_{L^2}^2  + C\|u\|_{L^2(\RR^{d}\times\SS^{d-1})},
\end{aligned}
\eeqq
with $C$ depending on both:
\beqq
\begin{aligned}
&\sup_{v''\in B_R}\int \int \frac{ \left|\zeta_5(v,v')- \zeta_5(v'',v')\right|}{|v-v''|^{d-1+2s-1}}\lvpr^{d-1}dv dv';\\
&\sup_{v'\in \RR^{d-1}}\int_{B_R} \int \frac{ \left|\zeta_5(v,v')- \zeta_5(v'',v')\right|}{|v-v''|^{d-1+2s-1}}\lvpr^{d-1}dv dv''.
\end{aligned}
\eeqq
We need to make sure that these quantities do not diverge as $\delta\to 0$. Indeed, rescaling all three variables $v,v'$ and $v''$, we get
\beqq
\begin{aligned}
&\int \int \frac{ \left|\zeta_5(v,v')- \zeta_5(v'',v')\right|}{|v-v''|^{d-1+2s-1}}\lvpr^{d-1}dv dv'\\
&=\int \int \frac{ \left|\eta(v-v')\left(\frac{1}{\langle \delta v\rangle^{d-1-2s}} - \frac{1}{\langle \delta v'\rangle^{d-1-2s}}\right)- \eta(v''-v')\left(\frac{1}{\langle \delta v''\rangle^{d-1-2s}} - \frac{1}{\langle \delta v'\rangle^{d-1-2s}}\right)\right|}{\delta^{2s-1}|v-v''|^{d-1+2s-1}}\langle \delta v'\rangle^{d-1}dv dv'.
\end{aligned}
\eeqq
The numerator above is bounded by
\beqq
\begin{aligned}
&|v-v''|\int^1_0|\nabla\eta(v''+t(v-v'')-v')|\left|\frac{1}{\langle \delta (v''+t(v-v''))\rangle^{d-1-2s}} - \frac{1}{\langle \delta v'\rangle^{d-1-2s}}\right|\langle \delta v'\rangle^{d-1}dt\\
&+\delta|v-v''|\int^1_0|\eta(v''+t(v-v'')-v')|\frac{\delta\left|v''+t(v-v'')\right|}{\langle \delta(v''+t(v-v''))\rangle^{d+1-2s} }\langle \delta v'\rangle^{d-1}dt\\
&\lesssim \delta|v-v''|\int^1_0 (|\nabla\eta(v''+t(v-v'')-v')| + |\eta(v''+t(v-v'')-v')| )dt,
\end{aligned} 
\eeqq
which implies that the quantities above are $O(\delta^{2(1-s)})$.

We conclude this section with the estimation of $\|I_1\|_{L^2}$. We see that
\beqq
\begin{aligned}
\|I_1\|_{L^2}^2&\lesssim \int\left(\frac{(\xi^\delta\otimes\eta^\delta*u
)(y,v'')}{\langle v''\rangle^{d-1}}\frac{((-\Delta_v)^{s}\chi)(y,v'')}{\langle v''\rangle^{d-1-4s}}\right)^2dydv''\\
&\lesssim\left\|\frac{(\xi^\delta\otimes\eta^\delta*u
)(y,v'')}{\langle v''\rangle^{d-1}}\right\|_{L^2}^2\left\|\frac{((-\Delta_v)^{s}\chi)(y,v'')}{\langle v''\rangle^{d-1-4s}}\right\|_{\infty}^2.
\end{aligned}
\eeqq
The first term on the left is estimated by noticing that
\beqq
\begin{aligned}
\left\|\frac{(\xi^\delta\otimes\eta^\delta*u
)(y,v'')}{\langle v''\rangle^{d-1}}\right\|_{L^2}^2&= \left\|\int \frac{\xi^\delta(y-x)\eta^\delta(v''-v)\lvr^{d-1}}{\langle v''\rangle^{d-1}}\frac{u(x,v)}{\lvr^{d-1}} dxdv\right\|_{L^2}^2 \leq C \|u\|_{L^2(\RR^d\times\SS^{d-1})}^2,
\end{aligned}
\eeqq
for a constant depending on
\beqq
\sup_{v''\in\RR^{d-1}}\int \frac{\eta^\delta(v''-v)\lvr^{d-1}}{\langle v''\rangle^{d-1}}dv
=\sup_{v''\in\RR^{d-1}}\int \frac{\eta(v''-v)\langle \delta v\rangle^{d-1}}{\langle \delta v''\rangle^{d-1}}dv \lesssim \|\eta\|_{L^1};
\eeqq
and similarly
\beqq
\sup_{v\in\RR^{d-1}}\int \frac{\eta^\delta(v''-v)\lvr^{d-1}}{\langle v''\rangle^{d-1}}dv''
\lesssim \|\eta\|_{L^1}.
\eeqq
On the other hand, 
\beqq
\begin{aligned}
\left|\frac{((-\Delta_v)^{s}\chi)(y,v'')}{\langle v''\rangle^{d-1-4s}}\right|& = \frac{1}{\langle v''\rangle^{d-1-4s}} \int \frac{|\chi(y,v''+z) + \chi(y,v''-z) - 2\chi(y,v'')|}{|z|^{d-1+2s}}dz\\
&\leq \left(\sup_{v'',z\in\RR^{d-1}}\frac{|\nabla^2_z\chi(y,v''+z)|}{\langle v''\rangle^{d-1-4s}}\right)\int\frac{1}{|z|^{d-1-2(1-s)}}dz \\
&\quad + \left(\sup_{v'',z\in\RR^{d-1}}\frac{|\chi(y,v''+z)|}{\langle v''\rangle^{d-1-4s}}\right)\int\frac{1}{|z|^{d-1+2s}}dz.
\end{aligned}
\eeqq
Therefore, $\left\|\frac{((-\Delta_v)^{s}\chi)(\cdot)}{\langle \cdot\rangle^{d-1-4s}}\right\|_{\infty}<\infty$.

\section{Estimates for $h(v,v')$}\label{appdx:func_h}
We recast $h(v,v')$ as
\beq\label{ident_h}
\begin{aligned}
h(v,v')=\lvr\lvpr - v\cdot v' -1 &= \lvr\lvpr - \frac{1}{2}(|v|^2 + |v'|^2) + \frac{1}{2}|v-v'|^2 -1\\
&=\frac{1}{2}|v-v'|^2 - \frac{1}{2}\left(\lvr - \lvpr\right)^2.
\end{aligned}
\eeq
This directly gives
\beq\label{ineq1_h}
h(v,v')\leq \frac{1}{2}|v-v'|^2.
\eeq
On the other hand, denoting $p(v') = \lvpr $, one realizes that it satisfies
\beqq
|\nabla p(v')| \leq \frac{|v'|}{\lvpr}=\sqrt{1-\frac{1}{\lvpr^2}}.
\eeqq
Therefore,
\beqq
|p(v') - p(v)| \leq \left(\int^1_0\sqrt{1-\frac{1}{\langle v+t(v'-v)\rangle^2}}dt\right)|v'-v|.
\eeqq
If $\rho =|v-v'| >|v| $ we see that
\beqq
\begin{aligned}
&\int^1_0\sqrt{1-\frac{1}{\langle v+t(v'-v)\rangle^2}}dt \leq \int^1_0\sqrt{1-\frac{1}{1+2(1+t^2)\rho^2}}dt\\
\leq& \frac{\sqrt{2}\rho}{\sqrt{1+2\rho^2}}\int^1_0\sqrt{1+t^2}dt
<\left.\left(\frac{t\sqrt{1+t^2}}{2} + \frac{1}{2}\ln\left(t+\sqrt{1+t^2}\right)\right)\right|_{0}^1
=\frac{1}{\sqrt{2}}\left(1-\frac{1}{\sqrt{2}}\ln(1+\sqrt{2})\right),
\end{aligned}
\eeqq
which is strictly positive (this follows from the inequality $\ln(1+x)<x$ for $x>0$) and smaller than 1. This implies
\beqq
|p(v') - p(v)| \leq \frac{1}{\sqrt{2}}\left(1-\frac{1}{\sqrt{2}}\ln(1+\sqrt{2})\right)|v'-v|.
\eeqq
Similarly, if now we set $\rho = |v|>|v'-v|$, the same computations leads us to conclude that there is $c\in(0,1)$ such that
\beqq
|p(v') - p(v)| \leq c|v-v'|,\quad\forall v,v'\in\RR^{d-1}.
\eeqq
Then, from the above, \eqref{ident_h}, and \eqref{ineq1_h}, we conclude that
\beq\label{final_ineq_h}
\frac{\beta}{2}|v-v'|^2\leq h(v,v')\leq \frac{1}{2}|v-v'|^2,
\eeq
for $\beta=\left(1-\frac{1}{2}\left(1-\frac{1}{\sqrt{2}}\ln(1+\sqrt{2})\right)^2\right)\in(0,1)$.

In addition, $h(v,v+z)$ satisfies the following inequality for $|v|<R$ and $|z|<1$:
\beq\label{ineq_h}
\begin{aligned}
&\left|\frac{1}{h(v,v+z)^{\alpha}} - \frac{1}{h(v,v-z)^{\alpha}}\right|\\
&\lesssim
\frac{\left|h(v,v+z)^{\alpha-1}+ h(v,v-z)^{\alpha}\right|}{h(v,v+z)^{\alpha}h(v,v-z)^{\alpha-1}}\left| \left(\lvr - \langle v+z\rangle\right)^2-\left(\lvr - \langle v-z\rangle\right)^2\right|\\
&\lesssim \frac{1}{|z|^{2(\alpha+1)}}\left| \left(\lvr - \langle v+z\rangle\right)^2-\left(\lvr - \langle v-z\rangle\right)^2\right|\\
&\lesssim \frac{1}{|z|^{d-1+2s+2}}\left|\langle v+z\rangle- \langle v-z\rangle\right|\left| 2\lvr - \langle v+z\rangle - \langle v-z\rangle \right|
\ \lesssim \frac{1}{|z|^{d-1+2s-1}},
\end{aligned}
\eeq
for a constant depending on $R$.

\section{Subelliptic estimates}\label{appdx:subellipticity}
In this appendix we use results from \cite{A} to obtain subelliptic regularity for a kinetic equation with singular scattering kernel of fractional order. We consider
\beq\label{kinetic_eq}
\partial_tf+v\cdot\nabla_xf = \int_{\RR^{d}}k(t,x,v,z)(f(v+z) - f(v))dz + g,\quad t\in\RR,\;x\in\RR^d, v\in\RR^d,
\eeq
for a kernel given by
\beq\label{kinetic_kernel}
k(t,x,v,z)=a(t,x)\frac{\lvr^s\langle v+z\rangle^s}{h(v,v+z)^{\frac{d-1}{2}+s}},\quad \text{with}\quad h(v,v+z) = \frac{1}{2}|z|^2-\frac{1}{2}(\lvr - \langle v+z\rangle)^2,
\eeq
and with $a(t,x)$ a continuous and positive function, bounded from below by a positive constant. 
We prove the following:
\begin{theorem}\label{thm:ext_Alexandre}
Let $f$ be a smooth solution to \eqref{kinetic_eq}-\eqref{kinetic_kernel} with compact support contained in a bounded set $\Omega\subset\RR\times\RR^d\times B_\rho$ for some small $\rho>0$. Then,
\beqq
\|(-\Delta_v)^{s}f\|_{L^2} + \|(-\Delta_x)^{\frac{s}{1+2s}}f\|_{L^2}\leq C_{k,\Omega}\left(\|g\|_{L^2}+\|f\|_{L^2}+ \|(-\Delta_v)^{s/2}f\|_{L^2}\right).
\eeqq

\begin{proof} We rewrite equation \eqref{kinetic_eq} in the following form
\beq\label{kinetic_eq2}
\partial_t+v\cdot\nabla_xf + \tilde{a}(t,x,v)(-\Delta_v)^sf= \tilde{g},
\eeq
with $\tilde{a}(t,x,v) = \left(\frac{2}{\beta}\right)^{\frac{d-1}{2}+s}\lvr^{2s}a(t,x)$ and, denoting $\alpha = \frac{d-1}{2}+s$, a source term given by
\beqq
\begin{aligned}
\tilde{g}&= g +\lvr^{s}a(t,x)\int \frac{(f(v+z) - f(v))(\langle v+z\rangle^s - \lvr^s)}{h(v,v+z)^{\alpha}}dz \\
&\quad+ \lvr^{2s}a(t,x)\int (f(v+z) - f(v))\left(\frac{1}{h(v,v+z)^{\alpha}} -  \frac{1}{(\frac{\beta}{2}|z|^2)^{\alpha}}\right)dz.
\end{aligned}
\eeqq
Furthermore, since $f(t,x,v)$ is compactly supported, we can assume without loss of generality that $\tilde{a}$ is of the form $b^2(t,x,v)\chi(t,x,v)+a_-$, for $b^2(t,x,v)=(\tilde{a}(t,x,v)-\tilde{a}_-)$, $\chi$ smooth and compactly supported, and $\tilde{a}_-=\inf_{t,x,v}\tilde{a}(t,x,v)>0$ constant. This can be seen by following an application of Jensen's inequality when integrating over the region $|v|>R\gg \rho$. In fact,
\beqq
\begin{aligned}
\int_{|v|>R}|\lvr^{2s}(-\Delta_v)^sf |^2dv &= \int_{|v|>R}\lvr^{4s}\left(\int_{|z|<\rho}\frac{|f(z)|}{|v-z|^{d-1+2s}}dz\right)^2dv\\
&\lesssim \int_{|z|<\rho}|f(z)|^2\left(\int_{|v|>R}\frac{\lvr^{4s}}{(|v|-\rho)^{2(d-1)+4s}}dv\right)dz
\ \lesssim \|f\|^2_{L^2}.
\end{aligned}
\eeqq
We then take a cut-off function $\chi$, with compact support satisfying $\Omega\subset\supp(\chi)\subset\RR\times\RR^{d-1}\times B_{R}$, and write
\beqq
\begin{aligned}
\tilde{a}(t,x,v)(-\Delta_v)^sf &= (\tilde{a}(t,x,v)-\tilde{a}_-)(-\Delta_v)^sf +\tilde{a}_-(-\Delta_v)^sf \\
&= (b^2(t,x,v)\chi(t,x,v)+\tilde{a}_-)\int\frac{(f(v)-f(v+z))}{|z|^{d-1+2s}}dz\\
&\quad + b^2(t,x,v)(1- \chi(t,x,v))\int\frac{(f(v)-f(v+z))}{|z|^{d-1+2s}}dz.
\end{aligned}
\eeqq
Therefore, $f$ satisfies the equation
\beqq
\partial_t+v\cdot\nabla_xf + (b^2\chi +\tilde{a}_-)(-\Delta_v)^sf= \tilde{G},
\eeqq
for a source $\tilde{G}$ given by
\beqq
G = \tilde{g} + b^2(t,x,v)(1- \chi(t,x,v))\int\frac{(f(v)-f(v+z))}{|z|^{d-1+2s}}dz,
\eeqq
which we claim it satisfies
\beqq
\|G\|_{L^2}\lesssim \|g\|_{L^2}+\|f\|_{L^2}+\|(-\Delta_v)^{s/2}f\|_{L^2}+\rho \|(-\Delta_v)^{s}f\|_{L^2}.
\eeqq
For this, it only remains to estimate the norm of the two integral terms in the definition of $\tilde{g}$. Let us write $\tilde{g} = g_0 + g_1+g_2$ with $g_0 = g$. 
For the second summand, we see that in the region $|v|>R\gg \rho$, we repeat some computations above in order to get
\beqq
\begin{aligned}
\int_{|v|>R}|g_1(v)|^2dv &=\int_{|v|>R}\left(\lvr^sa(t,x)\int\frac{|f(v+z)|(\langle v+z\rangle^s - \lvr^s)}{|z|^{d-1+2s}}dz\right)^2dv\\
&\lesssim \int_{|z|<\rho}|f(z)|^2\left(\int_{|v|>R}\frac{\lvr^{4s}}{(|v|-\rho)^{2(d-1)+4s}}dv\right)dz
\ \lesssim \ \|f\|^2_{L^2}.
\end{aligned}
\eeqq
For $|v|<R$, the estimate follows by noticing that $g_1(v)$ satisfies
\beqq
\begin{aligned}
&\left| g_1(v)\right| \lesssim \int_{|z|<1}\frac{|f(v+z)-f(v)|}{|z|^{d-1+2s-1}}dz + \int_{|z|>1}\frac{|f(v+z)-f(v)|}{|z|^{d-1+2s}}dz.
\end{aligned}
\eeqq
Therefore, using H\"older's inequality, and Jensen's inequalities for the finite measures $|z|^{-(d-1+2s)}dz$ in $|z|>1$, we obtain
\beqq
\begin{aligned}
\int_{|v|<R}\left| g_1(v)\right|^2dv  &\lesssim \int_{|v|<R}\left(\int_{|z|<1}\frac{|f(v+z)-f(v)|}{|z|^{d-1+2s-1}}dz\right)^2dv + \int_{|v|<R}\left(\int_{|z|>1}\frac{|f(v+z)-f(v)|}{|z|^{d-1+2s}}dz\right)^2dv\\
&\lesssim \int_{|v|<R}\left(\int_{|z|<1}\frac{|f(v+z)-f(v)|^2}{|z|^{d-1+2s}}dz\right)\left(\int_{|z|<1}\frac{1}{|z|^{d-1-(1-s)}}\right)dv \\
&\quad + \int_{|v|<R}\int_{|z|>1}\frac{|f(v+z)-f(v)|^2}{|z|^{d-1+2s}}dzdv
\ \lesssim \ \|(-\Delta_v)^{s/2}f\|^2_{L^2}+\|f\|^2_{L^2}.
\end{aligned}
\eeqq
We now estimate $g_2(v)$, which we rewrite as
\beq\label{g2}
\begin{aligned}
g_2(v)=&\frac{\lvr^{2s}a(t,x)}{2}\int (f(v+z) + f(v-z)- 2f(v))\left(\frac{1}{h(v,v+z)^{\alpha}} - \frac{1}{(\frac{1}{2}|z|^2)^{\alpha}}\right)dz\\
&-\frac{\lvr^{2s}a(t,x)}{2}\int (f(v+z) - f(v))\left(\frac{1}{h(v,v+z)^{\alpha}} - \frac{1}{h(v,v-z)^{\alpha}}\right)dz.
\end{aligned}
\eeq
We then split the integration in two and write $g_2(v) = g_2^{>}(v)+g_2^{<}(v)$, with each summand associated to the respective region of integration $|z|>1$ and $|z|<1$. The integrals in $g_2^{>}(v)$ are easily bounded by $\|f\|_{L^2}$ and it thus remains to analyze $g_2^{<}(v)$.

We next decomposes $g_2^{<}(v) = g_{2,1}^{<}(v)+g_{2,2}^{<}(v)$ following \eqref{g2}. We notice that for $g_{2,2}^{<}(v)$ we can repeat the trick of restricting the estimate to $|v|<R$ for some $R\gg \rho$, since otherwise, the compact support of $f$ allows to estimate the $L^2$-norm of those integrals by $\|f\|_{L^2}$. Then, the remaining part (i.e. for $|v|<R$) is bounded by a constant factor times
\beqq
\begin{aligned}
&\int_{|z|<1} |f(v+z) - f(v)|\left|\frac{1}{h(v,v+z)^{\alpha}} - \frac{1}{h(v,v-z)^{\alpha}}\right|dz.
\end{aligned}
\eeqq
It follows directly from \eqref{ineq_h} that
\beqq
|g_{2,2}^{<}(v)|\lesssim \int_{|z|<1}\frac{|f(v+z)-f(v)|}{|z|^{d-1+2s-1}}dz,\quad \text{for all }|v|<R,
\eeqq
and consequently
\beqq
\|g_{2,2}^{<}(v)\|_{L^2}\lesssim \|(-\Delta_v)^{s/2}f\|_{L^2}+\|f\|_{L^2}.
\eeqq

Finally for $g_{2,1}^{<}(v)$, we take a smooth cutoff $\chi$ supported inside $\{|v|<2\rho\}$, and $\chi=1$ in $\{|v|\leq \rho\}$, and split the coefficient $\lvr^{2s}a(t,x) = a_1(t,x,v) + a_2(t,x,v)$, with $a_1(t,x,v)=\chi(v)\lvr^{2s}a(t,x)$. When estimating the norm of $g_{2,1}^{<}(v)$, we obtain an integral associated to each coefficient $a_1$ and $a_2$. For the latter, the associated integral is easily bounded by $C_\rho\|f\|_{L^2}$ as we have done before (now taking $R=2\rho$), for a constant that increases as $\rho \to 0$. For $a_1$ we have that
\beqq
\begin{aligned}
&\|\chi(v)g_{2,1}^{<}(v)\|_{L^2}^2 \\
&= \int\left(a_1(t,x,v)\int_{|z|<1} (f(v+z) + f(v-z)- 2f(v))\left(\frac{1}{h(v,v+z)^{\alpha}} - \frac{1}{(\frac{1}{2}|z|^2)^{\alpha}}\right)dz\right)^2dv\\
&=\int\left(\int_{|z|<1} \hat{f}(\eta)(e^{i\eta\cdot z} + e^{-i\eta\cdot z}- 2)\int e^{-i\eta\cdot v}a_1(t,x,v)\left(\frac{1}{h(v,v+z)^{\alpha}} - \frac{1}{(\frac{1}{2}|z|^2)^{\alpha}}\right)dvdz\right)^2d\eta.
\end{aligned}
\eeqq
We first notice that due to the compact support of $a_1(t,x,v)$ as a function of $v$, which is contained in the ball $|v|<2\rho$, we have that
\beqq
\left|\int e^{-i\eta\cdot v}a_1(t,x,v)\left(\frac{1}{h(v,v+z)^{\alpha}} - \frac{1}{(\frac{1}{2}|z|^2)^{\alpha}}\right)dv\right|\lesssim \frac{\rho^2}{|z|^{2\alpha}}.
\eeqq
This follows from the fact that the above difference inside the parentheses is bounded (up to a constant factor) by $|v|^2/|z|^{2\alpha}$.

We split the integration in $\|\chi(v)g_{2,2}^{<}(v)\|_{L^2}^2$ into several regions. For $|\eta|<1$, we simply estimate
\beqq
|e^{i\eta\cdot z} + e^{-i\eta\cdot z}- 2|\lesssim |\eta|^2|z|^2\leq |z|^2,
\eeqq
which allows us to bound the integral on this region by a constant times $\|f\|_{L^2}^2$. For $1<|\eta|<1/|z|$, we use the the inequality $|e^{i\eta\cdot z} + e^{-i\eta\cdot z}- 2|\lesssim |\eta|^2|z|^2$, and
\beqq
|\eta|^2\int_{|z|<1/|\eta|}\frac{1}{|z|^{2\alpha-2}}dz\lesssim |\eta|^{2s}.
\eeqq
Finally, for $|\eta|>1/|z|$ and $\delta\in(0,2s)$, we have that $|z|^{-2\alpha}\leq |\eta|^{2s-\delta}/|z|^{d-1+\delta}$ and
\beqq
|\eta|^{2s-\delta}\int_{1/|\eta|<|z|<1}\frac{1}{|z|^{d-1+\delta}}dz\lesssim |\eta|^{2s-\delta}\frac{1}{t^\delta}\Big|^{t=1}_{t=1/|\eta|}\lesssim |\eta|^{2s}.
\eeqq
With the aids of the previous inequalities and Jensen's inequality, we obtain
\beqq
\begin{aligned}
\|\chi(v)g_{2,1}^{<}(v)\|_{L^2}^2&\lesssim \rho^2\int |\hat{f}(\eta)|^2(1 + |\eta|^{2s})d\eta\lesssim\rho^2\left(\|f\|^2_{L^2}+\|(-\Delta_v)^{s}f\|^2_{L^2}\right).
\end{aligned}
\eeqq
In summary, we have deduced the estimate
\beqq
\|g_2(v)\|_{L^2}\lesssim C_\rho \|f\|_{L^2} + \|(-\Delta_v)^{s/2}f\|_{L^2}+\rho\|(-\Delta_v)^sf\|_{L^2}.
\eeqq

We then apply Theorem 1.2 from \cite{A} and obtain
\beqq
\|(-\Delta_v)^{s}f\|_{L^2} + \|(-\Delta_x)^{\frac{s}{1+2s}}f\|_{L^2}\leq C_k\left(\|\tilde{G}\|_{L^2}+\|f\|_{L^2}\right),
\eeqq
where if $f$ is assumed to have a sufficiently small support, then we can absorb the term $\rho \|(-\Delta_v)^{s}f\|_{L^2}$ in the upper bound with the left hand side and conclude that
\beqq
\|(-\Delta_v)^{s}f\|_{L^2} + \|(-\Delta_x)^{\frac{s}{1+2s}}f\|_{L^2}\leq C_{k,\rho}\left(\|g\|_{L^2}+\|f\|_{L^2}+\|(-\Delta_v)^{s/2}f\|_{L^2}\right).
\eeqq
\end{proof}

\begin{remark}\label{remark_D}
From the previous theorem and its proof we also obtain that 
\beqq
\left\|\int_{\RR^{d}}k(t,x,v,z)(f(v+z) - f(v))dz\right\|_{L^2}\lesssim \|f\|_{L^2}+ \|(-\Delta_v)^{s/2}f\|_{L^2}+\|(-\Delta_v)^{s}f\|_{L^2}.
\eeqq
\end{remark}
\end{theorem}

\end{appendix}


\end{document}